\documentclass[reqno]{amsart}
\usepackage{graphicx, amsmath,amsfonts,amsthm,amssymb, paralist, fullpage}
\usepackage{hyperref}

\newtheorem{theorem}{Theorem}[section]
\newtheorem{conjecture}[theorem]{Conjecture}
\newtheorem{corollary}[theorem]{Corollary}
\newtheorem{lemma}[theorem]{Lemma}

\theoremstyle{definition}
\newtheorem{definition}[theorem]{Definition}
\theoremstyle{remark}
\newtheorem{remark}[theorem]{Remark}
\numberwithin{equation}{section}

\newcommand{\g}{\geqslant}

\newcommand{\RR}{\mathbb{R}}
\newcommand{\ZZ}{\mathbb{Z}}
\newcommand{\CC}{\mathbb{C}}

\newcommand{\NN}{\mathbb{N}}
\newcommand{\p}{\partial}

\newcommand{\les}{\leqslant}
\newcommand{\lesa}{\lesssim}

\newcommand{\mc}[1]{\mathcal{#1}}
\newcommand{\mf}[1]{\mathfrak{#1}}

\newcommand{\lr}[1]{ \langle #1 \rangle}
\newcommand{\ind}{\mathbbold{1}}

\DeclareSymbolFont{bbold}{U}{bbold}{m}{n}
\DeclareSymbolFontAlphabet{\mathbbold}{bbold}

\DeclareMathOperator*{\supp}{supp}

\setdefaultenum{(i)}{(a)}{1}{A}

\begin{document}

\title{Minimal non-scattering solutions for the Zakharov System}%

\author[T.~Candy]{Timothy Candy}
\address[T.~Candy]{Department of Mathematics and Statistics, University of Otago, PO Box 56, Dunedin 9054, New Zealand}
\email{tcandy@maths.otago.ac.nz}

\begin{abstract}
We consider the Zakharov system in the energy critical dimension $d=4$ with energy below the ground state. It is known that below the ground state solutions exist globally in time, and scatter in the radial case. Scattering below the ground state in the non-radial case is an open question. We show that if scattering fails, then there exists a minimal energy non-scattering solution below the ground state. Moreover the orbit of this solution is precompact modulo translations. The proof follows by a concentration compactness argument, together with a refined small data theory for energy dispersed solutions.
\end{abstract}



\maketitle


\section{Introduction}

We are interested in understanding the global dynamics of solutions to the Zakharov system
    \begin{equation}\label{eqn:Zak intro}
        \begin{split}
            i\p_t u + \Delta u &= v u, \\
            \alpha^{-2} \p_t^2 v - \Delta v &= \Delta |u|^2,
        \end{split}
    \end{equation}
in four spatial dimensions, where $(u, V)(t,x) :\RR\times \RR^4 \to \CC\times \CC$ and $\alpha>0$ is a fixed real constant. The equation \eqref{eqn:Zak} was introduced by Zakharov \cite{Zakharov1972} as a model for Langmuir turbulence in plasma, where $u$ denotes the slow oscillation of the magnetic field, $v$ is the ion density, and $\alpha>0$ is the ion sound speed. We refer the reader to \cite{Sulem2007, Texier2007, Zakharov1972} for further background on the derivation of the Zakharov system and its physical origin. The Zakharov system has received significant attention in the mathematical community, see for instance \cite{Bourgain1996, Bejenaru2010,Bejenaru2015,Bejenaru2008,Colliander2008c,Ginibre2006,Guo2021,Kato2017,Kenig1995,Merle1996} and the references therein.\\

It is more convenient for our purposes to set $\alpha=1$ and consider the first order formulation of the Zakharov system
    \begin{equation}\label{eqn:Zak}
        \begin{split}
            i\p_t u + \Delta u &= \Re(V) u, \\
            i\p_t V + |\nabla| V &= -|\nabla| |u|^2, \\
               (V, u)(0) &= (f, g) \in H^s \times H^\ell,
        \end{split}
    \end{equation}
which can be obtained  from \eqref{eqn:Zak intro} by letting $V = v - i |\nabla|^{-1}  \p_t v $ (and so $v = \Re(V)$). Sufficient regular solutions $(u, V)$ to \eqref{eqn:Zak} conserve both the energy
    \begin{equation}\label{eqn:zak energy}
         \mc{E}_Z\big( u(t), V(t)\big) = \int_{\RR^4} \frac{1}{2} |\nabla u(t,x)|^2 + \frac{1}{4} |V(t,x)|^2 + \frac{1}{2} \Re(V)(t,x) |u(t,x)|^2 dx,
    \end{equation}
and the (Schr\"odinger) mass
    \begin{equation}\label{eqn:Zak mass}
         \mc{M}\big(u(t)\big) = \int_{\RR^4} \frac{1}{2} |u(t,x)|^2 dx.
    \end{equation}
Although the system \eqref{eqn:Zak} lacks a natural scaling invariance, it should be thought of as energy critical in four spatial dimensions. For instance, the energy regularity $H^1\times L^2$ lies on the boundary of the admissible regularities for local well-posedness (in fact we have some form of ill-posedness if the wave regularity is lowered to $H^s$ for $s<0$, see for instance \cite{Grube2022} are the references therein). Furthermore, the potential energy in \eqref{eqn:zak energy} is just barely controlled by the kinetic energy in four dimensions,  namely the sharp Sobolev embedding $\dot{H}^1(\RR^4) \subset L^4(\RR^4)$ implies that
    $$ \int_{\RR^4} \Re(V) |u|^2 dx \les \| V \|_{L^2_x} \| u \|_{L^4_x}^2 \lesa \|V \|_{L^2_x} \| u \|_{\dot{H}^1}^2. $$

The Zakharov system \eqref{eqn:Zak} is closely connected to the nonlinear Schr\"{o}dinger equation (NLS). For instance, in the subsonic limit as the ion sound speed $\alpha \to \infty$, we have $v=-|u|^2$ and so formally the Zakharov system \eqref{eqn:Zak intro} reduces to the focussing cubic NLS
    \begin{equation}\label{eqn:NLS}
            i\p_t u + \Delta u = - |u|^2 u.
    \end{equation}
This convergence can be made rigorous, see for instance \cite{Schochet1986, Masmoudi2008} and the references therein. As the cubic NLS is energy critical in four dimensions, this gives another reason for considering the Zakharov system to be energy critical in four dimensions. The connection between the Zakharov system and cubic NLS also appears on the level of the energy, for instance we can write
        $$ \mc{E}_Z(f, g) = \mc{E}_{NLS}(f) + \frac{1}{4} \big\| g + |f|^2 \big\|_{L^2}^2  $$
where
        $$ \mc{E}_{NLS}(f) = \int_{\RR^4} \frac{1}{2} |\nabla f |^2  - \frac{1}{4} |f|^4 dx$$
is the conserved energy (or Hamiltonian) for the focussing cubic NLS on $\RR\times \RR^4$. Furthermore, a key role in the global dynamics of both the cubic NLS \eqref{eqn:NLS} and the Zakharov system \eqref{eqn:Zak} is played by the famous Aubin-Talenti function \cite{Aubin1976,Talenti1976}
        $$ Q(x) = \frac{8}{8 + |x|^2} $$
which serves as a ground state solution for both the NLS and Zakharov equations. In particular the stationary solutions $u(t) = Q$ to \eqref{eqn:NLS} and $(u, v)(t) = (Q, -Q^2)$ to \eqref{eqn:Zak intro} (and their rescalings) give  global solutions which do not disperse for large times. The existence of non-dispersing solutions is a typical feature of focussing equations, and immediately rules out the possibility of large data global existence and scattering.

The ground state solution $Q$ in fact gives a sharp threshold for linear scattering. More precisely, for the NLS \eqref{eqn:NLS} it is known that any solution with energy (strictly) below the ground state $Q$ is global and scatters to a free Schr\"odinger wave as $t\to \pm \infty$ \cite{Kenig2006, Dodson2019}. This is sharp as $u(t) = Q$ gives a (global, stationary) solution to \eqref{eqn:NLS} which clearly does not scatter to linear solution to the Schr\"odinger equation. The corresponding result for the Zakharov equation \eqref{eqn:Zak intro} in full generality is open. However recent work has shown that \emph{radial} solutions to \eqref{eqn:Zak intro} with energy below the ground state $(Q, -Q^2)$ are global and scatter \cite{Guo2021}. On the other hand, in the non-radial case, it has recently been shown \cite{Candy2021} that all solutions with data below the ground state are global in time. As we require this result later, we recall the precise statement from \cite{Candy2021}.

\begin{theorem}[GWP below ground state \cite{Candy2021}]\label{thm:gwp below ground state}
Let $(f, g)\in H^1\times L^2$ with
        $$ 4\mc{E}_Z(f,g) < \| Q \|_{\dot{H}^1}^2, \qquad \| g \|_{L^2}\les \| Q \|_{\dot{H}^1}. $$
Then there exists a unique global solution $(u,V)\in C(\RR, H^1\times L^2)$  to \eqref{eqn:Zak} with $u\in L^2_{t, loc} W^{\frac{1}{2}, 4}_x$ and data $(u,V)(0) = (f,g)$. Moreover the energy $\mc{E}_Z(u,V)$ is conserved and we have the global bounds
        $$ \max\big\{ \| u \|_{L^\infty_t \dot{H}^1_x}, \|V\|_{L^\infty_t L^2_x} \big\} \les 4 \mc{E}_Z(f,g). $$
\end{theorem}

\begin{remark}
The variational characterisation of the ground state $Q$ implies that we have
    $$ \mc{E}_Z(Q, -Q^2) = \mc{E}_{NLS}(Q) = \frac{1}{4} \| Q \|_{\dot{H}^1}^2, \qquad \|Q^2 \|_{L^2} = \|Q\|_{\dot{H}^1} $$
and thus we could alternatively state the assumed bound on the data in Theorem \ref{thm:gwp below ground state} as
     $$ \mc{E}_Z(f,g) < \mc{E}_Z(Q, -Q^2), \qquad \| g \|_{L^2}\les \| Q^2 \|_{L^2}. $$
It is also possible to replace the condition $\| g\|_{L^2} \les \|Q^2\|_{L^2}$ with $\min\{\|f\|_{\dot{H}^1}, \|g\|_{L^2}\}\les \|Q\|_{\dot{H}^1}$ (see for instance Lemma \ref{rem:ground state cond} below). This is not surprising in view of the corresponding assumption in the NLS case \cite{Kenig2006, Dodson2019}.
\end{remark}

The asymptotic behaviour of the global solutions given by Theorem \ref{thm:gwp below ground state} is currently open, although in case of radial solutions it is known that solutions scatter to free solutions to the Zakharov equation \cite{Guo2021}. However, a number of scattering criteria have been obtained. For instance all solutions which are dispersive, in the sense that a Strichartz type norm is finite over the time of evolution, can be shown to scatter \cite{Candy2021a}.

\begin{theorem}[Scattering criterion \cite{Candy2021a}]\label{thm:scat}
Let $(u, V)\in C(\RR; H^\frac{1}{2}\times L^2)$ be a global solution to \eqref{eqn:Zak} with
		$$ \|u \|_{L^2_t W^{\frac{1}{2}, 4}_x(\RR\times \RR^4)} < \infty \qquad \text{and} \qquad (u, V)(0)\in H^1\times L^2. $$
Then in fact $(u, V)\in C(\RR; H^1\times L^2)$ and moreover there exists $(f_{\pm}, g_{\pm})\in H^1\times L^2$ such that
		$$ \lim_{t\to \pm \infty} \big\| \big( e^{-it\Delta} u, e^{-it|\nabla|} V\big)(t) - (f_{\pm}, g_{\pm})\big\|_{H^1\times L^2} = 0. $$
\end{theorem}

It is worth pointing out that the sharp threshold regularity threshold for local well-posedness for the Zakharov equation in 4 dimensions is $(f, g) \in H^\frac{1}{2}\times L^2$, see for instance \cite{Bejenaru2015, Candy2021a, Kato2017}. In particular the energy regularity $H^1\times L^2$ is $\frac{1}{2}$ a derivative (at least for the Schr\"odinger component of the evolution) above the endpoint space $H^\frac{1}{2}\times L^2$. This should help to explain the regularity assumption $u\in L^2_t W^{\frac{1}{2}, 4}_x$ in Theorem \ref{thm:scat}. On the other hand, the wave regularity cannot be lowered, and in fact it is the lack of room in the wave regularity which is the key difficulty in the above theorems. In particular, even proving local well-posedness in the energy space $H^1\times L^2$ is slightly subtle (and in some sense more challenging than the endpoint $H^\frac{1}{2}\times L^2$) and in \cite{Bejenaru2015} was obtained by an additional indirect compactness argument. An alternative more direct argument was later found in \cite{Candy2021a}.\\

In view of the above results, together with the known linear scattering in the radial case and corresponding know results in the case of energy critical NLS \cite{Kenig2006, Dodson2019}, it is generally expected that the following conjecture holds.

\begin{conjecture}[Scattering below ground state]\label{conj:scattering}
Let $(f,g)\in H^1\times L^2$ with
		$$ 4\mc{E}_Z(f,g)<  \|Q\|_{\dot{H}^1}^2, \qquad  \| g \|_{L^2} \les  \| Q \|_{\dot{H}^1} .$$
Then there exists a unique global solution $(u, V) \in C(\RR, H^1\times L^2)$ to \eqref{eqn:Zak} with data $(u,V)(0)=(f,g)$ such that
			$$ \|  u \|_{L^2_t W^{\frac{1}{2}, 4}_x(\RR^{1+4})} < \infty. $$
\end{conjecture}

If Conjecture \ref{conj:scattering} holds, then Theorem \ref{thm:scat} immediately implies that any solution with energy below the ground state scatters to a linear solution as $t\to \pm\infty$. Moreover, for global solutions to \eqref{eqn:Zak intro}, it is not so difficult to prove that scattering is equivalent to boundedness of the dispersive norm $L^2_t W^{\frac{1}{2}, 4}_x$. In particular Conjecture \ref{conj:scattering} could also be stated in terms of a scattering condition as $t\to \pm \infty$. However it is technically more convenient to work with a dispersive condition such as $\| u \|_{L^2_t W^{\frac{1}{2}, 4}_x }<\infty$. \\

Our goal in the current article is to make progress towards the resolution of Conjecture \ref{conj:scattering}. Unfortunately, we are unable to resolve this question, essentially due to the lack of a suitable way to control the translation symmetry. In the case of the NLS or nonlinear wave equations, translations are typically controlled by using Galilei invariance/Lorentz invariance to reduce to the case of zero momentum. However the Zakharov equation is not invariant under either transformation, and thus removing translations seems to be a difficult obstruction. On other hand, in the positive direction, we can show that if Conjecture \ref{conj:scattering} fails, then there must be a \emph{minimal} counterexample, namely a minimal energy  non-scattering solution. Moreover this minimal counterexample must be \emph{almost periodic} in the sense that it satisfies a strong compactness property modulo spatial translations. It is a well-known phenomenon in dispersive PDE that the failure of scattering implies the existence of a minimal non-scattering solutions with strong compactness properties. In fact in \cite{Kenig2006} a general road map was developed to prove results of the form in Conjecture \ref{conj:scattering} via a concentration-compactness argument (to show the existence of a minimal energy non-scattering solution) together with a rigidity argument. Thus our results give the first step of this argument, but we are unable to provide the `rigidity' step. The existence of almost periodic solutions (or minimal energy/mass non-scattering solutions) was first obtained in \cite{Kenig2006, Keraani2006} in the context of the NLS. These arguments are closely related to original induction on energy approach \cite{Bourgain1999}  (see for instance \cite{Colliander2008b, Tao2008c, Killip2013} and the references therein) which produced \emph{approximate} almost periodic solutions.

To state out result precisely, we require additional notation.  Given $E>0$, we define the set
       \begin{equation}\label{eqn:defn Omega}
       \begin{split}
            \Omega(E) = \Big\{ (u,V)\in C(\RR;H^1\times L^2) \,\, \Big|\,\, (u, V) &\text{ solves \eqref{eqn:Zak} with }  u\in L^2_{t,loc} W^{\frac{1}{2},4}_x \\
                        &\text{ and } \mc{E}_Z(u, V)\les E, \,\,\|V\|_{L^\infty_t L^2_x} \les \|Q\|_{\dot{H}^1}\Big\}.
       \end{split}
       \end{equation}
Thus $\Omega(E)\subset C(\RR;H^1\times L^2)$ is the set of all $H^1\times L^2$ solutions to the Zakharov equation with energy below a threshold $E$, and wave mass uniformly below the ground state $Q$. Note that if $(f,g)\in H^1\times L^2$ with $\mc{E}_Z(f,g)\les E < \frac{1}{4}\|Q\|_{\dot{H}^1}^2$ and $\|g \|_{L^2} \les \|Q\|_{\dot{H}^1}$, then by Theorem \ref{thm:gwp below ground state} there exists a (unique) global solution $(u,V)\in \Omega(E)$. In particular, $\Omega(E)$ is non-empty. The condition $u\in L^2_{t,loc} W^{\frac{1}{2}, 4}$ is needed to ensure uniqueness, and is harmless in practise as all solutions constructed in this paper always satisfy this condition. There are a number of different ways to characterise the set $\Omega(E)$ when $E< \frac{1}{4} \|Q\|_{\dot{H}^1}^2$, see for instance Lemma \ref{lem:ground state} and Remark \ref{rem:ground state cond}.

It is tempting to use the set $\Omega(E)$ as the basis for a concentration compactness argument. However, due to the lack of a good well-posedness theory in the homogeneous energy space $\dot{H}^1\times L^2$, we impose an additional constraint on the mass of our solutions to ensure that the limiting critical element remains in $H^1\times L^2$ (instead of just $\dot{H}^1\times L^2$). More precisely, given $E, M>0$ we define
\begin{equation}\label{eqn:L(M, E) defn}
L(M, E) = \sup \Big\{ \| u\|_{L^2_t W^{\frac{1}{2}, 4}_x(\RR^{1+4})} \,\,\Big|\,\, (u, V) \in \Omega(E) \text{ and } \mc{M}(u)\les M\Big\}.
\end{equation}
Note that, in view of Theorem \ref{thm:scat}, scattering with mass/energy $(M, E)$ holds precisely when $L(M, E)<\infty$.

\begin{definition}[Mass/energy threshold]\label{defn:mass/energy}
We say a pair $(E_c, M_c) \in (0, \infty)\times (0, \infty)$ is an  \emph{mass/energy threshold} for the Zakharov equation if $L(M_c, E_c)=\infty$ and for any $\epsilon>0$ we have
            $$ L(M_c, E_c-\epsilon) + L(M_c-\epsilon, E_c) <\infty. $$
\end{definition}

In other words $(M, E)$ is a mass/energy threshold if: (i) the dispersive norm $\|\cdot\|_{L^2_t W^{\frac{1}{2}, 4}_x}$ blows up along some sequence of solutions with mass/energy at most $(M, E)$, and (ii) any solution with either less energy or less mass scatters. The existence of a mass/energy threshold with $4E_c<\|Q\|_{\dot{H}^1}^2$ is inconsistent with Conjecture \ref{conj:scattering}. In fact the failure of Conjecture \ref{conj:scattering} implies the existence of a mass/energy threshold, see Lemma \ref{lem:existence of threshold} below. More importantly, the failure of scattering must be witnessed by the existence of a minimal or critical element in the energy space. The following is our main result.

\begin{theorem}[Reduction to almost periodic solutions]\label{thm:min coun}
Suppose that Conjecture \ref{conj:scattering} fails. Then there exists a mass/energy threshold $(M_c, E_c)$ with $E_c<\frac{1}{4} \|Q\|_{\dot{H}^1}^2$ and a global solution $(\psi, \phi)\in C(\RR, H^1\times L^2)$ to \eqref{eqn:Zak} with $\psi \in L^2_{t,loc} W^{\frac{1}{2}, 4}_x$ such that
		$$ \mc{E}_Z(\psi, \phi)=E_c, \qquad \mc{M}(\psi)=M_c, \qquad \| \phi \|_{L^\infty_t L^2_x} \les \|Q\|_{\dot{H}^1},$$
and

        $$ \| \psi \|_{L^2_t W^{\frac{1}{2}, 4}_x((-\infty, 0]\times \RR^4)} = \| \psi \|_{L^2_t W^{\frac{1}{2}, 4}_x([0, \infty)\times\RR^4)} = \infty. $$
Moreover, there exists $x(t):\RR \to \RR^4$ such that the orbit $\{ (\psi, \phi)(t, x + x(t)) \mid t\in \RR\}$ is precompact in $H^1\times L^2$.
\end{theorem}

The conclusion of Theorem \ref{thm:min coun} implies that the orbit $\{ (\psi, \phi)(t) | t\in \RR\}$ is compact modulo translations, such solutions are often called \emph{almost periodic modulo translations}, see the discussion in \cite{Killip2013}. These solutions clearly do not scatter to free solutions as $t\to \pm \infty$, and should be thought as essentially `soliton like' as they retain their profile for large times. Theorem \ref{thm:min coun} is a stepping stone towards a resolution of Conjecture \ref{conj:scattering} as it reduces the proof of scattering to ruling out the existence of almost periodic solutions below the ground state. It is important to observe that the compactness is only modulo translations, in particular there is no scaling/frequency symmetries required. This is perhaps not surprising in view of the fact that the Zakharov equation is not invariant under rescalings. In some sense the lack of frequency scaling is a consequence of the fact that the energy space is subcritical in the Schr\"odinger regularity (but still critical in the wave regularity!),  but it can also be thought of as saying that the various asymptotic limits of \eqref{eqn:Zak} are not an obstruction to scattering.

The proof of Theorem \ref{thm:min coun} proceeds via the usual concentration compactness argument essentially as in \cite{Keraani2006, Kenig2006}, see also \cite{Killip2013} for an outline of this approach. Roughly the idea is that if Conjecture \ref{conj:scattering} failed, then there exists a mass/energy threshold with $(M_c, E_c)$ with $4E_c < \|Q\|_{\dot{H}^1}^2$. In particular, there exists a sequence of $H^1\times L^2$ solutions $(u_n, V_n)$ to \eqref{eqn:Zak intro} with mass and energy converging to $(M_c, E_c)$ and $\| u_n\|_{L^2_t W^{\frac{1}{2}, 4}_x} \to \infty$. Theorem \ref{thm:min coun} then follows by extracting a converging subsequence via a concentration compactness argument. It is here that we meet the first major difficulty. Typically extracting some compactness from the sequence $(u_n, V_n)$ relies on a profile decomposition in the spirit of \cite{Bahouri1999} with an error going to zero in some dispersive norm. The issue is that there is no profile decomposition known in the endpoint Strichartz space $L^2_t L^4_x$. Moreover, obtaining such a profile decomposition seems extremely difficult. Consequently the small data theory in \cite{Candy2021, Candy2021a} which required smallness in $L^2_t L^4_x$ is not sufficient to conclude Theorem \ref{thm:min coun}.

An alternative to trying to obtain a profile decomposition in the endpoint Strichartz space is to extend the known small data theory to rely on a more tractable dispersive norm. This is the approach we take here. In particular, we obtain a refined small data theory with smallness measured in an ``energy dispersion'' type norm. More precisely, define the norm
        $$ \| u \|_{Y(I)} = \sup_{\lambda \in 2^\NN} \lambda^{-4} \|u_{\lambda}\|_{L^\infty_{t,x}(I\times \RR^4)} $$
where $u_\lambda$ is the standard (inhomogeneous) Littlewood-Paley projection to frequencies $|\xi|+1 \approx \lambda$. The norm $\| u \|_{Y}$ roughly measures how dispersed a solution is. In particular, if $\| u \|_Y$ is small compared to the energy, then $u$ cannot be concentrated either spatially or in frequency. Moreover, as free Schr\"odinger waves disperse, we have $\| e^{it\Delta} f \|_{Y([T, \infty))} \to 0$ as $T\to \infty$. The $Y$ norm is much weaker than $H^\frac{1}{2}$, and at least at large frequencies, roughly scales like $\dot{H}^{-2}$. In fact a short computation using Bernstein's inequality and the endpoint Strichartz estimate gives
       \begin{equation}\label{eqn:Y bounded by sobolev intro}
        \| u \|_{Y(I)} \lesa \min\Big\{ \| u\|_{L^\infty_t L^2_x(I\times \RR^4)}, \| u \|_{L^\infty_t \dot{H}^1_x(I\times \RR^4)}\Big\} \qquad \text{and} \qquad \|e^{it\Delta} f \|_{Y(I)} \lesa \| e^{it\Delta} f \|_{L^2_t W^{\frac{1}{2}, 4}_x(I\times \RR^4)}.
       \end{equation}
This bound can clearly be improved (see \eqref{eqn:Y norm bound by sobolev} below) but suffices for our purposes. In particular, smallness in either mass or $\dot{H}^1$ implies that the energy dispersion norm $Y$ is also small. This ability to transfer smallness in homogeneous Sobolev spaces is crucial, as we eventually need to conclude that if a solution is close to a dispersive solution in energy, then it is also close in the norm $\|\cdot\|_Y$. Note that this is not true for the inhomogeneous dispersive norm $L^2_t W^{\frac{1}{2}, 4}_x$. We can now state a precise version of the small data theory that we require.

\begin{theorem}[Refined small data theory in $H^s\times L^2$]\label{thm:small data Hs}
Let $A>0$, $0<B<\|Q\|_{\dot{H}^1}$, and $\frac{1}{2}<s<2$. There exists $\epsilon=\epsilon(A,B,s)>0$ and $C=C(A, B, s)>0$  such that for any interval  $0\in I\subset \RR$, and any $(f, g) \in H^s\times L^2$ satisfying
    \begin{equation}\label{eqn:thm small data Hs:data cond}
       \|f \|_{H^s} \les A, \qquad \|e^{it\Delta} f\|_{Y(I)}\les \epsilon, \qquad \| g \|_{L^2_x} \les B,
    \end{equation}
there exists a (unique) solution $(u, V)\in C(I, H^s\times L^2)$ to \eqref{eqn:Zak} with data $(u, V)(0) = (f,g)$ and
        $$ \| u \|_{L^2_t W^{\frac{1}{2}, 4}_x(I\times \RR^4)} \les C \| f\|_{H^\frac{1}{2}}, \qquad \| V\|_{L^\infty_t L^2_x(I)} \les \tfrac{1}{2} \big( B + \|Q \|_{\dot{H}^1}\big). $$
\end{theorem}

Strictly speaking, we also require a generalisation of Theorem \ref{thm:small data Hs} that implies dispersive solutions are stable, see Theorem \ref{thm:stab II} for a precise statement. In view of the bound \eqref{eqn:Y bounded by sobolev intro}, Theorem \ref{thm:small data Hs} is more general than the small data theory developed in \cite{Candy2021a} which required smallness in the endpoint Strichartz space $L^2_t W^{\frac{1}{2}, 4}_x$. It is also worth noting that showing that smallness in $L^2$ suffices for scattering is straightforward via the convexity type bound
        $$ \| f\|_{H^\frac{1}{2}} \lesa \|f \|_{L^2}^\theta \| f \|_{H^s}^{1-\theta}. $$
On the other hand, proving that smallness in $\dot{H}^1$ suffices is significantly more challenging, as the low frequencies can be badly behaved. In particular there does not seem to be an easy direct argument to conclude that smallness in energy suffices for scattering, and instead it seems necessary to argue via the norm $\| \cdot \|_Y$ (or some similar alternative). \\

The proof of Theorem \ref{thm:small data Hs} is somewhat involved, and there are essentially three main obstructions to overcome: (i) Large wave data, (ii) Gaining a power of the $Y$ norm, (iii) Propagating smallness over large time intervals. We detail each of these obstructions below.

\begin{enumerate}
  \item (\textit{Large wave data.}) The wave data $g\in L^2$ in Theorem \ref{thm:small data Hs} is not small, and consequently we need to absorb the free wave into the left hand side of the equation. This issue was resolved in \cite{Candy2021} by proving a uniform Strichartz estimate for the Schr\"odinger equation with a free wave potential, and we exploit this estimate here.  As a simple toy model, consider the equation
                    $$ i\p_t u + \Delta u - \Re( e^{it|\nabla|} g ) u = F.$$
      Since $g$ is not small, it cannot be regarded as a perturbation over long time intervals. There are two approaches to resolving this issue. One is to exploit the dispersive effect of the free wave, and essentially observe that $e^{it|\nabla|}g$ converges to zero in $L^2 + L^\infty$, see for instance \cite{Bejenaru2015, Candy2021a}. Although this argument works for arbitrary wave data in $L^2$, it leads to estimates which are \emph{not uniform} in $g$. The lack of a uniform estimate is a major issue for the large data theory, as below the ground state we can control $\|V(t)\|_{L^2}$, but not its profile. In particular the key step in the global theory \cite{Candy2021} was the proof of a \emph{uniform} Strichartz estimate with a free wave potential, see also \cite{Guo2021} for the radial case. The uniform Strichartz estimate from \cite{Candy2021} also plays a key role in the proof of Theorem \ref{thm:small data Hs}, see Theorem \ref{thm:schro energy} below for a precise statement.\\

  \item (\textit{Gaining power of $Y$ norm.}) This step is the heart of the proof of Theorem \ref{thm:small data Hs}, and shows that the nonlinear terms in \eqref{eqn:Zak intro} are perturbative provided that $\|u\|_{Y(I)} \ll 1$. The small data theory for \eqref{eqn:Zak intro} essentially relies on a fixed point argument in a Banach space based on the endpoint Strichartz norm $L^2_t L^4_x$. Roughly speaking (ignoring derivatives for simplicity), after discarding the non-resonant terms and applying the endpoint Strichartz estimate \cite{Keel1998}, the argument reduces to a simple application of H\"older's inequality via the estimate
        \begin{equation}\label{eqn:intro Strichartz argument}
             \| V u \|_{L^2_t L^{\frac{4}{3}}_x} \lesa \|V \|_{L^\infty_t L^2_x} \| u \|_{L^2_t L^4_x}.
        \end{equation}
      Our goal is to then improve \eqref{eqn:intro Strichartz argument} by adding some power of the $Y$ norm to the righthand side, i.e. essentially replacing $\|u\|_{L^2_t L^4_x}$ with $\| u\|_{L^\infty_{t,x}}^\theta \| u \|_{L^2_t L^4_x}^{1-\theta}$. The issue is that in the fully resonant case (so both inputs and outputs have Fourier support close to their relevant characteristic surfaces) it seems very hard to improve on the Strichartz bound \eqref{eqn:intro Strichartz argument} as there is no room in the exponents. In particular, it does not seem possible to extract a factor $\|u\|_{L^\infty_{t,x}}$. Note that even placing $V\in L^q_t L^r_x$ for some wave admissable pair $(q,r)$ does not help, since the \emph{only} admissable wave pair for which we have
                $$ \| V u \|_{L^2_t L^\frac{4}{3}_x} \lesa \| V \|_{L^q_t L^r_x} \| u \|_{L^a_t L^b_x}$$
        with $(a,b)$ Schr\"odinger admissable, is $(q,r)=(\infty, 2)$ and $(a,b) = (2, 4)$. In the radial case, there is additional room in the Strichartz exponents, and thus the above issue can be resolved by placing $u\in L^2_t L^{4-\delta}_x$ for some $\delta>0$. In the non-radial case, the argument is more involved, and to create some room in the above chain of estimates, we instead require the inhomogeneous bilinear restriction estimates from \cite{Candy2021}. These estimates essentially extend the bilinear restriction estimates for the paraboloid obtained by Tao \cite{Tao2003a} to \emph{inhomogeneous} Schr\"odinger waves. More precisely, \cite{Candy2021} proved that for any $r>\frac{5}{3}$ and $\mu \in \ZZ$ we have the bound
           \begin{equation}\label{eqn:bilinear restriction intro}
                \| \dot{P}_\mu(\phi \psi) \|_{L^1_t L^r_x} \lesa \mu^{2-\frac{4}{r}} \Big( \| \phi(0)\|_{L^2} + \| (i\p_t + \Delta) \phi\|_{L^2_t L^{\frac{4}{3}}_x}\Big)\Big( \| \psi(0)\|_{L^2} + \| (i\p_t + \Delta) \psi\|_{L^2_t L^{\frac{4}{3}}_x}\Big)
           \end{equation}
     where $\dot{P}_\mu$ denotes the homogeneous Littlewood-Paley projection to frequencies $|\xi|\approx \mu \in 2^\ZZ$. Note that the case $r=2$ follows from H\"older's inequality and the endpoint Strichartz estimate. The point of the bilinear restriction type estimate \eqref{eqn:bilinear restriction intro} is that it allows $r<2$, and adapting the argument from \cite{Candy2021}, this additional room essentially allows us to improve \eqref{eqn:intro Strichartz argument} to something of the form
            $$\Big\| \int_0^t e^{i(t-s)\Delta} \big(\Re(V) u\big)(s) ds \Big\|_{L^2_t L^4_x} \lesa\|V \|_{W}  \|u\|_{Y}^\theta \| u \|_{S}^{1-\theta}$$
     where $S$ is our iteration norm for the Schr\"odinger component of the evolution, and $W$ is the corresponding norm for the wave evolution, see Section \ref{sec:bil est} for the details. \\

     \item (\textit{Propagating smallness over large time intervals.}) The bilinear estimates obtained in (ii) essentially allow use to conclude that the nonlinearity in \eqref{eqn:Zak intro} are perturbative provided we have smallness in the energy dispersed norm $\| \cdot \|_{Y}$. In particular, after replacing the wave component of the evolution by the corresponding Duhamel formula, (ii) allows us to schematically prove a bound of the form
                $$ \| u \|_S \lesa \text{data} + \| u \|_S^{3-\theta} \|u\|_Y^\theta. $$
        However there is a technical issue in that the power $\theta>0$ can be very small (and much smaller than $1$), and thus even if the data is small in $\|\cdot \|_Y$, the estimates in (ii) are not sufficient to directly conclude that the solution will remain small in $Y$. In particular, iterating this smallness over a potentially large number of intervals becomes technically challenging. Potentially this obstruction could be addressed by proving sharper estimates with $\|u\|_Y$ on the left hand side (or working harder in the iteration stage). However, here we take a simple and more flexible approach to this issue, by first obtaining a small data theory based on a tractable controlling quantity, and then second, proving that this controlling quantity can be bounded by the energy dispersed norm $\|\cdot \|_Y$.

        To illustrate the general idea, suppose we were trying to construct a fixed point to the (cubic) nonlinear problem
                    \begin{equation}\label{eqn:cubic intro}
                         u = u_0 + \mc{I}[u^3]
                    \end{equation}
        where as usual $u_0$ is the free evolution of the data and $\mc{I}$ denotes the corresponding Duhamel term. Suppose we  have a Banach space $X$ satisfying the nonlinear estimate
                    \begin{equation}\label{eqn:gen nonlinear bounds intro}
                             \| \mc{I}[u v w]\|_X \lesa \|u\|_X \|v \|_X \| w \|_X.
                    \end{equation}
        The bound \eqref{eqn:gen nonlinear bounds intro} together with a standard fixed point argument implies that if $u_0 \in X$ is small, then \eqref{eqn:cubic intro} has a solution $u\in X$. Our key observation is that the smallness condition on $u_0$ can be weakened significantly by the following simple observation. Namely, define the \emph{controlling quantity}
                    $$ \rho(u) = \sup_{\|\phi\|_X\les 1} \| \mc{I}[\phi^2 u ]\|_X $$
        which roughly measures the nonlinear interaction of $u$ with a generic element $\phi \in X$. Then for any $A>0$ there exists $\epsilon>0$ such that if the (free evolution of the) data satisfies
                    \begin{equation}\label{eqn:data assump illus intro}
                        \| u_0\|_X \les A, \qquad \rho(u_0) \les \epsilon
                    \end{equation}
        then a fixed point $u\in X$ exists, and moreover we have $\rho(u)\lesa_A \epsilon$ (see Subsection \ref{subsec:controlling quantity} below). Clearly in view of the bound \eqref{eqn:gen nonlinear bounds intro} if $u_0$ is small in $X$, then \eqref{eqn:data assump illus intro} holds. On the other hand, if say we could improve \eqref{eqn:gen nonlinear bounds intro} to
                    \begin{equation}\label{eqn:non bound with gain intro} \| \mc{I}[u^3]\|_X \lesa \|u\|_D^\theta \| u \|_X^{3-\theta}\end{equation}
        for some $\theta>0$ and some dispersive norm $\|\cdot \|_D$, then assuming $\|u_0\|_D$ is small the definition of $\rho$ together with \eqref{eqn:non bound with gain intro} would imply that \eqref{eqn:data assump illus intro} holds. Of course, assuming the bound \eqref{eqn:non bound with gain intro} holds, it is also possible to argue directly (i.e. without proceed via $\rho$) that we have a solution $u\in X$ to \eqref{eqn:cubic intro}. However arguing via the controlling quantity $\rho$  has the advantage that: (a) smallness $\rho(u)\lesa_A \epsilon$ is preserved, (b) it is flexible as it completely separates the issue of the improved nonlinear estimate \eqref{eqn:non bound with gain intro} from the existence theory, (c) we only need to prove smallness of the free evolution $\rho(u_0)$. To illustrate (a), note that if $\theta<1$ in \eqref{eqn:non bound with gain intro}, then arguing via $Y$ it is not possible to conclude that $\|u\|_Y \lesa_A \epsilon$. The advantage of (b) is that if at a later point a more convenient dispersive norm $D'$ arises (or alternatively some more involved condition on the data, for instance concentration along light rays as in \cite{Tao2009b}), it is not necessary to repeat the small data theory. It would simply suffice to prove that smallness in $D'$ implies the controlling quantity $\rho(u_0)$ is small (for instance via a version of the improved nonlinear estimate \eqref{eqn:non bound with gain intro}). Finally, (c) is a somewhat minor point, and simply makes the observation that as we only need to prove $\rho(u_0)$ is small, we may replace $u^3$ with $u^2 u_0$ on the lefthand side of \eqref{eqn:non bound with gain intro} which in some settings may be helpful. In the application to the Zakharov equation, the necessary controlling quantity $\rho$ is slightly more involved (as we are dealing with a system), but the underlying idea remains the same. Thus we first provide a small data theory (relying on the estimates from \cite{Candy2021, Candy2021a}) with smallness measured in some controlling quantity $\rho$, and then later prove refined bilinear estimates which allows us to conclude that smallness in $Y$ implies smallness of the controlling quantity.
\end{enumerate}

\subsection{Outline of Paper}

Section \ref{sec:notation} fixes the notation used throughout this paper, defines the key function spaces that are used to control solutions to the Zkharov equation, and recalls a number of key estimates from the papers \cite{Candy2021, Candy2021a}. In Section \ref{sec:small data}, we turn to the small data theory for the Zakharov equation, assuming that a certain controlling norm is small. In particular, we give a persistence of regularity result, and prove a key stability property for the Zakharov equation. The results in Section \ref{sec:small data} only rely on the estimates from \cite{Candy2021, Candy2021a}. The improved bilinear estimates giving a gain in the energy dispersed norm are then proved in Section \ref{sec:bil est}. In Section \ref{sec:small data II} we apply the improved bilinear estimates together with the small data theory from Section \ref{sec:small data} to give the proof of Theorem \ref{thm:small data Hs}, as well as give a stability theory based on the energy dispersed norm $\|\cdot\|_Y$.

The final four sections contain the concentration compactness and variational arguments needed to prove Theorem \ref{thm:min coun}. More precisely, in Section \ref{sec:prof decomp} we state a version of the profile decomposition of Bahouri-G\'erard \cite{Bahouri1999} while in Section \ref{sec:ground state constraint} we recall the variational properties of the Zakharov equation from \cite{Guo2021}. The key step in the proof of existence of minimal mass/energy blow up solutions is given in Section \ref{sec:P-S} where we prove a fundamental Palais-Smale type condition which gives a crucial compactness property of bounded sequences of solutions to the Zakharov equation. Finally in Section \ref{sec:almost per} we prove the existence of a mass/energy threshold (under the assumption that Conjecture \ref{conj:scattering} fails), and conclude the proof of Theorem \ref{thm:min coun}.

\section{Notation}\label{sec:notation}

\subsection{Fourier Multipliers}

We let $\chi\in C^\infty_0(\RR)$ denote a smooth function satisfying the standard Littlewood-Paley conditions
		$$\supp \chi \subset \{\tfrac{1}{2}\les r \les 2\}, \qquad 0\les \chi \les 1, \qquad \sum_{\lambda \in 2^\ZZ} \chi(\lambda^{-1} r) = 1 \,\,\text{ for all $r>0$}. $$
The corresponding standard inhomogeneous Littlewood-Paley projections are then given by
		$$ P_1 = \sum_{\substack{\mu \in 2^\ZZ \\ \mu \les 1}} \chi(\mu^{-1} |\nabla|), \qquad P_\lambda = \chi(\lambda^{-1} |\nabla|)  \,\, \text{ for $\lambda\in 2^\NN$, $\lambda>1$.}$$
We also require the homogeneous version. To this end, we take $\dot{P}_\lambda$ with $\lambda \in 2^\ZZ$ to be the homogeneous Littlewood-Paley multipliers. We often use the short hand $u_\lambda = P_\lambda u$, this is reserved exclusively for the inhomogeneous Littlewood-Paley projection, thus $\lambda \in 2^\NN$.

Given $\mu\in 2^\ZZ$, we let
		$$C_{\les \mu} = \sum_{\substack{\nu \in 2^\ZZ \\ \nu \les \mu}} \chi\big( \nu^{-1} |i\p_t + \Delta|\big)$$
be a smooth Fourier multiplier with support in the set $\{ |\tau + |\xi|^2| \les 2\mu\}$. We also require a temporal version of the Littlewood-Paley decomposition. To this end, we let
		$$ P^{(t)}_{\les \mu} = \sum_{\substack{\nu \in 2^\ZZ \\ \nu \les \mu}} \chi\big( \nu^{-1} |\p_t|\big).$$
Thus $P^{(t)}_{\les \mu}$ restricts the temporal Fourier support to the set $|\tau|\les 2\mu$.

\subsection{Function Spaces}\label{subsec:function spaces}
We fix the spatial dimension to be four, and often use the short hand $L^q_t L^r_x = L^q_t L^r_x(\RR^{1+4})$. The arguments to follow unfortunately involve a large number of norms. Similar to \cite{Candy2021} (except for the dispersive norm $Y$), these norms can be seen as special cases of the norms appearing in \cite{Candy2021a}. These come in two flavours, the Schr\"odinger variants $S^s$, $\underline{S}^s$, and $N^s$, and the wave variants $W^\ell$, $\underline{W}^\ell$, and $R^\ell$. The solution $(u, V)$ is controlled in $S^s\times W^\ell$, with the underlined versions of the spaces giving slightly stronger control. The nonlinear terms are then estimated in $N^s\times R^\ell$.

In more detail, for $\lambda \in 2^\NN$ and $s\in \RR$, we control the frequency localised Schr\"odinger component of the evolution using the norm
	$$ \| u \|_{S_\lambda^s} = \lambda^s \| u \|_{L^\infty_t L^2_x(\RR^{1+4})} + \lambda^s \| u \|_{L^2_t L^4_x(\RR^{1+4})} + \lambda^{s-1} \| (i\p_t + \Delta) u \|_{L^2_{t,x}(\RR^{1+4})}  $$
and the stronger norm
	$$ \| u \|_{\underline{S}^s_\lambda} = \lambda^s \| u \|_{L^\infty_t L^2_x(\RR^{1+4})} +  \| (i\p_t + \Delta) u \|_{N^s_\lambda} $$
where
    $$ \| F \|_{N^s_\lambda} = \lambda^s \| C_{\les 2^{-16} \lambda^2} F \|_{L^2_t L^\frac{4}{3}_x(\RR^{1+4})} + \lambda^{s-1} \| F \|_{L^2_{t,x}(\RR^{1+4})}. $$
Similarly, to estimate the frequency localised wave evolution, given $\lambda \in 2^\NN$ and $\ell \in \RR$ we use the norm
		$$ \| v \|_{W^{\ell}_\lambda} = \lambda^\ell \| v \|_{L^\infty_t L^2_x(\RR^{1+4})} + \lambda^{\ell - 1} \| (i\p_t + |\nabla|) v \|_{L^2_{t,x}(\RR^{1+4})} $$
together with the stronger variant
		$$ \| v \|_{\underline{W}^\ell_\lambda} = \lambda^\ell \| v \|_{L^\infty_t L^2_x(\RR^{1+4})} + \| (i\p_t + |\nabla|) v \|_{R^\ell_\lambda} $$
where
    $$ \| G \|_{R^\ell_\lambda} =  \lambda^{\ell} \| P^{(t)}_{\les 2^{-16} \lambda^2 + 2^{16}} G \|_{L^1_t L^2_x(\RR^{1+4})} + \lambda^{\ell -1 } \| G \|_{L^2_{t,x}(\RR^{1+4})} $$
and $\lambda\in 2^\ZZ$ and $\ell\in \RR$. The iteration spaces are then formed by summing up over frequencies $\lambda \in 2^\NN$ in $\ell^2(\NN)$, thus
        $$ \| u \|_{S^s} = \Big( \sum_{\lambda \in 2^\NN} \| u_\lambda \|_{S^s_\lambda}^2\Big)^\frac{1}{2}, \qquad \| u \|_{\underline{S}^s} =\Big( \sum_{\lambda \in 2^\NN} \| u_\lambda \|_{\underline{S}^s_\lambda}^2\Big)^\frac{1}{2}$$
and
       $$ \| v \|_{W^\ell} = \Big( \sum_{\lambda \in 2^\NN} \| v_\lambda \|_{W^\ell_\lambda}^2\Big)^\frac{1}{2}, \qquad \| v \|_{\underline{W}^\ell} =\Big( \sum_{\lambda \in 2^\NN} \| v_\lambda \|_{\underline{W}^\ell_\lambda}^2\Big)^\frac{1}{2}$$
and
        $$ \| F \|_{N^s} = \Big( \sum_{\lambda \in 2^\NN} \| F_\lambda \|_{N^s_\lambda}^2\Big)^\frac{1}{2}, \qquad \| G \|_{R^\ell} =\Big( \sum_{\lambda \in 2^\NN} \| G_\lambda \|_{R^\ell_\lambda}^2\Big)^\frac{1}{2}.$$
The spaces $S^s$ and $\underline{S}^s$ are nested, contain all $H^s$ free solutions to the Schr\"odinger equation, and moreover also control the Strichartz spaces $L^q_t L^r_x$. More precisely, an application of \cite[Lemma 2.1]{Candy2021a} and the endpoint Strichartz estimate \cite{Keel1998} shows that for any Schr\"odinger admissible exponents $(q,r)$ on $\RR\times \RR^4$, thus $2\les q, r \les  \infty$ and $\frac{1}{q} +\frac{2}{r} = 1$, and any $\lambda \in 2^\NN$ we have the bounds
        \begin{equation}\label{eqn:schro stri}
          \lambda^s  \| u_\lambda \|_{L^q_t L^r_x(\RR^{1+4})}  \lesa \| u_\lambda\|_{S^s_\lambda} \lesa \| u_\lambda \|_{\underline{S}^s_{\lambda}}, \qquad \| e^{it\Delta} f \|_{\underline{S}^s} \lesa \| f \|_{H^s}.
        \end{equation}
A similar bound applies in the wave case, although strictly speaking only the stronger space $\underline{W}^\ell$ controls the (wave admissible) spaces $L^a_t L^b_x$. In fact this was the key reason to introduce the stronger variant $\underline{W}^\ell$. In more detail, if $(a,b)$ are wave admissible on $\RR\times \RR^4$, thus $2\les a, b \les \infty$ and $\frac{1}{a} + \frac{3}{2r} = \frac{3}{2}$, then we have bounds
        \begin{equation}\label{eqn:wave stri}
            \lambda^{\ell - \frac{5}{3a}} \|  V_\lambda \|_{L^a_t L^b_x(\RR^{1+4})} \lesa \| V_\lambda \|_{\underline{W}^\ell_\lambda}, \qquad \| V_\lambda \|_{W^\ell_\lambda} \lesa \|V_\lambda \|_{\underline{W}^\ell_\lambda}, \qquad \| e^{it|\nabla|} g \|_{\underline{W}^\ell} \lesa \| f \|_{H^\ell}.
        \end{equation}
Strictly speaking the bounds in \eqref{eqn:wave stri} are not contained in \cite{Candy2021a}, but follow from an identical argument. For instance, after decomposing $V_\lambda = P^{(t)}_{<2^{-16}\lambda + 2^{16}} V_\lambda + P^{(t)}_{\g 2^{-16}\lambda + 2^{16}} V_\lambda$ an application of Bernstein's inequality and the endpoint free Strichartz estimate for the wave equation \cite{Keel1998} gives for any $\lambda \in 2^\NN$
        \begin{align*}
             \| V_\lambda \|_{L^2_t L^6_x}&\les \| P^{(t)}_{<2^{-16}\lambda + 2^{16}} V_\lambda\|_{L^2_t L^6_x}+ \lambda^{\frac{4}{3}} \|P^{(t)}_{\g 2^{-16}\lambda + 2^{16}} V_\lambda\|_{L^2_{t,x}}\\
              &\lesa \lambda^{\frac{5}{6}}\Big( \| P^{(t)}_{<2^{-16}\lambda + 2^{16}} V_\lambda(0)\|_{L^2_x} + \| P^{(t)}_{<2^{-16}\lambda + 2^{16}} V_\lambda\|_{L^1_t L^2_x}\Big) + \lambda^{\frac{1}{3} - 1} \| (i\p_t + |\nabla|) V_\lambda \|_{L^2_{t,x}} \\
              &\lesa \lambda^{\frac{5}{6} - \ell} \|V_\lambda \|_{\underline{W}^\ell}.
        \end{align*}
The convexity of $L^p$ norms, together with the trivial bound $\| V_\lambda \|_{L^\infty_t L^2_x} \lesa \| V_\lambda \|_{\underline{W}^\ell}$, then gives the first inequality in \eqref{eqn:wave stri}. The remaining bounds in \eqref{eqn:wave stri} are immediate from the definition of the norms.

We also require some way to measure how dispersed a solution is, this is accomplished via the ``dispersive'' norms
       $$ \| u \|_{D} = \| u \|_{L^2_t W^{\frac{1}{2}, 4}_x(\RR^{1+4})}, \qquad \| u \|_{Y} =  \sup_{\lambda \in 2^\NN} \lambda^{-4} \| u_\lambda \|_{L^\infty_{t,x}(\RR^{1+4})}.$$
These norms form a crucial component of the analysis, and are used to extract some smallness out of the error terms in the profile decomposition. Note that the Strichartz control \eqref{eqn:schro stri} implies that
    \begin{equation}\label{eqn:D bdded by S}
          \|u\|_{D} \lesa \Big( \sum_{\lambda \in 2^\NN} \lambda \| u_\lambda \|_{L^2_t L^4_x}^2\Big)^\frac{1}{2} \lesa \| u \|_{S^\frac{1}{2}}
    \end{equation}
and thus the dispersive norm $\|\cdot\|_D$ is controlled by the iteration norm $\| \cdot \|_{S^\frac{1}{2}}$. On the other hand, the norm $\| \cdot\|_Y$ can be controlled by the \emph{homogeneous} Sobolev spaces. More precisely, a short computation using Bernstein's inequality gives for any $-2\les s < 2$
        \begin{align}
        \| u \|_{Y} &\les \sum_{\mu\in 2^\ZZ, \mu \les 1} \|\dot{P}_\mu u \|_{L^\infty_{t,x}} + \sup_{\lambda\in 2^\NN, \lambda>1} \lambda^{-4} \| u_\lambda \|_{L^\infty_{t,x}} \notag \\
                    &\lesa \sum_{\mu\in 2^\ZZ, \mu \les 1} \mu^2 \|\dot{P}_\mu u \|_{L^\infty_t L^2_x} + \sup_{\lambda\in 2^\NN, \lambda>1} \lambda^{-2} \| u_\lambda \|_{L^\infty_t L^2_x} \notag \\
                    &\lesa \sup_{\mu\in 2^\ZZ} \mu^s \| \dot{P}_\mu  u \|_{L^\infty_t L^2_x} \lesa \| u \|_{L^\infty_t \dot{H}^s_x}. \label{eqn:Y norm bound by sobolev}
        \end{align}
Moreover, at least for free Schr\"odinger waves, the two dispersive norms can be compared by noting that another application of Bernstein's inequality (in both $t$ and $x$) together with the endpoint Strichartz estimate gives
    $$ \| e^{it\Delta} f\|_{Y} \lesa \sup_{\lambda\in 2^\NN} \lambda^{-2} \| e^{it\Delta} f_\lambda \|_{L^2_tL^4_x} \lesa \|e^{it\Delta} f \|_{D}.$$
Note that there is a  derivative loss in this bound. Roughly the $Y$ norm scales like $\dot{H}^{-2}_x$, while $D$ scales like $\dot{H}^\frac{1}{2}$ (at least at high frequencies). The choice of scaling for the $Y$ norm is slightly arbitrary however. In fact, due to the fact that the energy is subcritical (at least for the Schr\"odinger component of the Zakharov evolution), we could equally well have worked with a norm that scales at high frequencies like $\dot{H}^{-N}$ for any $N$.

Finally, a key step in proof of the main bilinear estimate requires the bilinear restriction theory for the paraboloid. As in \cite{Candy2021}, the precise estimates we need are phrased in terms of the inhomogeneous Strichartz type norm
		$$ \| u \|_Z = \| u \|_{L^\infty_t L^2_x} + \| (i\p_t+ \Delta) u \|_{L^1_t L^2_x + L^2_t L^4_x}. $$

Given a Banach space of functions $X$ defined on $\RR\times \RR^d$ together with an interval $I\subset \RR$, we define the temporally restricted spaces $X(I)$ via the usual construction
        $$ \| u \|_{X(I)} = \inf_{\tilde{u}\in X \text{ and } \tilde{u}|I = u} \| \tilde{u} \|_X $$
with the convention that $\| u \|_{X(I)}=\infty$ if no such extension exists. As we are using the standard convention $L^q_t L^r_x = L^q_t L^r_x(\RR\times \RR^4)$, the above definition together with a simple argument gives the identity
        $$ \| u \|_{L^q_t L^r_x(I)} = \| u \|_{L^q_t L^r_x(I\times \RR^4)}. $$
This short hand is used frequently in the following. \\

\subsection{Solution Operators}
Define the inhomogeneous solution operators for the Schr\"odinger equation and wave equation with zero data at $t=t_0$ as
	$$\mc{I}_{0, t_0}[F] = -i \int_{t_0}^t e^{i(t-s)\Delta} F(s) ds, \qquad \mc{J}_{0, t_0}[G] = -i \int_{t_0}^t e^{i(t-s)|\nabla|} G(s) ds. $$
In the special case $t_0=0$ we simply write $\mc{I}_{0, t_0} = \mc{I}_0$ and $\mc{J}_{0, t_0} = \mc{J}_0$. The arguments in \cite[Lemma 2.4 and Lemma 2.6]{Candy2021a} give for any $s, \ell \in \RR$ the energy type bounds
    \begin{equation}\label{eqn:energy ineq}
        \|\mc{I}_{0, t_0}[F] \|_{\underline{S}^s} \lesa \| F \|_{N^s}, \qquad \|\mc{J}_{0, t_0}[G]\|_{\underline{W}^\ell} \lesa \| G \|_{G^\ell}.
    \end{equation}
In the following, we frequently have to work with the Duhamel expressions $\mc{I}_{0, t_0}[F]$ in the local in time spaces $S^s(I)$. To make this slightly easier, we have the following useful relationship on the local spaces $S^s(I)$ which roughly gives a distinguished extension of $\mc{I}_{0, t_0}[F]$ from the interval $I$ to $\RR$.

\begin{lemma}[Local in time characterisation]\label{lem:duh loc time}
Let $s, \ell \in \RR$. There exists a constant $C>0$ such that for any interval $I\subset \RR$ and any initial time $t_0\in I$ we have
    $$  C^{-1} \| \mc{I}_{0, t_0}[F]\|_{S^s(I)} \les \| \mc{I}_0[\ind_I F] \|_{S^s}\les C  \| \mc{I}_{0, t_0}[F]\|_{S^s(I)}$$
and
    $$C^{-1} \| \mc{J}_{0, t_0}[G]\|_{W^\ell(I)} \les \| \mc{J}_0[\ind_I G] \|_{W^\ell}\les C  \| \mc{J}_{0, t_0}[G]\|_{W^\ell(I)}. $$
\end{lemma}
\begin{proof}
We only prove the Schr\"odinger case as the wave case is similar. We begin by observing that for any $\tau, \tau'\in \RR$ we have the identity
		\begin{equation}\label{eqn:duh loc time:ident}
			 \mc{I}_{0, \tau}[F](t) = \mc{I}_{0, \tau'}[F](t) - e^{i(t-\tau)\Delta} \mc{I}_{0, \tau'}[F](\tau).
		\end{equation}
Hence, after noting that
		$$ \| e^{i(t-\tau)\Delta} \mc{I}_{0, \tau'}[F](\tau) \|_{S^s} \lesa \| \mc{I}_{0, \tau'}[F](\tau) \|_{H^s} \lesa \| \mc{I}_{0, \tau'}[F] \|_{L^\infty_t H^s_x} \lesa \| \mc{I}_{0,\tau'}[F] \|_{S^s}, $$
we conclude that for any $\tau \in \RR$ we have the useful bound
		\begin{equation}\label{eqn:duh loc time:init trans}
			\| \mc{I}_{0}[F] \|_{S^s} \approx \|  \mc{I}_{0, \tau}[F] \|_{S^s}.
		\end{equation}
Let $t_0\in I \subset \RR$ and suppose that $\psi \in S^s$ is an extension of $\mc{I}_{0, t_0}[F]$ from $I$ to $\RR$. Then letting $t_-< t_+$ denote the endpoints of the interval $I$ we have
        $$ \mc{I}_{0, t_0}[ \ind_I F](t) = \begin{cases}  e^{i(t-t_+)\Delta} \psi(t_+), &t_+<t, \\
 											\psi(t), &t_-\les t \les t_+, \\
 											e^{i(t-t_-)\Delta}\psi(t_-), &t<t_-.
 											\end{cases} $$
and therefore the definition of the norm $\|\cdot \|_{S^s_\lambda}$ gives for any $\lambda \in 2^\NN$
	\begin{align*}
		\|\mc{I}_{0,t_0}[\ind_I F_\lambda]\|_{S^s_\lambda} &\lesa \lambda^s \Big( \| \psi_\lambda \|_{L^\infty_t L^2_x \cap L^2_t L^4_x(I)} + \| e^{i(t-t_+)\Delta} \psi_\lambda(t_+) \|_{L^\infty_t L^2_x \cap L^2_t L^4_x(t>t_+)}\\
		&\qquad \qquad + \|e^{i(t-t_-)\Delta}\psi_\lambda(t_-) \|_{L^\infty_t L^2_x \cap L^2_t L^4_x(t<t_-)} + \lambda^{-1} \| (i\p_t + \Delta) \psi_\lambda \|_{L^2_{t,x}(I)} \Big)\\
					&\lesa 	\lambda^s \Big( \| \psi_\lambda \|_{L^\infty_t L^2_x \cap L^2_t L^4_x} + \lambda^{-1}\| (i\p_t + \Delta)\psi_\lambda \|_{L^2_{t,x}}\Big) = \|\psi_\lambda \|_{S^s_\lambda}.
	\end{align*}
Consequently, after summing up over frequencies and applying \eqref{eqn:duh loc time:init trans}, we obtain
    $$ \| \mc{I}_0[\ind_I F] \|_{S^s} \approx \| \mc{I}_{0, t_0}[\ind_I F ]\|_{S^s} \lesa \| \psi \|_{S^s}. $$
Taking the infimum over all possible extensions, the definition of the time restricted spaces $S^s(I)$ gives
    $$ \| \mc{I}_0[\ind_I F ] \|_{S^s} \lesa \| \mc{I}_{0, t_0}[F]\|_{S^s(I)}. $$
To prove the converse inequality, we simply note that since $t_0\in I$ we have the identity
        $$ \ind_I \mc{I}_{0, t_0}[F] = \ind_I \mc{I}_{0, t_0}[\ind_I F]$$
and hence again applying \eqref{eqn:duh loc time:init trans} we see that
        $$ \| \mc{I}_{0, t_0}[F] \|_{S^s(I)} \les \| \mc{I}_{0, t_0}[\ind_I F ] \|_{S^s} \lesa \| \mc{I}_0[\ind_I F] \|_{S^s}. $$
\end{proof}

Due to the presence of the derivative in time, bounding $\| u \|_{S^s(I\cup J)}$ by the sum $\| u \|_{S^s(I)} + \| u \|_{S^s(J)}$ involves losing a factor which depends on the size of the intervals \cite[Lemma 2.8]{Candy2021a}. This is a slight technical inconvenience. However the Duhamel operator $\mc{I}_{0, t_0}$ \emph{does} satisfy a useful decomposability bound which is independent of the intervals.

\begin{lemma}[Duhamel operator is decomposable]\label{lem:duh decomposable}
Let $s\in \RR$ and $I, I_j \subset \RR$, $1\les j \les N$, be intervals with $I= \cup_{j=1}^N I_j$. Then for any $t_0\in I$ and $t_j\in I_j$ we have
	$$ \|\mc{I}_{0, t_0}[F]\|_{S^s(I)} \lesa \sum_{j=1}^N \| \mc{I}_{0, t_j}[F]\|_{S^s(I_j)} $$
\end{lemma}
\begin{proof}
The proof is a direct consequence of Lemma \ref{lem:duh loc time}. Choose disjoint intervals $I_j^*\subset I_j$ such that $t_j\in I_j^*$ and $I=\cup I_j^*$. Then as $t_0\in I$, two applications of Lemma \ref{lem:duh decomposable} give
	\begin{align*}
		\| \mc{I}_{0, t_0}[F] \|_{S^s(I)} \approx  \| \mc{I}_{0}[\ind_I F] \|_{S^s} \les \sum_{j=1}^N \| \mc{I}_{0}[\ind_{I_j^*} F]\|_{S^s} \approx \sum_{j=1}^N \| \mc{I}_{0, t_j}[F]\|_{S^s(I_j^*)} \lesa \sum_{j=1}^N \| \mc{I}_{0, t_j}[F]\|_{S^s(I_j)}. 	
    \end{align*}
\end{proof}

\subsection{Solution Operators with Potentials}
Our eventual goal is to understand solutions to the Zakharov equation with data close to the ground state. In particular, this means we need to consider the Schr\"odinger flow with a free wave potential which is not small. To this end, following \cite{Candy2021} we introduce notation for the solution operator for the linear flow perturbed by a potential. More precisely, given a time dependent potential $v(t,x):\RR^{1+4} \to \RR$, we let $\mc{U}_{v, t_0}(t) f$ denote the solution to the homogeneous problem
            $$  i\p_t u + \Delta u - v u = 0, \qquad u(t_0) = f. $$
Similarly, we let $\mc{I}_{v, t_0}[F]$ denote the solution to the inhomogeneous problem
    $$ (i\p_t + \Delta - v) u = F, \qquad u(t_0) = 0. $$
As previously, in the special case where the initial time $t_0=0$, we write $\mc{U}_{v,t_0} = \mc{U}_v$, $\mc{I}_{v, t_0} = \mc{I}_v$. Under reasonable assumptions on the potential $v$, data $f$, and forcing term $F$, the operators $\mc{U}_{v, t_0}f, \mc{I}_{v, t_0}[F] \in C(I; H^s)$ are always well-defined \cite{Candy2021}. For instance if the potential $v = \Re( e^{it|\nabla|} g) \in L^\infty_t L^2_x$ is a free wave, then for any $0\les s < 1$ we have the energy estimate
        $$ \| \mc{I}_{v, t_0}[F] \|_{\underline{S}^s(I)} \lesa_g \|F \|_{N^s(I)}, $$
see \cite{Candy2021a}. This can be improved to a \emph{uniform} in $g$ estimate provided that $g$ has $L^2$ norm below the ground state. In fact provided $0<B<\|Q\|_{\dot{H}^1}$ we have for any $\|g\|_{L^2}\les B$ the uniform bound
        $$ \| \mc{I}_{v, t_0}[F] \|_{\underline{S}^s(I)} \lesa_B \| F \|_{N^s(I)}$$
where the implied constant depends only on $B$ (and $0\les s < 1$) \cite{Candy2021}. More formally, we have the following improvement to \eqref{eqn:energy ineq}.

\begin{theorem}[Strichartz with free wave potential {\cite[Theorem 6.1]{Candy2021}}]\label{thm:schro energy}
Let $0 \les  s <1$ and $0<B<\|Q\|_{\dot{H}^1}$. There is a constant $C = C(B)>0$ such that if $I\subset \RR$ is an interval, then for any $g\in L^2$ with $\| g \|_{L^2}\les B$ we have
    $$ \| u \|_{\underline{S}^s(I)} \les C\Big[ \inf_{t\in I} \| u(t) \|_{H^s} + \big\| \big( i\p_t + \Delta - \Re(e^{it|\nabla|}g) \big) u \big\|_{N^s(I)}\Big]. $$
\end{theorem}

\subsection{Standard Bilinear Estimates}

We frequently rely on the bilinear estimates contained in \cite{Candy2021a}. Although the estimates in \cite{Candy2021a} are valid in the full region of admissible regularities for the Zakharov equation (including the energy space $H^1\times L^2$), the general statement requires more complicated iteration spaces than those we require here. For our purposes, it suffices to work in the energy subcritical setting $H^s\times L^2$ with $\frac{1}{2}\les s < 1$. In this regime, the estimates in \cite{Candy2021a} can be stated as follows.

\begin{theorem}[Standard bilinear estimates \cite{Candy2021a}]\label{thm:bi prev}
Let $0\les s < 1$ and $0\les \ell < \frac{1}{2}$. Then for any interval $0\in I\subset \RR$ we have
        \begin{align*}
                 \| v u \|_{N^s(I)} &\lesa \| v\|_{W^0(I)} \| u \|_{S^s(I)}, \\
                 \big\| \mc{J}_0\big[ |\nabla|(\overline{u} w) \big] \big\|_{\underline{W}^\ell(I)} &\lesa \|u\|_{S^{\ell + \frac{1}{2}}(I)} \|w\|_{S^{\ell + \frac{1}{2}}(I)},
        \end{align*}
and the corresponding endpoint counterparts
        \begin{align*}
                 \| v u \|_{N^\frac{1}{2}(I)} &\lesa \| v \|_{W^0(I)} \Big( \| u \|_{D(I)} \| u \|_{S^\frac{1}{2}(I)}\Big)^\frac{1}{2},  \\
                 \big\| \mc{J}_0\big[ |\nabla|(\overline{u} w) \big] \big\|_{\underline{W}^0(I)} &\lesa \Big(\| u \|_{D(I)} \| w\|_{D(I)} \| u \|_{S^\frac{1}{2}(I)} \|w \|_{S^\frac{1}{2}(I)} \Big)^\frac{1}{2}.
        \end{align*}
\end{theorem}

The bilinear estimates in Theorem \ref{thm:bi prev} suffice for the small data global theory. However the proof of Theorem \ref{thm:min coun} requires a substantial improvement over these estimates. In particular, partially due to the fact that we don't have a good profile decomposition in the endpoint Strichartz spaces $L^2_t L^4_x$, we require a version of Theorem \ref{thm:bi prev} with a factor of the norm $\|\cdot\|_Y$ on the righthand side, see Theorem \ref{thm:bi est improved} for a precise statement.

\section{The Small Data Theory}\label{sec:small data}

In this section we collect together a number of consequences of the bilinear estimates contained in Theorem \ref{thm:bi prev}. In particular all results proved in this section can be seen as consequences of the estimates contained in \cite{Candy2021, Candy2021a}.

\subsection{Persistence of Regularity}

We start with a persistence of regularity type result which follows from the estimates in \cite{Candy2021, Candy2021a}. It closely related to  a more general persistence of regularity result \cite[Theorem 8.1]{Candy2021a} (which is valid for general regularities). The advantage here is that, via the uniform Strichartz estimate in Theorem \ref{thm:schro energy} we obtain a uniform bound in the iteration spaces $\underline{S}^s\times \underline{W}^0$ which is independent of the profile of the solution $(u, V)$. This uniformness is particularly useful later, as we frequently have good control over the size of $(u, V)$ in various norms, but not its profile.

Note that the energy space $H^1\times L^2$ is excluded from the result below. However this is not an issue, as for our purposes it suffices to work in $H^s\times L^2$ with $\frac{1}{2}\les s <1$. As explained in the introduction, once we have a global solution with $\|u\|_{D} < \infty$, scattering in the energy space $H^1\times L^2$ follows from Theorem \ref{thm:scat}.

\begin{theorem}[Persistence of regularity]\label{thm:persistence}
Let $\frac{1}{2}\les s<1$, $A>0$, and $0<B<\| Q \|_{\dot{H}^1}$. There exists a constant $C_{A,B}>0$ such that for any interval $I\subset \RR$, if $(u, V) \in C(I; H^\frac{1}{2}\times L^2)$ solves \eqref{eqn:Zak} on $I\times \RR^4$ with
        \begin{equation}\label{eqn:persistence:assump}
            \| u \|_{L^\infty_t H^\frac{1}{2}_x(I\times \RR^4)} + \| u \|_{D(I)} \les A, \qquad \| V \|_{L^\infty_t L^2_x(I\times \RR^4)} \les B,
        \end{equation}
then
        $$ \| u \|_{\underline{S}^s(I)} \les C_{A, B}  \inf_{t\in I} \| u(t) \|_{H^s}, \qquad \|V\|_{\underline{W}^0(I)} \les \inf_{t\in I} \| V(t)\|_{L^2} + C_{A,B}\inf_{t\in I} \|u(t)\|_{H^\frac{1}{2}}^2. $$
\end{theorem}
\begin{proof}
We start by noting that an application of Theorem \ref{thm:bi prev} gives for any $0\les s < 1$ and $ 0 \in I' \subset \RR$  the bounds
    \begin{equation}\label{eqn:persistence:bi schro}
         \| V u \|_{N^s(I')} \lesa \| V \|_{W^0(I')} \| u \|_{S^s(I')}
    \end{equation}
and
    \begin{equation}\label{eqn:persistence:bi wave}
            \big\| \mc{J}_0\big[ |\nabla| |u|^2 \big] \big\|_{\underline{W}^{0}(I')} \lesa \| u \|_{D(I')} \| u \|_{S^\frac{1}{2}(I')}.
    \end{equation}
Together with the uniform Strichartz estimate, Theorem \ref{thm:schro energy}, we conclude that provided
    $$\| V(0) \|_{L^2}\les B < \| Q \|_{\dot{H}^1}$$
we have for any $\frac{1}{2} \les s < 1$
    \begin{align*}
	\| u \|_{S^s(I')} &\lesa_B \| u(0) \|_{H^s} + \big\| [i\p_t + \Delta - \Re(e^{it|\nabla|}V(0))] u \big\|_{N^s(I)} \\
				&\lesa_B \|u(0)\|_{H^s} + \big\| \mc{J}_0\big[ |\nabla| |u|^2 \big] \big\|_{\underline{W}^{0}(I')} \| u \|_{S^s(I')}
				\lesa_B \| u(0) \|_{H^s} + \| u \|_{D(I')} \| u \|_{S^\frac{1}{2}(I')} \|u \|_{S^s(I')}.
	 \end{align*}
Time translation invariance and the assumption \eqref{eqn:persistence:assump}, together with the local well-posedness in $S^s\times W^0$ from \cite{Candy2021a} and a continuity argument (starting with the case $s=\frac{1}{2}$ and then $s>\frac{1}{2}$)  implies that there exists $\delta \ll_{A, B} 1$ such that for any interval $I'\subset I$ we have the implication
        \begin{equation}\label{eqn:persistence:assum small disp norm}
            \| u \|_{D(I')} \les \delta \qquad \Longrightarrow \qquad \| u \|_{S^s(I')} \lesa_B \inf_{t\in I'} \| u(t)\|_{H^s}.
        \end{equation}
To extend this bound to the whole interval $I\subset \RR$, we again exploit the assumption \eqref{eqn:persistence:assump}. More precisely, the upper bound on $\| u \|_{D(I)}$ implies that we can find a collection of $J\in \NN$ intervals $I_j \subset I$ and times $t_j \in I_j \cap I_{j+1}$ such that
    $$I=\cup_{j=1}^J I_j, \qquad  \| u \|_{D(I_j)} \les \delta, \qquad J\lesa \delta^{-2} A^2\approx_{A,B} 1, \qquad t_j \in I_j \cap I_{j+1}. $$
A short argument using the implication \eqref{eqn:persistence:assum small disp norm} then implies that
			\begin{equation}\label{eqn:persistence:sum}
				\sum_{j=1}^J \| u \|_{S^s(I_j)}\lesa_{A,B}\inf_{t\in I} \|u(t)\|_{H^s}.
            \end{equation}
Hence applying this bound with $s=\frac{1}{2}$ together with the energy inequality \eqref{eqn:energy ineq} and the bilinear estimate \eqref{eqn:persistence:bi wave} gives
            $$ \sup_j \| V \|_{W^0(I_j)} \lesa B + \sup_j \big\| \mc{J}_0\big[ |\nabla| |u|^2 \big] \big\|_{\underline{W}^{0}(I_j)} \lesa B + \sup_j \| u_j \|_{S^\frac{1}{2}(I_j)}^2\lesa_{A, B} 1. $$
Therefore, applying Lemma \ref{lem:duh decomposable}, the energy inequality \eqref{eqn:energy ineq}, and \eqref{eqn:persistence:bi schro} we conclude that
        \begin{align*}
          \| u \|_{S^s(I)} \lesa \inf_{t\in I} \|u(t)\|_{H^s} + \sum_{j=1}^J \| \Re(V) u \|_{N^s(I_j)} \lesa_{A, B} \Big( 1 + \sup_j \| V \|_{W^0(I_j)}\Big) \inf_{t\in I} \|u(t)\|_{H^s} \lesa_{A, B} \inf_{t\in I} \|u(t)\|_{H^s}.
        \end{align*}
To upgrade this bound to the space $\underline{S}^s(I)$, we simply plug in the equation once more and observe that
    $$ \| V \|_{\underline{W}^0(I)} \les \inf_{t\in I} \| V(t) \|_{L^2} + C \| u \|_{S^\frac{1}{2}(I)}^2 \les  \inf_{t\in I} \| V(t) \|_{L^2} + C_{A, B} \inf_{t\in I}  \| u(t) \|_{H^\frac{1}{2}}^2\lesa_{A, B} 1$$
and
    $$ \| u \|_{\underline{S}^s(I)} \lesa \inf_{t\in I} \| u(t) \|_{H^s} + \| \Re(V) u \|_{N^s(I)} \les C_{A, B} \inf_{t\in I} \|u(t) \|_{H^s}. $$
\end{proof}

\begin{remark}
By arguing as in Section 7 of \cite{Candy2021a}, at a cost of allowing the constant $C_{A, B}$ to depend on the profile of $V$ and not just its size, the assumption $\|V \|_{L^\infty_t L^2_x}\les B < \| Q \|_{\dot{H}^1}$ in Theorem \ref{thm:persistence} can be removed. This is a consequence of the dispersive properties of the free wave evolution. Moreover, although the energy scaling $s=1$ is excluded, a version of the above persistence of regularity result is true for all (admissible) regularities $H^s\times H^\ell$, see \cite{Candy2021a} for details.
\end{remark}

Theorem \ref{thm:persistence} implies that for the purpose of proving global well-posedness and scattering, the condition $\| u \|_{D}<\infty$ is equivalent to $\|u \|_{L^2_t B^s_{4,2}}<\infty$. In particular we could equally well have taken our dispersive norm to be the non-endpoint Besov variant $\| \cdot \|_{L^2_t B^s_{4,2}}$.

\subsection{Motivation for Controlling Quantity} \label{subsec:controlling quantity}
Suppose (for instance) we were trying to construct a fixed point to the (cubic) nonlinear problem
            $$ u = u_0 + \mc{I}[u^3]$$
where as usual $u_0$ is the free evolution of the data, and $\mc{I}$ denotes the corresponding Duhamel term. Assume we have a Banach space $X$ together with the nonlinear bound
        $$ \| \mc{I}[u v w]\|_X \lesa \|u\|_X \|v \|_X \| w \|_X. $$
One approach to construct the fixed point $u$ would be to consider the sequence $(u_k)_{k\in \NN} \subset X$ defined iteratively as usual via
    $$ u_k = u_0 + \mc{I}[ u_{k-1}^3]. $$
If we are working in a small data setting, we expect that the fixed point $u$ should be close to the free evolution $u_0$. This suggests that we should expand the nonlinearity into factors $u-u_0$. Together with our assumed nonlinear bound, we would then obtain
    \begin{align*}
         \| \mc{I}[ u^3]\|_X &= \| \mc{I}[ (u-u_0)^3 + 3 (u-u_0)^2 u_0 + 3 (u-u_0) u_0^2 + u_0^3] \|_X \\
                &\lesa \| u - u_0\|_X^3 + \| u - u_0\|_X^2 \|u_0\|_X + \| \mc{I}[ (u-u_0) u_0^2]\|_X + \| \mc{I}[u_0^3]\|_X.
    \end{align*}
and
       \begin{align*}
        \| \mc{I}[u^3 - v^3]\|_X &= \|\mc{I}[(u-v)^3 + 3(u-v)^2 v + 3 (u-v) v^2 ] \|_X \\
        &\lesa \| u - v\|_X\Big(\| u - v\|_X^2 + \| u - v\|_X \| v\|_X + \| v - u_0\|_X \|v\|_X +  \| v - u_0\|_X^2 \Big) + \| \mc{I}[ (u-v)u_0^2] \|_X.
       \end{align*}
In other words, if we define the controlling quantity
            $$\rho(v) = \sup_{\|\phi\|_X \les 1} \| \mc{I}[ \phi v^2] \|_X,$$
then provided $ \|u_0\|_X\les A$ we conclude that
    $$ \| u_{k+1} - u_0\|_X \lesa  \| u_{k} - u_0\|_X^3 +  A \| u_{k} - u_0\|_X^2    + \big( \| u_k - u_0\|_X + A\big) \rho(u_0)$$
and
    $$ \|u_{k+1} - u_{k}\|_X \lesa  \|u_k - u_{k-1}\|_X \Big(\|u_k - u_{k-1}\|_X^2  + (\|u_k - u_{k-1}\|_X + \|u_{k-1} - u_0\|_X) ( \| u_{k-1}-u_0\|_X + A ) + \rho(u_0)\Big). $$
Therefore, assuming that $\epsilon = \epsilon(A)>0$ is sufficiently small, under the smallness assumption
            $$ \|u_0\|_X \les A, \qquad \rho(u_0) \les \epsilon$$
we conclude the existence of a fixed point $u\in X$ to the problem $u = u_0 + \mc{I}[u^3]$. In other words, we do not require the data to be small in $X$, but only the nonlinear interaction of the free solution with a generic element of our iteration space $X$ must be small. Note that the nonlinear estimate immediately implies
        $$ \rho(u_0) \lesa \| u_0 \|_X^2$$
and so clearly it suffices to have smallness in $X$. On the other hand, typically in applications we try to prove an improved estimate of the form
        $$ \rho(u_0) \les  \| u_0 \|_Y^{\theta} \|u_0\|_X^{2-\theta}$$
where $Y$ is some intermediate norm, thus improving the assumption of smallness in $X$ to requiring just smallness in $Y$. The advantage of considering the quantity $\rho$ is that we only have to prove this estimate for the free solution $u_0$, as opposed to all elements of $X$ (which is what we would be forced to do if we added it into the iteration argument). Moreover the quantity $\rho$ is very flexible, as we can potentially formulate the improved estimate with a variety of choices of norms $Y$.

\subsection{A Stability Theorem}
Motivated by the discussion in the previous subsection, given an interval $I\subset \RR$ and $u\in C(I;H^\frac{1}{2})$, we define the controlling quantity
		\begin{equation}\label{eqn:rho defn}
	\rho_{I}(u) = \sup_{ \|v\|_{\underline{W}^0} \les 1} \big\| \mc{I}_{0}\big[ \ind_I \Re(v) u \big] \big\|_{S^\frac{1}{2}}  + \sup_{\|w\|_{\underline{S}^\frac{1}{2}}\les 1} \big\| \mc{J}_{0}\big[ \ind_I |\nabla| \Re( \overline{w} u) \big] \big\|_{W^0}.
		\end{equation}
In view of \eqref{eqn:duh loc time:init trans}, the initial time $t=0$ in the definition of $\rho_I$ could be replaced with any $t_0\in \RR$ at a cost of multiplying by some absolute constant. A quick application of Lemma \ref{lem:duh loc time} gives for any $t_0\in I$ the bounds
    $$ \| \mc{I}_{0, t_0}[ \Re(v) u ] \|_{S^\frac{1}{2}(I)} \lesa \rho_I(u) \| v \|_{\underline{W}^0(I)}, \qquad \| \mc{J}_{0, t_0}[ |\nabla|\Re(\overline{w}u)]\|_{W^0(I)} \lesa \rho_I(u) \| w \|_{\underline{S}^\frac{1}{2}(I)}. $$
Moreover, a short computation via Theorem \ref{thm:bi prev} and \eqref{eqn:energy ineq} gives
        $$ \rho_{I}(u) \lesa \|u\|_{S^\frac{1}{2}(I)}, \qquad \rho_{I}(u+u') \les \rho_{I}(u)  + \rho_I(u'). $$
Finally, we make the straightforward observation that if $I\subset J$, then by Lemma \ref{lem:duh loc time} we always have
        $$ \rho_I(u) \lesa \rho_J(u). $$

The small data scattering theory in \cite{Candy2021, Candy2021a} relied on knowing that the data was small in either $H^\frac{1}{2}\times L^2$ or the dispersive norm $\| \cdot \|_D$. Our goal is to extend this to only requiring that the free evolution of the data is small with respect to the controlling quantity $\rho_I$. Note that smallness of the controlling quantity is much more general than smallness in norm, for instance an application of Theorem \ref{thm:bi prev} gives
		\begin{equation}\label{eqn:rho bound I}
				\rho_I(e^{it\Delta} f) \lesa \Big(\|f \|_{H^\frac{1}{2}} \| e^{it\Delta} f \|_{D(I)}\Big)^\frac{1}{2} \lesa \| f\|_{H^\frac{1}{2}}.
		\end{equation}
On the other hand, as we shall see later, it is also possible to bound $\rho_I$ by the a factor of the  ``energy dispersion'' or $L^\infty_{t,x}$ type norm $\| \cdot\|_Y$. This additional flexibility provided by $\rho_I$ is extremely useful when we eventually come to studying compactness properties of a hypothetical counterexample to Conjecture \ref{conj:scattering}. In particular, we can avoid the technically challenging issue to trying to obtain a profile decomposition in the endpoint Strichartz space $L^2_t L^4_x$. \\

The construction of a minimal counter example relies on a crucial stability property of solutions to the Zakharov equation. Roughly speaking, we need to that if we have a solution $(\psi, \phi)\in C(I; H^\frac{1}{2}\times L^2)$ is a solution to the perturbed equation
   \begin{equation}\label{eqn:stab:approx sys}
        \begin{split}
      (i\p_t + \Delta)\psi &= \Re(\phi) \psi + F, \\
      (i\p_t + |\nabla|)\phi&= |\nabla| |\psi|^2 + G,
        \end{split}
    \end{equation}
which is dispersive, thus $\psi\in S^\frac{1}{2}(I)$, then under a suitable smallness assumption on $F$ and $G$, there exists a genuine solution $(u, V)$ to the Zakharov equation \eqref{eqn:Zak} which is also dispersive and `close' to $(\psi, \phi)$. A precisely formulation of this stability property is as follows.

\begin{theorem}[Stability]\label{thm:stab}
Let $A>0$, $0<B<\|Q\|_{\dot{H}^1}$, and $\frac{1}{2}\les s < 1$. There exists $\epsilon_0(A,B)>0$ such that if $0<\epsilon \les \epsilon_0(A, B)$,  $t_0\in I\subset \RR$ is an interval, and
    $$(f, g) \in H^s\times L^2, \qquad (F, G)\in N^\frac{1}{2}(I)\times M^0(I), \qquad (\psi, \phi) \in C(I, H^\frac{1}{2} \times L^2)$$
satisfies the boundedness condition
    \begin{equation}\label{eqn:stab:bound}
       \|f\|_{H^\frac{1}{2}} + \| \psi \|_{S^\frac{1}{2}(I)}\les A, \qquad \| \phi \|_{L^\infty_t L^2_x (I\times \RR^4)} \les B,
    \end{equation}
the smallness condition
      \begin{equation}\label{eqn:stab:error}
           \rho_{I}\big(e^{i(t-t_0)\Delta} [f-\psi(t_0)]\big) + \| g - \phi(t_0) \|_{L^2}  + \| \mc{I}_{0, t_0}[F] \|_{S^\frac{1}{2}(I)} + \| \mc{J}_{0, t_0}[G] \|_{W^0(I)} < \epsilon,
      \end{equation}
and $(\psi, \phi)$ solve the perturbed Zakharov equation \eqref{eqn:stab:approx sys} on $I\times \RR^4$, then there exists a (unique) solution $(u, V)\in C(I, H^s\times L^2)$ to \eqref{eqn:Zak} with data $(u, V)(t_0) = (f,g)$. Moreover, we have the bounds
    $$\rho_I(u-\psi) + \big\| u - \psi - e^{i(t-t_0)\Delta}[f-\psi(t_0)] \big\|_{S^\frac{1}{2}(I)} + \| V - \phi \|_{W^0(I)} \lesa_{A, B} \epsilon $$
and
    $$ \| u \|_{\underline{S}^s(I)} \lesa_{A, B, s}  \| f\|_{H^s}, \qquad \| V\|_{L^\infty_t L^2_x(I)} \les \tfrac{1}{2} \big( B + \|Q \|_{\dot{H}^1}\big) $$
where the implied constants depend only on $A$, $B$, and $s$.
\end{theorem}

\begin{proof} We begin by considering the case $s=\frac{1}{2}$ under the additional assumption that the dispersive norm $\| \psi\|_{D(I)}$ is sufficiently small. This additional smallness assumption is then removed by arguing as in the proof of Theorem \ref{thm:persistence}. Upgrading the regularity from $s=\frac{1}{2}$ to $\frac{1}{2}<s<1$ follows from a direct application of Theorem \ref{thm:persistence}. We now turn to the details. Suppose for the moment that the following claim holds:
\begin{quote}
  		\textbf{Claim:} Let $A>0$ and $0<B<\|Q\|_{\dot{H}^1}$. There exists $\epsilon^*(A,B), \delta^*(A, B)>0$ such that if
        $$0<\epsilon\les \epsilon^*(A, B) \qquad \text{ and } \qquad 0<\delta\les \delta^*(A, B),$$
        and $(\psi, \phi) \in C(I;H^\frac{1}{2}\times L^2)$ is a solution to \eqref{eqn:stab:approx sys} satisfying \eqref{eqn:stab:bound}, \eqref{eqn:stab:error}, and
    		\begin{equation}\label{eqn:stab:disp}
                \| \psi \|_{D(J)}<\delta,
    		\end{equation}
    for some $t_0\in J\subset I$, then there exists a solution $(u,V)\in C(J;H^\frac{1}{2}\times L^2)$ to the Zakharov system \eqref{eqn:Zak} such that $(u, V)(t_0) = (f,g)$ and we have
        $$ \big\|u -\psi - e^{i(t-t_0)\Delta}[f-\psi(t_0)]\big\|_{S^\frac{1}{2}(J)} \lesa_B (1+A)^2\epsilon$$
   and
         $$ \big\|V-\phi \big\|_{W^0(J\times \RR^4)}\lesa_B (1+A)^3\epsilon  $$
   where the implied constant depends only on $B$.
\end{quote}
To prove that the claim implies the required result, we argue as in the proof of Theorem \ref{thm:persistence} and decompose the initial interval $I$ into subintervals on which we have the required smallness \eqref{eqn:stab:disp}. Fix $A>0$,  $0<B<\|Q\|_{\dot{H}^1}$, and assume that we have $(\psi, \phi)\in C(I;H^\frac{1}{2}\times L^2)$ such that the conditions \eqref{eqn:stab:bound}, \eqref{eqn:stab:error},   and \eqref{eqn:stab:approx sys} hold where $\epsilon>0$ is to be determined later. Since $\|\psi\|_{D(I)}\les \| \psi \|_{S^\frac{1}{2}(I)} \les A$, we can find a collection of intervals $I_j \subset I$ such that
    $$I=\cup_{j=-N}^N I_j, \qquad  \| \psi \|_{D(I_j)} \les \delta, \qquad N\approx  \delta^{-2} A^2 \approx_{A, B} 1, \qquad I_j \cap I_{j+1} \not= \varnothing, \qquad t_0\in I_0$$
where $\delta = \delta^*(2A, B)>0$ is as in the statement of the above claim. To construct the solution $(u, V)$ forward in time, i.e. on the interval $\cup_{j\g 0} I_j$, we first choose times $t_{j+1}\in I_j \cap I_{j+1}$. Assuming we have a solution on the interval $I_j$, we note that
    \begin{align*}
       \rho_{I}\big(e^{i(t-t_{j+1})\Delta} [u-\psi](t_{j+1})\big) &- \rho_{I}\big(e^{i(t-t_{j})\Delta} [u-\psi](t_j)\big)\\
                &\lesa \rho_I\Big( e^{i(t-t_{j+1})\Delta}\big( [u-\psi](t_{j+1})- e^{i(t_{j+1}-t_j)\Delta}[u-\psi](t_j)\big)\Big)\\
                &\lesa \big\| [u-\psi](t_{j+1})- e^{i(t_{j+1}-t_j)\Delta}[u-\psi](t_j)\big\|_{H^\frac{1}{2}} \\
                &\lesa \big\| u - \psi - e^{i(t-t_j)\Delta}[u-\psi](t_j)\big\|_{S^\frac{1}{2}(I_j)}
    \end{align*}
and similarly
    \begin{align*}
      \| [u - \psi](t_{j+1}) \|_{H^\frac{1}{2}} - \| [u-\psi](t_j)\|_{H^\frac{1}{2}} &\les \big\| [u-\psi](t_{j+1}) - e^{i(t_{j+1} - t_j)\Delta}[u-\psi](t_j) \big\|_{H^\frac{1}{2}} \\
        &\lesa \big\| u - \psi - e^{i(t-t_j)\Delta}[u-\psi](t_j)\big\|_{S^\frac{1}{2}(I_j)}.
    \end{align*}
In view of \eqref{eqn:duh loc time:init trans} (and a similar version for the wave Duhamel term), for any $j$,
    $$ \| \mc{I}_{0, t_{j}}[F]\|_{S^\frac{1}{2}(I)} + \| \mc{J}_{0, t_j}[G]\|_{W^0(I)} \lesa \| \mc{I}_{0, t_{0}}[F]\|_{S^\frac{1}{2}(I)} + \| \mc{J}_{0, t_0}[G]\|_{W^0(I)} \lesa \epsilon.$$
Therefore, provided we have say
        $$ \epsilon \les \frac{1}{N+1} \big[ C_B ( 1 + 2A)^3\big]^{-(N+1)} \epsilon^*(2A, B)$$
where $C_B$ is a suitably chosen constant depending only on $B$ (via the above implied constants), repeatedly applying the above claim, together with the uniqueness of solutions in $D(I)$, gives a solution $(u, V) \in C(I; H^\frac{1}{2}\times L^2)$ such that for any $|j|\les N$ we have
	$$ \rho_I\big( e^{i(t-t_j)\Delta}[u-\psi](t_j) \big) + \big\| u-\psi - e^{i(t-t_j)\Delta}[u-\psi](t_j)\big\|_{S^\frac{1}{2}(I_j) } + \|V-\phi\|_{W^0(I_j)} \lesa_{A, B}\epsilon,$$
and
	$$ \|u\|_{S^{\frac{1}{2}}(I_j)} + \| V \|_{W^0(I_j)}\lesa_{A, B} 1.      $$
To extend these bounds to the whole interval $I$, we observe that an application of Lemma \ref{lem:duh decomposable} gives
	\begin{align*}
		&\big\| \mc{I}_{0, t_0}\big[ \Re(V) u - \Re(\phi)\psi \big]\big\|_{S^\frac{1}{2}(I)}\\
					&\lesa \sum_{j=-N}^N \Big(\big\| \mc{I}_{0, t_0}\big[\Re(V-\phi)\psi \big] \big\|_{S^\frac{1}{2}(I_j)}\\
					&\qquad  + \big\| \mc{I}_{0, t_0}\big[\Re(V)\big( u - \psi - e^{i(t-t_j)\Delta}[u-\psi](t_j) \big) \big] \big\|_{S^\frac{1}{2}(I_j)} + \big\| \mc{I}_{0, t_0}\big[\Re(V) e^{i(t-t_j)\Delta}[u-\psi](t_j) \big] \big\|_{S^\frac{1}{2}(I_j)} \Big)\\
					&\lesa \sum_{j=-N}^N \Big[ \|V-\phi\|_{W^0(I_j)} \|\psi \|_{S^\frac{1}{2}(I)} \\
					&\qquad  + \|V\|_{\underline{W}^0(I_j)}\Big( \| u - \psi - e^{i(t-t_j)\Delta}[u-\psi](t_j)\|_{S^\frac{1}{2}(I_j)} + \rho_I\big( e^{i(t-t_j)\Delta} [u-\psi](t_j)\big) \Big)\Big] \\
					&\lesa_{A,B} \epsilon
	\end{align*}
where we used the bound $N\lesa_{A, B} 1$ together with
		$$ \|V \|_{\underline{W}^0(I_j)} \lesa \| V(t_j)\|_{L^2_x} + \| u \|_{S^\frac{1}{2}(I_j)}^2 \lesa_{A, B} 1.  $$
Therefore we obtain the global smallness bound
	\begin{align*}
		\big\| u - \psi - e^{i(t-t_0)}[f-\psi(t_0)]\big\|_{S^\frac{1}{2}(I)}&= \big\| \mc{I}_{0, t_0}\big[\Re(V) u - \Re(\phi)\psi - F \big] \big\|_{S^\frac{1}{2}(I)}\lesa_{A, B} \epsilon. 	\end{align*}
After observing that
		\begin{align*}
     \rho_I(u-\psi) &\les \rho_I\big( u - \psi - e^{i(t-t_0)\Delta}[u-\psi](t_0)\big) + \rho_I\big( e^{i(t-t_0)\Delta}[u-\psi](t_0) \big) \\
     &\lesa \| u - \psi - e^{i(t-t_0)\Delta}[u-\psi](t_0)\big\|_{S^\frac{1}{2}(I)} + \epsilon
     \lesa_{A, B}\epsilon
        \end{align*}
we have the required smallness bounds on the Schr\"odinger evolution. Moreover, an application of Theorem \ref{thm:persistence} gives for any $\frac{1}{2}\les s < 1$ the higher regularity bounds
	$$  \| u \|_{\underline{S}^s(I)} \lesa_{A, B, s} \| f \|_{H^s}. $$
To bound the wave evolution we simply observe that since
		$$ |u|^2 - |\psi|^2 = \Re\big[ (\overline{u} + \overline{\psi})(u-\psi)\big] =  \Re\big[ (\overline{u} + \overline{\psi} - \mc{I}_{0, t_0}[F])(u-\psi) + \mc{I}_{0, t_0}[F](u-\psi)\big] $$
and
		$$ \| \psi - \mc{I}_{0, t_0}[F]\|_{\underline{S}^\frac{1}{2}(I)} \lesa \| \phi \|_{W^0(I)} \| \psi\|_{S^\frac{1}{2}(I)} \lesa \big(\| \phi(t_0)\|_{L^2} + \| \psi\|_{S^\frac{1}{2}(I)}^2\big) \| \psi\|_{S^\frac{1}{2}(I)} \lesa_{A, B} 1$$
we conclude that
	\begin{align*}
		\| V-\phi \|_{W^0(I)} &\les \|g-\phi(t_0)\|_{L^2} + \big\| \mc{J}_{0, t_0}\big[|u|^2-|\psi|^2 - G\big] \big\|_{W^0(I)} \\
									&\lesa \epsilon + \big\| \mc{J}_{0, t_0}\big[\Re( (u+\psi - \mc{I}_{0, t_0}[F])(u-\psi)\big)\big]\big\|_{W^0(I)} + \big\| \mc{J}_{0, t_0}\big[ \Re\big( \mc{I}_{0, t_0}[F] (u-\psi)\big)\big] \big\|_{W^0(I)} \\
									&\lesa \epsilon + \big( \|u \|_{\underline{S}^\frac{1}{2}(I)} + \| \psi - \mc{I}_{0, t_0}[F]\|_{S^\frac{1}{2}(I)} \big) \rho_I(u-\psi) + \| \mc{I}_{0, t_0}[F] \|_{S^\frac{1}{2}(I)} \big( \| \psi \|_{S^\frac{1}{2}(I)} + \| u \|_{S^\frac{1}{2}(I)}\big) \\
									&\lesa_{A, B} \epsilon.
	\end{align*}
Therefore theorem follows from the claimed local version under the additional constraint \eqref{eqn:stab:disp}.

It remains to verify the claim. By translation invariance, we may fix $t_0 =0$. Let $\phi_0 = e^{it|\nabla|} \phi(0)$. Our goal is to construct a solution to the problem
    \begin{equation}\label{eqn:stab:u' eqn}
        \begin{split}
         \big(i\p_t + \Delta - \Re(\phi_0)\big) u' &= \Re(V'+\phi) (u'+\psi) - \Re(\phi) \psi - \Re(\phi_0) u' - F, \\
         (i\p_t + |\nabla|)V' &= |\nabla|\big( |u' + \psi|^2 - |\psi|^2\big) - G,
         \end{split}
    \end{equation}
with data
    $$ (u', V')(0) = (f', g') := \big(f - \psi(0), g - \phi(0) \big). $$
Provided we obtain appropriate bounds on $(u', V')$, the claim then follows by taking $(u, V)=(\psi + u', \phi+V')$. We begin by observing that an application of Theorem \ref{thm:bi prev}  together with \eqref{eqn:stab:disp} and \eqref{eqn:stab:bound} gives
    \begin{equation}\label{eqn:stab:phi bound}
        \| \phi - \phi_0\|_{W^0(J)} = \big\| \mc{J}_0[ |\nabla| |\psi|^2 + G] \big\|_{W^0(J)} \lesa \| \psi \|_{D(J)} \| \psi \|_{S^\frac{1}{2}(J)} + \| \mc{J}_0[G] \|_{W^0(J)} \lesa \delta A + \epsilon
    \end{equation}
and (assuming $\epsilon \ll 1$ and $\delta A \ll 1$)
    \begin{align}\label{eqn:stab:strong psi bound}
        \| \psi - \mc{I}_0[F] \|_{\underline{S}^\frac{1}{2}(J)}
                    \lesa \| \psi(0) \|_{H^\frac{1}{2}} + \| \mc{I}_0[\Re(\phi) \psi] \|_{\underline{S}^\frac{1}{2}(J)}
                    \lesa A + \| \phi \|_{W^0(J)} \| \psi \|_{S^\frac{1}{2}(J)} \lesa A
    \end{align}
Define $u_0' = e^{it\Delta} f'$ and for any $k\in \NN$ consider the sequence
            \begin{align*}
              u_{k+1}' &= \mc{U}_{\Re(\phi_0)}(t) f' + \mc{I}_{\Re(\phi_0)}\big[ \Re(V'_{k+1} + \phi) ( u_k' + \psi) - \Re(\phi) \psi - \Re(\phi_0) u_k' - F\big] \\
              V_{k+1}' &= e^{it|\nabla|} g' + \mc{J}_0\big[ |\nabla|\big( |u_k' + \psi|^2 - |\psi|^2\big) - G \big].
            \end{align*}
We make the inductive assumption that
    \begin{equation}\label{eqn:stab:induc}
            \| u'_k - u_0' \|_{S^\frac{1}{2}(J)} \les A.
    \end{equation}
This bound clearly holds when $k=0$, and the estimates below substantially improve on this weak upper bound, however it simplifies the notation somewhat. Write
    \begin{align*}
        |u'_k + \psi|^2 - |\psi|^2 = |u_k' - u_0'|^2 + 2 \Re[ \overline{u_0'}(u'_k - u_0' + \mc{I}_0[F])]+2 \Re[ \overline{u_0'}(u_0' + \psi -  \mc{I}_0[F])] + 2 \Re[ \overline{\psi} (u_k' - u_0')].
    \end{align*}
Then the definition of $\rho_I(u_0')$, Theorem \ref{thm:schro energy}, and the bounds \eqref{eqn:stab:bound}, \eqref{eqn:stab:error}, \eqref{eqn:stab:disp}, and \eqref{eqn:stab:strong psi bound} give
    \begin{align*}
      \| V_{k+1}' \|_{W^0(J)} &\lesa \|g'\|_{L^2} + \big\| |\nabla| \mc{J}_0\big[ |u_k' - u_0'|^2\big] \big\|_{W^0(J)} + \big\| |\nabla|\mc{J}_0\big[\overline{u_0'} (u_0' + \psi - \mc{I}_0[F])\big]\big\|_{W^0(J)} \\
      &\qquad  \qquad \qquad \qquad \qquad + \big\| |\nabla| \mc{J}_0\big[ \Re\big( \overline{u_0'}(u_0' + \psi -  \mc{I}_0[F])\big)\big] \big\|_{W^0(J)} + \big\| |\nabla| \mc{J}_0[\overline{\psi} (u_k' - u_0')]\big\|_{W^0(J)}\\
      &\lesa \epsilon + \|u_k' - u_0'\|_{S^\frac{1}{2}(J)}^2 + \|u_0'\|_{S^\frac{1}{2}(J)}\big( \|u_k' - u_0'\|_{S^\frac{1}{2}(J)} + \| \mc{I}_0[F]\|_{S^\frac{1}{2}(J)}\big)\\
      &\qquad \qquad \qquad \qquad \qquad + \rho_I(u_0')\big( \| u_0'\|_{\underline{S}^\frac{1}{2}(J)} + \| \psi - \mc{I}_0[F] \|_{\underline{S}^\frac{1}{2}(J)}\big) + \| \psi \|_{S^\frac{1}{2}(J)} \| u_k' - u_0'\|_{S^\frac{1}{2}(J)}\\
      &\lesa (1+A)\epsilon + A \|u_k' - u_0'\|_{S^\frac{1}{2}(J)}
    \end{align*}
where we also applied the inductive assumption \eqref{eqn:stab:induc}. Similarly, for the differences $V_{k+1}' - V_{j+1}'$, after writing
        $$ V_{k+1}' - V_{j+1}' =  |\nabla| \mc{J}_0\big[ |u_k' + \psi|^2 - |u_j' + \psi|^2\big] = |\nabla| \mc{J}_0\big[ |u_k' - u_j'|^2 + 2\Re\big( (\overline{u}_k' - \overline{u}_j')(u_j' + \psi)\big) \big]$$
we can bound the difference in the stronger space $\underline{W}^0(J)$ via
    \begin{align*}
        \| V_{k+1}' - V_{j+1}' \|_{\underline{W}^0(J)}
            &\lesa \| u_j' - u_k'\|_{S^\frac{1}{2}(J)}^2 + \|u_j' -  u_k' \|_{S^\frac{1}{2}(J)}\big( \|u_j' - u_0'\|_{S^\frac{1}{2}(J)} + \| u_0' \|_{S^\frac{1}{2}(J)} + \| \psi \|_{S^\frac{1}{2}(J)}\big) \\
            &\lesa A \|u_j' -  u_k' \|_{S^\frac{1}{2}(J)}.
    \end{align*}
The fact that we don't assume a good bound for $\mc{J}_0[G]$ in the stronger space $\underline{W}^0(J)$ prevents us from improving the above bound for $V_k'$ from $W^0(J)$ to $\underline{W}^0(J)$. However this is possible for components of the wave evolution. More precisely, we have (again using the inductive assumption \eqref{eqn:stab:induc})
    \begin{align*}
      \| V_{k+1}' + \phi - \phi_0\|_{\underline{W}^0(J)} &= \big\| e^{it|\nabla|}g' + \mc{J}_0\big[|\nabla| |u_k' + \psi|^2\big] \big\|_{\underline{W}^0(J)} \lesa \epsilon + \|u'_k + \psi \|_{S^\frac{1}{2}(J)}^2 \lesa A^2.
    \end{align*}
We now turn to the bounds for the Schr\"odinger component of the evolution. The first step is to observe that the identities
    $$\mc{I}_{\Re(\phi_0)}[\Phi] = \mc{I}_0[\Phi] + \mc{I}_{\Re(\phi_0)}[\Re(\phi_0) \mc{I}_0[\Phi]]$$
and
    $$ \mc{U}_{\Re(\phi_0)}f = e^{it\Delta} f +  \mc{I}_{\Re(\phi_0)}[ \Re(\phi_0) e^{it\Delta} f] = e^{it\Delta} f + \mc{I}_0[\Re(\phi_0) e^{it\Delta} f] + \mc{I}_{\Re(\phi_0)}\big[ \Re(\phi_0) \mc{I}_0[\Re(\phi_0) e^{it\Delta} f] \big]$$
together with the uniform energy estimate in Theorem \ref{thm:schro energy} and the bilinear estimates in Theorem \ref{thm:bi prev} gives (since $\| \phi(0)\|_{L^2} \les B < \|Q\|_{\dot{H}^1}$)
    $$ \| \mc{I}_{\Re(\phi_0)}[\Phi] \|_{S^\frac{1}{2}(J)} \lesa_B \| \mc{I}_0[\Phi]\|_{S^\frac{1}{2}(J)} + \| \phi_0 \mc{I}_0[\Phi] \|_{N^\frac{1}{2}(J)} \lesa_B \| \mc{I}_0[\Phi]\|_{S^\frac{1}{2}(J)}$$
and
    \begin{align*}
    \| \mc{U}_{\Re(\phi_0)}f' - u_0'\|_{S^\frac{1}{2}(J)} &\lesa_B \| \mc{I}_0[\Re(\phi_0) u_0'\|_{S^\frac{1}{2}(J)} + \| \Re(\phi_0) \mc{I}_0[\Re(\phi_0) u_0'] \|_{N^\frac{1}{2}(J)} \\
                &\lesa_B \| \mc{I}_0[\Re(\phi_0) u_0'] \|_{S^\frac{1}{2}(J)} \lesa_B \| \phi_0 \|_{\underline{W}^0(J)} \rho_I(u_0') \lesa_B \epsilon.
    \end{align*}
Hence after decomposing
    $$ \Re(V'_{k+1} + \phi) (u_k' + \psi) - \Re(\phi) \psi - \phi_0 u_k' = \Re( V_{k+1}' + \phi - \phi_0) (u_k' - u_0') +  \Re( V_{k+1}' + \phi - \phi_0) u_0' + \Re(V'_{k+1}) \psi$$
we see that
    \begin{align*}
      \|u'_{k+1} - u_0'\|_{S^\frac{1}{2}(J)}
       &\les \| \mc{U}_{\Re(\phi_0)} f' - u_0' \|_{S^\frac{1}{2}(J)}
            + \big\| \mc{I}_{\Re(\Phi_0)}\big[ \Re(V'_{k+1} + \phi) (u_k' + \psi) - \Re(\phi) \psi - \phi_0 u_k' \big] \big\|_{S^\frac{1}{2}(J)} \\
      &\lesa_B \epsilon + \big\| \mc{I}_0\big[\Re( V_{k+1}' + \phi - \phi_0) (u_k' - u_0') +  \Re( V_{k+1}' + \phi - \phi_0) u_0' + \Re(V'_{k+1}) \psi \big] \big\|_{S^\frac{1}{2}(J)} \\
      &\lesa_B \epsilon + \|V_{k+1}' + \phi - \phi_0\|_{W^0(J)} \|u_k' - u_0'\|_{S^\frac{1}{2}(J)} \\
      &\qquad \qquad\qquad + \|V_{k+1}' + \phi - \phi_0 \|_{\underline{W}^0(J)} \rho_I(u_0')  + \|V'_{k+1} \|_{W^0(J)} \big( \|\psi\|_{D(J)} \| \psi \|_{S^\frac{1}{2}(J)}\big)^\frac{1}{2} \\
      &\lesa_B \epsilon + \big[ (1+A)\epsilon + \delta A +  A \|u_k' - u_0'\|_{S^\frac{1}{2}(J)} \big] \|u_k' - u_0'\|_{S^\frac{1}{2}(J)} \\
      &\qquad \qquad\qquad + A^2 \epsilon + \big[(1+A)\epsilon + A \|u_k' - u_0'\|_{S^\frac{1}{2}(J)}\big] (\delta A)^\frac{1}{2} \\
      &\lesa_B ( 1 + A^2)\epsilon + \big[ (1+A)\epsilon + A^\frac{3}{2} \delta^\frac{1}{2} + A \|u_k'-u_0'\|_{S^\frac{1}{2}(J)}\big] \|u_k' - u_0\|_{S^\frac{1}{2}(J)}.
    \end{align*}
Similarly, for differences, we write
    $$ \Re(V_{k+1}' + \phi)(u_k' + \psi) - \Re(\phi_0) u_k' - \Re(V_{j+1}' + \phi)(u_j' + \psi) + \Re(\phi_0) u_j' = \Re(V_{k+1}' - V_{j+1}')(u_k' + \psi) + \Re(V_{j+1}' + \phi - \phi_0)(u_k' - u_j')$$
and observe that the above bounds give
    \begin{align*}
      \|u_{k+1}' - u_{j+1}'\|_{S^\frac{1}{2}(J)}
        &= \big\| \mc{I}_{\Re(\phi_0)}\big[ \Re(V_{k+1}' - V_{j+1}')(u_k' + \psi) + \Re(V_{j+1}' + \phi - \phi_0)(u_k' - u_j')\big] \big\|_{S^\frac{1}{2}(J)} \\
        &\lesa_B \big\| \mc{I}_0\big[ \Re(V_{k+1}' - V_{j+1}')(u_k' +\psi)\big] \big\|_{S^\frac{1}{2}(J)} + \big\| \mc{I}_0\big[(V'_{j+1} + \phi - \phi_0)(u_k' - u_j') \big] \big\|_{S^\frac{1}{2}(J)}\\
        &\lesa_B \|V'_{k+1} - V_{j+1}'\|_{\underline{W}^0(J)}\Big( \|u_k' - u_0'\|_{S^\frac{1}{2}(J)} + \rho_I(u_0') + \big( \| \psi \|_{S^\frac{1}{2}(J)} \| \psi \|_{D(J)}\big)^\frac{1}{2} \Big) \\
        &\qquad\qquad\qquad\qquad + \big( \| V_{j+1}' \|_{W^0(J)} + \| \phi - \phi_0\|_{W^0(J)}\big) \|u_k' - u_j'\|_{S^\frac{1}{2}(J)}\\
        &\lesa_B A  \| u_k' - u_j'\|_{S^\frac{1}{2}(J)} \big( \| u'_k - u_0'\|_{S^\frac{1}{2}(J)} + \epsilon + (A\delta)^\frac{1}{2}\big) \\
         &\qquad\qquad\qquad\qquad+ \big( (1+A)\epsilon  +  \delta A + A \| u_j'-u_0'\|_{S^\frac{1}{2}(J)} \big) \|u_k' - u_j'\|_{S^\frac{1}{2}(J)}\\
        &\lesa_B (1+A) \big( \epsilon + (A\delta)^\frac{1}{2} + \|u_j' - u_0'\|_{S^\frac{1}{2}(J)} + \| u_k' - u_0'\|_{S^\frac{1}{2}(J)} \big) \|u_j' - u_k' \|_{S^\frac{1}{2}(J)}.
    \end{align*}
Therefore, a standard argument shows that there exists $\epsilon^*(A, B), \delta^*(A, B)>0$ such that if $0<\epsilon < \epsilon^*(A, B)$ and $0< \delta < \delta^*(A, B)$ then there exists a solution $(u', V')$ to \eqref{eqn:stab:u' eqn} such that
        $$ \|u' - u_0'\|_{S^\frac{1}{2}(J)} \lesa_B (1 + A^2) \epsilon, \qquad \|V' - e^{it\Delta} f'\|_{W^0(J)} \lesa_B (1+A^3) \epsilon. $$
Letting $(u, V) = (\psi + u', \phi + V')$, and noting that for any $s\in J$ we have
    \begin{align*}
        \rho_I\big( e^{i(t-s)\Delta}[u(s) - \psi(s)]\big) &\les \rho_I\big( e^{it\Delta}[f-\psi(0)]\big) + \rho_I\Big( e^{i(t-s)\Delta}[u(s)  -\psi(s) - e^{is\Delta}[f-\psi(0)]\big]\Big)\\
                                                &\lesa \epsilon + \big\| u(s)  -\psi(s) - e^{is\Delta}[f-\psi(0)]\big\|_{H^\frac{1}{2}} \lesa_B (1+A^2)\epsilon
    \end{align*}
the claim now follows.
\end{proof}

A quick corollary of the previous theorem is an existence theorem for the Zakharov equation where the time of existence depends only on the size of the controlling quantity $\rho_I$.

\begin{corollary}[Refined small data theory]\label{cor:lwp with rho}
Let $A>0$, $0<B<\|Q\|_{\dot{H}^1}$, and $\frac{1}{2}\les s < 1$. There exists $\epsilon_0(A,B)>0$ such that if $0<\epsilon \les \epsilon_0(A, B)$,  $I\subset \RR$ is an interval, and $(f, g) \in H^s\times L^2$
such that for some $t_0\in I$ we have the boundedness condition
    \begin{equation}\label{eqn:lwp:bound}
       \|f \|_{H^\frac{1}{2}} \les A, \qquad \| g \|_{L^2_x} \les B,
    \end{equation}
and the smallness condition
      \begin{equation}\label{eqn:lwp:error}
           \rho_{I}\big(e^{i(t-t_0)\Delta} f \big) < \epsilon,
      \end{equation}
then there exists a (unique) solution $(u, V)\in C(I, H^s\times L^2)$ to \eqref{eqn:Zak} with data $(u, V)(0) = (f,g)$. Moreover, we have the bounds
    $$\rho_I(u) + \big\| u -  e^{i(t-t_0)\Delta}f \big\|_{S^\frac{1}{2}(I)} + \| V - e^{i(t-t_0)|\nabla| }g\|_{W^0(I)} \lesa_{A, B} \epsilon $$
and
    $$ \| u \|_{\underline{S}^s(I)} \lesa_{A, B, s}  \| f\|_{H^s}, \qquad \| V\|_{L^\infty_t L^2_x(I)} \les \tfrac{1}{2} \big( B + \|Q \|_{\dot{H}^1}\big) $$
where the implied constants depend only on $A$, $B$, and $s$.
\end{corollary}
\begin{proof}
We simply apply Theorem \ref{thm:stab} with $\psi = F = G = 0$ and $\phi = e^{it|\nabla|} g$.
\end{proof}

\subsection{Existence of Wave Operators}

To extract some compactness from a bounded sequence of solutions to the Zakharov equation, we require solutions with prescribed data at infinity (i.e the existence of wave operators). This is a simple consequence of the results in \cite{Candy2021, Candy2021a}.

\begin{theorem}[Existence of wave operators below ground state]\label{thm:wave operators}
Let $(f, g)\in H^1\times L^2$ with
        $$ 2 \| f\|_{\dot{H}^1}^2 + \|g \|_{L^2}^2 < \| Q \|_{\dot{H}^1}^2. $$
Then there exists a unique global solution $(u,V)\in C(\RR, H^1\times L^2)$  to \eqref{eqn:Zak} such that $u\in L^2_{t,loc} W^{\frac{1}{2}, 4}_x$ and
        $$ \lim_{t\to -\infty} \big\| (u,V)(t) - \big( e^{it\Delta} f, e^{it|\nabla|} g\big) \big\|_{H^1\times L^2} = 0. $$
Moreover, the energy $\mc{E}_Z(u,V)$ is conserved, and
        $$ 4\mc{E}_Z(u, V) = \lim_{t\to -\infty}  4 \mc{E}_Z\big(e^{it\Delta} f, e^{it|\nabla|} g\big)= 2 \| f\|_{\dot{H}^1}^2 +  \|g \|_{L^2}^2, \qquad  \| V  \|_{L^\infty_t L^2_x}^2 \les 2 \| f\|_{\dot{H}^1}^2 +  \|g \|_{L^2}^2. $$
\end{theorem}
\begin{proof}
The estimates in \cite{Candy2021a} (see Theorem \ref{thm:bi prev} and the energy inequality \eqref{eqn:energy ineq}) together with a standard fixed point argument give a time $T\in \RR$ and a unique solution $(u,V)\in C((-\infty, T], H^1\times L^2)$ to \eqref{eqn:Zak} with $u\in L^2_tW^{\frac{1}{2}, 4}_x((-\infty, T]\times \RR^4)$ and
        $$ \lim_{t\to -\infty} \big\| (u,V)(t) - \big( e^{it\Delta} f, e^{it|\nabla|} g\big) \big\|_{H^1\times L^2} = 0. $$
Moreover this solution conserves energy, and hence as
        \begin{align*}
            \big|\mc{E}_Z(u,v) &- \mc{E}_Z\big(e^{it\Delta}f, e^{it|\nabla|}g\big) \big| \\
                &\lesa \big( \|u(t)\|_{H^1} + \|V(t)\|_{L^2} + \| f \|_{H^1} + \| g\|_{L^2} \big)\big( \|u(t) - e^{it\Delta}f \|_{H^1} + \|V(t) - e^{it|\nabla|}g\|_{L^2}\big)
        \end{align*}
we have
        $$ \mc{E}_Z(u, V) = \lim_{t\to -\infty} \mc{E}_Z\big( e^{it\Delta}f, e^{it|\nabla|}g\big) = \frac{1}{2} \| f \|_{H^1}^2 + \frac{1}{4}\|g\|_{L^2}^2$$
where the last identity follows from the fact that, due to the dispersive properties of the Schr\"odinger evolution, we have $\lim_{t\to -\infty} \| e^{it\Delta} f \|_{L^4_x}=0$. Therefore in view of the assumptions on the asymptotic data $(f,g)$ we have
    $$ 4\mc{E}_Z(u, V) < \| Q\|_{\dot{H}^1}^2, \qquad \lim_{t\to -\infty} \|V(t) \|_{L^2}  = \| g \|_{L^2} < \| Q \|_{\dot{H}^1}. $$
Consequently we can find some $T^*\in (-\infty, T]$ such that
    $$ 4\mc{E}_Z\big(u(T^*), V(T^*)\big) < \| Q\|_{\dot{H}^1}^2, \qquad \|V(T^*) \|_{L^2} < \| Q \|_{\dot{H}^1} $$
and hence taking $(u,V)(T^*)$ as data, Theorem \ref{thm:gwp below ground state} extends the solution $(u,V)$ from $(-\infty, T]$ to $\RR$ with $(u, V)\in C(\RR, H^1\times L^2)$.
\end{proof}

\section{Bilinear Estimates}\label{sec:bil est}

This section contains the key bilinear estimates that form the foundation of the proof of our main results. Our main goal is to prove the nonlinear terms on the righthand side of \eqref{eqn:Zak} are perturbative provided that the $Y$ norm is small. As mentioned in the introduction, the fundamental issue is that in the fully resonant case (so both inputs and outputs have Fourier support close to their relevant characteristic surfaces) it seems very hard to improve on the Strichartz bounds
        \begin{equation}\label{eqn:stri example}
             \| \mc{I}_0[\Re(V)u]\|_{L^2_t L^4_x} \lesa \| \Re(V) u \|_{L^2_t L^\frac{4}{3}_x} \lesa \| V \|_{L^\infty_t L^2_x} \| u \|_{L^2_t L^4_x}.
        \end{equation}
In particular, it does not seem possible to extract a factor $\|u\|_{L^\infty_{t,x}}$. Instead, following the approach in \cite{Candy2021}, we require a version of the bilinear restriction estimates for the paraboloid. More precisely, we recall that provided $r>\frac{5}{3}$ and $\mu\in 2^\ZZ$ we have the (homogeneous) bilinear restriction estimate for the paraboloid
		\begin{equation}\label{eqn:hom bi restric}
				\big\| \dot{P}_\mu( \overline{e^{it\Delta} f} e^{it\Delta} g) \big\|_{L^1_t L^r_x} \lesa  \mu^{2-\frac{4}{r}} \|f\|_{L^2} \|g\|_{L^2}.
		\end{equation}
Note that the Stricharz estimates only give the case $r=2$. The key advantage of \eqref{eqn:hom bi restric} is that we can take $r<2$ and thus gain some additional room over simply applying H\"older's inequality followed by the linear Strichartz estimate. The range $r>\frac{5}{3}$ is sharp, in the sense that \eqref{eqn:hom bi restric} fails if $r<\frac{5}{3}$. The endpoint $r=\frac{5}{3}$ is open. The estimate \eqref{eqn:hom bi restric} was obtained in \cite{Candy2019a}. Bilinear restriction estimates for the paraboloid in $L^p_{t,x}$ up to the sharp region were obtain by Tao \cite{Tao2003a}, the mixed norm case $L^q_t L^r_x$ with $q>1$ was then considered in \cite{Lee2008}.

As it stands, the estimate \eqref{eqn:hom bi restric} is not particularly useful in the nonlinear problem  as the restriction to homogeneous solutions is too strong. However recently stronger estimates have been developed that improve on \eqref{eqn:hom bi restric} by allowing general functions belonging to suitable adapted function spaces \cite{Candy2019a}. Of particular importance for our purposes are the \emph{inhomogeneous} bilinear restriction estimates proved in \cite{Candy2021}, see Theorem \ref{thm:bi schro inhom} below. In fact, by adapting the arguments from \cite{Candy2021} we obtain the following crucial bilinear estimate.

\begin{theorem}[Improved Bilinear Estimates]\label{thm:bi est improved}
Let $s>\frac{1}{2}$. There exists $C>0$ and $0<\theta<1$ such that for any interval $0\in I\subset \RR$ we have
	$$  \big\| \mc{I}_0\big[ \Re(V) u\big] \big\|_{S^\frac{1}{2}(I)} \les C \Big( \|V\|_{Y(I)} \|u\|_{Y(I)} \Big)^\theta \Big(\|V \|_{\underline{W}^0(I)} \|u\|_{\underline{S}^s(I)}\Big)^{1-\theta}$$
and
	$$ \big\| \mc{J}_0\big[  |\nabla|\Re(\overline{w}u)\big] \big\|_{W^0(I)}\les C \Big( \|w\|_{Y(I)} \|u\|_{Y(I)} \Big)^\theta \Big(\|w \|_{\underline{S}^\frac{1}{2}(I)} \|u\|_{\underline{S}^s(I)}\Big)^{1-\theta}. $$
\end{theorem}

We give the proof of Theorem \ref{thm:bi est improved} at the end of this section, after first recalling the inhomogeneous bilinear restriction estimates from \cite{Candy2021}, and proving frequency localised versions of the bilinear estimates in Theorem \ref{thm:bi est improved}.

It is worth noting that there is a loss of (Schr\"odinger) regularity in Theorem \ref{thm:bi est improved}. In particular, we require $u\in \underline{S}^s$ with $s>\frac{1}{2}$, but can only place the Duhamel term $\mc{I}_0[\Re(V) u ]$ into $S^\frac{1}{2}$. This reflects the fact that obtaining a factor of $\| u \|_{L^\infty_{t,x}}$ is surprisingly challenging, and even after using the bilinear restriction estimates above, we are essentially forced to place $u\in L^a_t L^b_x$ for some $a>2$ (or at least some space which has temporal component of the form $L^a_t$). In turn, this means we have to control $V$ in $L^q_t$ for some $q<\infty$. As we see below, this is possible, but there is a derivative loss. It is unclear if this can be avoided. In any case, as our eventual goal is to study solutions with regularity $H^1\times L^1$, allowing some loss in the Schr\"odinger regularity is not a major obstruction. However, it does force some technical complications, which are mostly alleviated by considering smallness in terms of the controlling quantity $\rho$ (as in Section \ref{sec:small data}), as opposed to trying to propagate smallness in $Y(I)$ via Theorem \ref{thm:bi est improved} directly in an iteration argument.

\subsection{Bilinear Restriction Estimates}

The proof of global existence for the Zakharov equation relied crucially on the following bilinear restriction type estimates for the paraboloid and cone. The first estimate gives a useful bilinear estimate in $L^2_{t,x}$ in the case where $u$ can be an \emph{inhomogeneous} solution to the Schr\"odinger equation. Recall that we have defined
		$$ \| u \|_Z = \| u \|_{L^\infty_t L^2_x} + \| ( i\p_t + \Delta) u \|_{L^1_t L^2_x + L^2_x L^4_x}.$$

\begin{theorem}[Bilinear $L^2_{t,x}$ for inhomogeneous wave-Schr\"odinger interactions {\cite[Theorem 4.1]{Candy2021}}]\label{thm:bi wave schro inhom}
Let $0\les \alpha < \frac{1}{2}$. For all $\lambda \in 2^\NN$ and $ \mu \in 2^\ZZ$ with $\mu \lesa \lambda$, we have
    \begin{equation}\label{eqn:thm bi wave-schro:low-high}
        \| u_{\lambda}  e^{it|\nabla|} \dot{P}_{\mu} g \|_{L^2_{t,x}(\RR^{1+4})} \lesa \Big( \frac{\mu}{\lambda}\Big)^{\alpha} \mu \| u_{\lambda} \|_{Z} \| \dot{P}_\mu g \|_{L^2_x}.
    \end{equation}
\end{theorem}

The second bilinear estimate we require is an inhomogeneous version of a bilinear restriction estimate for the paraboloid.

\begin{theorem}[Bilinear restriction for inhomogeneous Schr\"odinger {\cite[Theorem 4.2]{Candy2021}}]\label{thm:bi schro inhom}
Let $r>\frac{5}{3}$. For any $\mu \in 2^\ZZ$ we have
    \begin{equation}\label{eq:thm bi-schro:high-high}
    		\| \dot{P}_\mu (\overline{w} u) \|_{L^1_t L^r_x(\RR^{1+4})} \lesa \mu^{2-\frac{4}{r}}  \| w \|_{Z} \|u \|_{Z}.
    	\end{equation}
\end{theorem}

For technical reasons, and to slightly simplify the arguments below, we require the following extension of the above bilinear restriction estimates.

\begin{corollary}[Bilinear restriction in iteration norms]\label{cor:bi restric}
Let $r>\frac{5}{3}$. If $\tilde{\mu}\in 2^\ZZ$ and $\mu, \lambda \in 2^\NN$ with $\mu \lesa \lambda$ then
	$$
	 \| u_\lambda \dot{P}_{\tilde{\mu}} v_\mu \|_{L^2_{t,x}}
            \lesa \Big( \frac{\tilde{\mu}}{\lambda}\Big)^\frac{1}{4} \tilde{\mu} \|u_\lambda \|_{\underline{S}^0_\lambda}  \|v_\mu\|_{\underline{W}^0_\mu}.
    $$
Moreover, for any $\tilde{\mu} \in 2^\ZZ$ and $\lambda_1, \lambda_2\in 2^\NN$ we have
	$$
		\| \dot{P}_{\tilde{\mu}}(\overline{w} u)\|_{L^1_t L^r_x}\lesa \tilde{\mu}^{2-\frac{4}{r}}\min\big\{ \|w\|_Z, \|w\|_{\underline{S}^0_{\lambda_1}}\big\} \| u \|_{\underline{S}^0_{\lambda_2}}.
	$$
\end{corollary}
\begin{proof}
The proof follows from Theorem \ref{thm:bi wave schro inhom} and Theorem \ref{thm:bi schro inhom} by adapting the standard $X^{s,b}$ transference argument. We start by replacing the norm $\|u \|_Z$ with $\| u \|_{\underline{S}^0_\lambda}$. To this end, we write
	$$ (C_d u)(t,x) = \int_\RR e^{it\tau} (e^{it\Delta} f^{(\tau)}_{d})(x) d\tau $$
with $\widehat{f}^{(\tau)}_{d}(\xi) = \mc{F}_{t,x}[C_d u](\tau + |\xi|^2, \xi)$ and note that
	\begin{align*}
\int_\RR \| f^{(\tau)}_{d} \|_{L^2_x} d\tau \approx \int_{|\tau|\approx d} \| f^{(\tau)}_{d} \|_{L^2_x} d\tau \lesa d^\frac{1}{2} \| C_d u \|_{L^2_{t,x}} \approx d^{-\frac{1}{2}} \| (i\p_t +\Delta)u \|_{L^2_{t,x}}
	\end{align*}
In particular, by definition of the norm $\| \cdot \|_{\underline{S}^0_\lambda}$, we have the bound
	$$ \| C_{\ll \lambda^2} u \|_Z + \sum_{d\gtrsim \lambda^2} d^{-\frac{1}{2}} \int_\RR \|f^{(\tau)}_{d}\|_{L^2_x} d\tau \lesa \| u \|_{\underline{S}^0_\lambda} $$
and hence, applying Theorem \ref{thm:bi wave schro inhom}, we see that
	\begin{align}
		 \|u_\lambda e^{it|\nabla|}\dot{P}_{\tilde{\mu}} g \|_{L^2_{t,x}} &\lesa \| C_{\ll \lambda^2} u_\lambda   e^{it|\nabla|}\dot{P}_{\tilde{\mu}} g \|_{L^2_{t,x}}+ \sum_{d\gtrsim \lambda^2} \| C_d u_\lambda e^{it|\nabla|}\dot{P}_{\tilde{\mu}} g \|_{L^2_{t,x}} \notag\\
		&\lesa  \Big(\frac{\tilde{\mu}}{\lambda}\Big)^\frac{1}{4} \tilde{\mu} \| C_{\ll \lambda^2}u_\lambda \|_Z\|\dot{P}_{\tilde{\mu}} g\|_{L^2_x} + \sum_{d\gtrsim \lambda^2} \int_{\RR} \| e^{it\Delta} f^{(\tau)}_{\lambda, d}e^{it|\nabla|} \dot{P}_{\tilde{\mu}} g \|_{L^2_{t,x}} d\tau \notag \\
		&\lesa  \Big(\frac{\tilde{\mu}}{\lambda}\Big)^\frac{1}{4} \tilde{\mu} \Big( \| C_{\ll \lambda^2}u_\lambda \|_Z + \sum_{d\gtrsim \lambda^2} \int_{\RR} \| f^{(\tau)}_{\lambda, d}\|_{L^2_{x}} d\tau \Big) \|\dot{P}_{\tilde{\mu}} g\|_{L^2_x} \notag \\
		&\lesa \Big(\frac{\tilde{\mu}}{\lambda}\Big)^\frac{1}{4} \tilde{\mu} \|u_\lambda \|_{\underline{S}^0_\lambda}\|\dot{P}_{\tilde{\mu}} g\|_{L^2_x} . \label{eqn:bi restric:free wave L2}
	\end{align}
Similarly, applying Theorem \ref{thm:bi schro inhom}, we have
	\begin{align*}
		\| \dot{P}_{\tilde{\mu}}( \overline{w} u ) \|_{L^1_t L^r_x}&\lesa \| \dot{P}_{\tilde{\mu}}(\overline{w} C_{\ll \lambda_2^2}u \|_{L^1_t L^r_x} + \sum_{d\gtrsim \lambda_2^2} \|\dot{P}_{\tilde{\mu}}( \overline{w} C_d u) \|_{L^1_t L^r_x} \\
		&\lesa \tilde{\mu}^{2-\frac{4}{r}} \| w\|_Z \| C_{\ll \lambda_2^2}u \|_{Z} + \sum_{d\gtrsim \lambda_2^2} \int_\RR \|\dot{P}_{\tilde{\mu}}( \overline{w} e^{it\Delta} f^{(\tau)}_{d} ) \|_{L^1_t L^r_x} d\tau \\
		&\lesa \tilde{\mu}^{2-\frac{4}{r}} \| w \|_Z \Big( \| C_{\ll \lambda^2_2}u \|_Z + \sum_{d\gtrsim \lambda_2^2}  \int_\RR \| f^{(\tau)}_{ d} \|_{L^2_x} d\tau \Big) \\
		&\lesa \tilde{\mu}^{2-\frac{4}{r}} \| w \|_Z \|u \|_{\underline{S}^0_{\lambda_2}}.
    \end{align*}
Repeating this argument for $w$ shows that we can replace $\|w\|_Z$ with $\|w\|_{\underline{S}^0_{\lambda_1}}$. Thus it only remains to replace the free wave $e^{it|\nabla|}g$ in \eqref{eqn:bi restric:free wave L2} with a general function $v\in \underline{W}^0_{\mu}$, but this follows from a similar argument. Namely we write
	\begin{align*}
			v_\mu &= P^{(t)}_{\ll \mu^2 } v_\mu + \sum_{d\gtrsim \mu^2} P^{(t)}_d v_\mu
	\end{align*}
and apply the Duhamel formula together with \eqref{eqn:bi restric:free wave L2} for the first term. On the other hand, for the second term we argue via the standard $X^{s,b}$ transference argument, and observe that we can write
		$$ P^{(t)}_d v_\mu = \int_{|\tau| \approx d} e^{it\tau} e^{it|\nabla|} g^{(\tau)}_{\mu, d} d\tau  $$
with $\widehat{g}^{(\tau)}_{\mu, d}(\xi) = \mc{F}_{t,x}[P^{(t)}_d v_\mu](\tau + |\xi|, \xi)$. Hence the required bound  follows by arguing as in the proof of \eqref{eqn:bi restric:free wave L2}.
\end{proof}

\subsection{Schr\"odinger Nonlinearity}

The next step is to combine the above bilinear restriction type estimates to give a gain when bounding the Schr\"odinger nonlinearity. The key observation is that we have additional regularity to exploit, namely in the applications to follow,
the Schr\"odinger evolution lies in the energy space $H^1$. In particular, if we only wish to obtain an improved bound at lower regularities, we have some room to exploit. On way to do this is via the following.

\begin{theorem}[Frequency localised bound for Schr\"{o}dinger nonlinearity]\label{thm:freq loc bi schro}
Let $0< s \les s'< 1$ and $\ell \g 0$. There exists $0<\theta<1$ such that for any interval $I\subset \RR$ and any $\mu, \lambda_0, \lambda_1 \in 2^\NN$ we have
        $$\big\| \mc{I}_0\big[ \ind_I P_{\lambda_0}(v_\mu u_{\lambda_1}) \big] \big\|_{S^s_{\lambda_0}} \lesa  \Big( \frac{\lambda_{min}}{\lambda_{max}}\Big)^{\theta} \Big(\| v_{\mu} \|_{Y(I)} \| u_{\lambda_1} \|_{Y(I)} \Big)^{\theta \sigma}  \Big( \| v_\mu \|_{\underline{W}^\ell_\mu} \| u_{\lambda_1} \|_{\underline{S}^{s'}_{\lambda_1}}\Big)^{1-\theta \sigma } $$
where $9\sigma = \ell + s' - s$,  $\lambda_{min} = \min\{\mu, \lambda_0, \lambda_1\}$,  and $\lambda_{max}=\max\{\mu, \lambda_0, \lambda_1\}$.
\end{theorem}
\begin{proof}
Fix $0<s<1$. It suffices to consider the case where $9\sigma = \ell + s'-s$ is sufficiently small, and $0\les \ell < \frac{1}{4}$. We begin by observing that if  $\frac{1}{q}= \frac{\sigma}{2(1-\sigma)}$ and $\frac{1}{r} = \frac{1}{2} - \frac{2}{3q}$, then provided $\sigma\g 0$ is sufficiently small, the pair $(q,r)$ is wave Strichartz admissible. Hence the Strichartz estimates \eqref{eqn:schro stri} and \eqref{eqn:wave stri} together with an application of Bernstein's inequality gives for any $ \mu, \lambda_0, \lambda_1 \in 2^\NN$, and any $\tilde{\mu} \in 2^\ZZ $
    \begin{align*}
       \| &P_{\lambda_0}( \dot{P}_{\tilde{\mu}} v_\mu u_{\lambda_1}) \|_{L^2_{t,x}(I)}  \\
        &\lesa (\min\{\lambda_0, \lambda_1, \tilde{\mu}\})^{1-\frac{13\sigma}{3}} \Big( \| v_\mu \|_{L^\infty_{t,x}(I)} \| u_{\lambda_1} \|_{L^\infty_{t,x}(I)}\Big)^\sigma \Big( \|  v_\mu \|_{L^q_t L^r_x} \| u_{\lambda_1} \|_{L^2_t L^4_x} \Big)^{1-\sigma}  \\
      &\lesa  (\min\{\lambda_0, \lambda_1, \tilde{\mu}\})^{1-\frac{13\sigma}{3}} \mu^{\frac{29\sigma}{6} - \ell(1-\sigma)}\lambda_1^{4\sigma - s'(1-\sigma)}  \Big( \| v_{\mu} \|_{Y(I)} \| u_{\lambda_1} \|_{Y(I)} \Big)^\sigma \Big( \| v_\mu \|_{\underline{W}^\ell_\mu} \| u_{\lambda_1} \|_{S^{s'}_{\lambda_1}} \Big)^{1-\sigma}.
    \end{align*}
Note that this bound is only nontrivial when $\lr{\tilde{\mu}} \approx \mu$, and two largest frequencies are comparable. Thus the additional homogeneous $\dot{P}_{\tilde{\mu}}$ multiplier is only interesting when $\mu =1$, as then $\dot{P}_{\tilde{\mu}}$ forces an additional localisation to the frequency region $|\xi| \approx \tilde{\mu}\lesa 1$ (compared to the inhomogeneous localisation $|\xi|\les 2$ provided by $v_\mu$ when $\mu=1$). In particular, a short computation shows that provided $\sigma\g0$ is sufficiently small, if $\mu, \lambda \in 2^\NN$ with $\mu \lesa \lambda$, and $\tilde{\mu}\in 2^\ZZ$, then
     \begin{equation}\label{eqn:bi res app:main L2}
      \lambda^{s-1} \|  \dot{P}_{\tilde{\mu}} v_\mu u_{\lambda} \|_{L^2_{t,x}(I)} \lesa \Big(\frac{\tilde{\mu}}{\lambda}\Big)^{\frac{3}{4}} \Big(\| v_{\mu} \|_{Y(I)} \| u_{\lambda_1} \|_{Y(I)} \Big)^{ \sigma} \Big( \| v_\mu \|_{\underline{W}^\ell_\mu} \| u_{\lambda} \|_{S^{s'}_{\lambda}} \Big)^{1-\sigma}
    \end{equation}
while if $\mu \approx \max\{\lambda_0, \lambda_1\} \gg \min\{\lambda_0, \lambda_1\}$ (and $\lambda_0, \lambda_1\in 2^\NN$) we instead have
    \begin{equation}\label{eqn:bi res app:main L2 gen freq}
    \begin{split}
      \lambda^{s-1}_0 \|  P_{\lambda_0}(v_\mu u_{\lambda_1}) \|_{L^2_{t,x}(I)}
      &\lesa \Big(\frac{\min\{\lambda_0, \lambda_1\}}{\mu}\Big)^{\min\{\frac{s}{2}, \frac{1-s}{2}\}} \Big(\| v_{\mu} \|_{Y(I)} \| u_{\lambda_1} \|_{Y(I)} \Big)^{ \sigma} \Big( \| v_\mu \|_{\underline{W}^\ell_\mu} \| u_{\lambda_1} \|_{S^{s'}_{\lambda_1}} \Big)^{1-\sigma}.
    \end{split}
    \end{equation}
This clearly suffices to bound the $L^2_{t,x}$ component of the $S^s_{\lambda_0}$ norm after potentially decomposing into small frequencies via  $v_\mu = \sum_{\tilde{\mu}\in 2^\ZZ, \lr{\tilde{\mu}} \approx \mu} \dot{P}_{\tilde{\mu}} v_\mu$.

To bound the Stricharz component of the $S^s_\lambda$ norm, we consider separately the cases $\mu \lesa \lambda_0 \approx \lambda_1$ and $\mu \gg \min\{\lambda_0, \lambda_1\}$, and adapt the argument used to prove \cite[Theorem 5.1]{Candy2021} together with Corollary \ref{cor:bi restric}. More precisely, in view of the duality (see \cite[Lemma 2.2]{Candy2021}), to  deal with the case $\mu \lesa \lambda_0 \approx \lambda_1$ it suffices to prove that if $1\lesa \mu\lesa \lambda$ then we have $\theta>0$ such that
    \begin{equation}\label{eqn:bi res app:low wave main}
        \Big| \int_{I\times \RR^4} v_\mu \overline{w} u_\lambda \,dx\,dt\Big| \lesa \lambda^{-s} \Big(\frac{\mu}{\lambda}\Big)^{\theta} \|w\|_Z  \Big( \| v_{\mu} \|_{Y(I)} \| u_{\lambda_1} \|_{Y(I)} \Big)^{\theta \sigma} \Big( \| v_\mu  \|_{\underline{W}^\ell_\mu} \| u_\lambda \|_{S^{s'}_\lambda} \Big)^{1-\theta \sigma}.
    \end{equation}
Let $K_{\tilde{\mu}} \in L^1$ denote the convolution kernel of the Fourier multiplier $\dot{P}_{\approx {\tilde{\mu}}}$, and take $\frac{1}{r} = \frac{1}{2} + \frac{\theta}{4(1-\theta)}$ where $0<\theta<1$. Then provided $\theta>0$ is sufficiently small, for any $w\in Z$ we have via H\"older's inequality and Corollary \ref{cor:bi restric}
	\begin{align*}
		\Big| \int_{I\times \RR^{4}}  v_\mu \overline{w} u_\lambda\,dx\,dt \Big|
        &\lesa \sum_{\substack{\tilde{\mu} \in 2^\ZZ \\ \lr{\tilde{\mu}}\approx \mu}} \Big| \int_{I\times \RR^{4}} \dot{P}_{\tilde{\mu}} v_\mu \overline{w} u_\lambda\,dx\,dt \Big| \\
        &\les  \sum_{\substack{\tilde{\mu} \in 2^\ZZ \\ \lr{\tilde{\mu}}\approx \mu}} \| \dot{P}_{\tilde{\mu}} v_\mu   \dot{P}_{\approx \tilde{\mu}}(\overline{w}u_\lambda) \|_{L^1_t L^\frac{4}{3}_x(I)}^\theta \| \dot{P}_{\tilde{\mu}} v_\mu \dot{P}_{\approx \tilde{\mu} }(\overline{w}u_\lambda)\|_{L^1_t L^{\frac{2r}{2+r}}_x}^{1-\theta}\\
        &\lesa \sum_{\substack{\tilde{\mu} \in 2^\ZZ \\ \lr{\tilde{\mu}}\approx \mu}} \Big( \int_{\RR^4} |K_{\tilde{\mu}}(y)| \| (\dot{P}_{\tilde{\mu}} v_\mu)(t,x) \overline{w}(t,x-y) u_\lambda(t,x-y) \|_{L^1_t L^{\frac{4}{3}}_x(I)} dy \Big)^{\theta}\\
         &\qquad \qquad \qquad \qquad \times \Big( \| v_\mu \|_{L^\infty_t L^2_x} \| \dot{P}_\mu ( \overline{w} u_\lambda )\|_{L^1_t L^r_x}\Big)^{1-\theta}\\
		&\lesa  \lambda^{-s} \| w \|_{Z}   \sum_{\substack{\tilde{\mu} \in 2^\ZZ \\ \lr{\tilde{\mu}}\approx \mu}}  \Big( {\tilde{\mu}^{-1}} \lambda^s \sup_{y\in \RR^4} \| (\dot{P}_{\tilde{\mu}} v_\mu)(t,x) u_\lambda(t,x-y) \|_{L^2_{t,x}(I)} \Big)^\theta \Big(\|v_\mu \|_{L^\infty_t L^2_x} \| u_\lambda \|_{\underline{S}^s_\lambda} \Big)^{1-\theta}.
	\end{align*}
To bound the $L^2_{t,x}$ norm in the above inequality, we combine Corollary \ref{cor:bi restric} and \eqref{eqn:bi res app:main L2} which gives
    \begin{align*}
       \tilde{\mu}^{-1} \lambda^s \| \dot{P}_{\tilde{\mu}} v_\mu u_\lambda \|_{L^2_{t,x}(I)}
       &\lesa \Big[ \Big(\frac{\tilde{\mu}}{\lambda}\Big)^{-\frac{1}{4}}  \Big( \| v_{\mu} \|_{Y(I)} \| u_{\lambda_1} \|_{Y(I)}  \Big)^{\sigma} \Big( \| v_\mu \|_{\underline{W}^\ell_\mu} \| u_\lambda \|_{S^{s'}_\lambda} \Big)^{1-\sigma} \Big]^\frac{1}{4} \Big[ \Big( \frac{\tilde{\mu}}{\lambda}\Big)^{\frac{1}{4} } \|v_\mu \|_{\underline{W}^\ell_\mu} \|u_\lambda \|_{\underline{S}^{s'}_\lambda} \Big]^\frac{3}{4}\\
        &\approx \Big(\frac{\tilde{\mu}}{\lambda}\Big)^{\frac{1}{4}}  \Big( \| v_{\mu} \|_{Y(I)} \| u_{\lambda_1} \|_{Y(I)} \Big)^{\frac{ \sigma}{4}} \Big( \|v_\mu \|_{\underline{W}^\ell_\mu} \| u_\lambda \|_{\underline{S}^{s'}_\lambda} \Big)^{1-\frac{\sigma}{4}}.
    \end{align*}
Therefore the bound \eqref{eqn:bi res app:low wave main} follows from the translation invariance of the norm $\| \cdot \|_{\underline{S}^s_\lambda}$, and summing up over low frequencies $\tilde{\mu} \lesa \mu$.

It only remains to consider the case $\mu \approx \max\{\lambda_0, \lambda_1\} \gg \min\{\lambda_0, \lambda_1\}$. Let $\frac{1}{q} = \frac{\sigma}{2(1-\sigma)}$, $\frac{1}{r} = \frac{1}{2} - \frac{2}{3q}$, and $\frac{1}{a} = \frac{1}{2} + \frac{7\sigma}{12(1-\sigma)}$. Then provided $\sigma\g 0$ is sufficiently small,  the pair $(q, r)$ is wave Strichartz admissible and hence an application of \eqref{eqn:wave stri} together with Corollary \ref{cor:bi restric} (noting that $\overline{w}_{\lambda_0} u_{\lambda_1} = P_{\approx \mu} (\overline{w}_{\lambda_0} u_{\lambda_1})$ )  and the endpoint Strichartz estimate for the Schr\"odinger equation gives
    \begin{align*}
        \Big| \int_{I\times \RR^4} v_\mu \overline{w}_{\lambda_0} u_{\lambda_1} \,dx\,dt\Big|&\lesa \Big( \| v_{\mu} \|_{L^\infty_{t,x}(I)} \|w_{\lambda_0} \|_{L^2_t L^4_x} \| u_{\lambda_1} \|_{L^\infty_{t,x}(I)} \Big)^{\sigma} \Big( \| v_\mu \|_{L^q_t L^r_x} \| \overline{w}_{\lambda_0} u_{\lambda_1} \|_{L^1_t L^a_x} \Big)^{1-\sigma}\\
        &\lesa \|w\|_Z \Big( \| v_{\mu} \|_{L^\infty_{t,x}(I)} \|u_{\lambda_1}\|_{L^\infty_{t,x}(I)} \Big)^{\sigma}
        \Big(\mu^{\frac{5}{3q} - \ell} \lambda_1^{2-\frac{4}{a} -s'} \| v_\mu\|_{\underline{W}^\ell_{\lambda}} \|u_{\lambda_1}\|_{\underline{S}^{s'}_\lambda}\Big)^{1-\sigma} \\
        &\lesa \mu^{6\sigma-\ell} \lambda_1^{3\sigma - s'} \|w\|_Z \Big(  \|v_\mu\|_{Y(I)} \| u_{\lambda_1} \|_{Y(I)} \Big)^{\sigma}
        \Big( \| v_\mu\|_{\underline{W}^\ell_{\lambda}} \|u_{\lambda_1}\|_{\underline{S}^{s'}_\lambda}\Big)^{1-\sigma}
    \end{align*}
where the last line again used the assumption $0\les \ell < \frac{1}{4}$ and $0<s'< 1$ to slightly simplify the frequency exponents. Therefore another application of duality (see \cite[Lemma 2.2]{Candy2021}) gives for any $\mu \approx \max\{\lambda_0, \lambda_1\} \gg \min\{\lambda_0, \lambda_1\}$ the bound
	$$ \lambda_0^s \| \mc{I}_0[\ind_I P_{\lambda_0}( v_\mu u_{\lambda_1})] \|_{L^\infty_t L^2_x \cap L^2_t L^4_x} \lesa \mu^{9\sigma-\ell+s} \lambda_1^{- s'}  \Big(  \|v_\mu\|_{Y(I)} \| u_{\lambda_1} \|_{Y(I)} \Big)^{\sigma}
        \Big( \| v_\mu\|_{\underline{W}^\ell_{\lambda}} \|u_{\lambda_1}\|_{\underline{S}^{s'}_\lambda}\Big)^{1-\sigma}.  $$
Since $9\sigma = \ell + s' - s$ and $s>0$, this suffices when $\mu \approx \lambda_1 \gg \lambda_0$. On the other hand, when $\mu \approx \lambda_0 \gg \lambda_1$, we instead interpolate with the high-low frequency gain
   $$ \lambda_0^s \| \mc{I}_0[\ind_I P_{\lambda_0}( v_\mu u_{\lambda_1})] \|_{L^\infty_t L^2_x \cap L^2_t L^4_x} \lesa \Big( \frac{\lambda_1}{\mu}\Big)^{1-s} \| v_\mu\|_{\underline{W}^0_{\lambda}} \|u_{\lambda_1}\|_{\underline{S}^{s}_\lambda}  $$
 which is a consequence of (the proof of) \cite[Theorem 3.1]{Candy2021} (see (3.5) and equation below (3.3) in \cite{Candy2021}) together with the energy inequality \eqref{eqn:energy ineq}. Note that the temporal cutoff $\ind_I$ can be removed by applying \eqref{eqn:duh loc time:init trans} and Lemma \ref{lem:duh loc time}.
\end{proof}

Note that in the previous theorem, we have two options. We can either take $\ell=0$ and $s'>s$ (so spend additional Schr\"odinger regularity), or $\ell>0$ and $s'=s$ (so spend additional wave regularity).

\subsection{Wave Nonlinearity}

We have a similar frequency localised improved bilinear estimate for the wave nonlinearity.

\begin{theorem}[Frequency localised control of wave nonlinearity]\label{thm:freq loc bi wave}
Let $s>\frac{1}{2}$. There exists $\theta>0$ such that for any interval $I \subset \RR$ and any $\mu, \lambda_1, \lambda_2 \in 2^\NN$ we have
    $$ \big\| \mc{J}_0\big[ \ind_I |\nabla| \Re\big(P_{\mu}(\overline{w}_{\lambda_1} u_{\lambda_2})\big)\big] \big\|_{W^0_\mu} \lesa \Big( \frac{\lambda_{min}}{\lambda_{max}}\Big)^{\theta} \Big(\| w_{\lambda_1} \|_{Y(I)} \| u_{\lambda_2} \|_{Y(I)} \Big)^{\theta}  \Big( \| w_{\lambda_1} \|_{\underline{S}^{\frac{1}{2}}_{\lambda_1}} \| u_{\lambda_2} \|_{\underline{S}^{s}_{\lambda_2}}\Big)^{1-\theta } $$
where $\lambda_{min} = \min\{\mu, \lambda_1, \lambda_2\}$ and $\lambda_{max}=\max\{\mu, \lambda_1, \lambda_2\}$.
\end{theorem}
\begin{proof}
It suffices to consider the case $\frac{1}{2}<s\les 1$. Let $\theta>0$ (sufficiently small), and take $\frac{1}{q} = \frac{1}{4(1-\theta)}$ and $\frac{1}{r} = \frac{1}{2} - \frac{1}{2q}$. An application of \eqref{eqn:schro stri} together with Bernstein's inequality gives
    \begin{align*}
        \mu^{-1} \big\| |\nabla| \Re\big(P_{\mu}(\overline{w}_{\lambda_1} u_{\lambda_2})\big) \big\|_{L^2_{t,x}(I)} &\lesa \Big( \| w_{\lambda_1} \|_{L^\infty_{t,x}(I)} \| u_{\lambda_2} \|_{L^\infty_{t,x}(I)}\Big)^\theta  \big\| P_{\mu}\big( \overline{w}_{\lambda_1} u_{\lambda_2} \big) \big\|_{L^\frac{q}{2}_{t,x}}^{1-\theta} \\
                &\lesa \Big( \| w_{\lambda_1} \|_{L^\infty_{t,x}(I)} \| u_{\lambda_2} \|_{L^\infty_{t,x}(I)}\Big)^\theta  \Big( \lambda_{min}^{4(\frac{2}{r} - \frac{2}{q})}\big\| P_{\mu}\big( \overline{w}_{\lambda_1} u_{\lambda_2} \big) \big\|_{L^\frac{q}{2}_t L^\frac{r}{2}_x}\Big)^{1-\theta} \\
                &\lesa \lambda_{min}^{1-4\theta} \lambda_{max}^{8\theta - (1-\theta)(\frac{1}{2}+s)} \Big(\| w_{\lambda_1} \|_{Y(I)} \| u_{\lambda_2} \|_{Y(I)} \Big)^{\theta}  \Big( \| w_{\lambda_1} \|_{\underline{S}^\frac{1}{2}_{\lambda_1}} \| u_{\lambda_2} \|_{\underline{S}^{s}_{\lambda_2}}\Big)^{1-\theta }.
    \end{align*}
After observing that $\lambda_{min}^{1-4\theta} \lambda_{max}^{8
\theta - (1-\theta)(\frac{1}{2}+s)}\lesa (\frac{\lambda_{min}}{\lambda_{max}})^{1-4\theta} \lambda_{max}^{\frac{1}{2}-s + 6\theta}$, the previous bound suffices to control the $L^2_{t,x}$ component of the norm $W^0_\mu$ provided that $s\g \frac{1}{2} + 6\theta$. To bound the $L^\infty_t L^2_x$ component, we take the wave admissible exponents $\frac{1}{q} = \theta$ and $\frac{1}{r} = \frac{1}{2} - \frac{2\theta}{3}$, and let $\frac{1}{a} = \frac{1}{1-\theta}(\frac{1}{2} + \frac{2\theta}{3})$. The bilinear restriction type estimate in Corollary \ref{cor:bi restric} gives for any $\tilde{\mu}\in 2^\ZZ$
        \begin{align*}
        \| \dot{P}_{\tilde{\mu}}(\overline{w}_{\lambda_1} u_{\lambda_2})\|_{L^{q'}_t L^{r'}_x(I)}
                        &\lesa \Big( \| w_{\lambda_1} \|_{L^\infty_{t,x}(I)} \|u_{\lambda_2} \|_{L^\infty_{t,x}(I)} \Big)^\theta \| \dot{P}_{\tilde{\mu}}(\overline{w}_{\lambda_1} u_{\lambda_2})\|_{L^1_t L^a_x(I)}^{1-\theta}\\
                        &\lesa  (\lambda_1 \lambda_2)^{4\theta} \big(\tilde{\mu}^{2-\frac{4}{a}} \lambda_1^{-\frac{1}{2}} \lambda_2^{-s} \big)^{1-\theta} \Big(\| w_{\lambda_1} \|_{Y(I)} \| u_{\lambda_2} \|_{Y(I)} \Big)^{\theta}  \Big( \| w_{\lambda_1} \|_{\underline{S}^\frac{1}{2}_{\lambda_1}} \| u_{\lambda_2} \|_{\underline{S}^{s}_{\lambda_2}}\Big)^{1-\theta } \\
                        &\lesa \tilde{\mu}^{-(\frac{5}{3}\theta +1)} \Big(\frac{\tilde{\mu}}{\lambda_1}\Big)^{\frac{1}{2}-3\theta}\Big( \frac{\tilde{\mu}}{\lambda_2} \Big)^\frac{1}{2} \lambda_2^{\frac{1}{2}-s+7\theta} \Big(\| w_{\lambda_1} \|_{Y(I)} \| u_{\lambda_2} \|_{Y(I)} \Big)^{\theta}  \Big( \| w_{\lambda_1} \|_{\underline{S}^\frac{1}{2}_{\lambda_1}} \| u_{\lambda_2} \|_{\underline{S}^{s}_{\lambda_2}}\Big)^{1-\theta }.
        \end{align*}
Consequently, as the pair $(q,r)$ is wave admissable, provided $s>\frac{1}{2}+7\theta$ the inhomogeneous Strichartz estimate for the wave equation gives
	\begin{align}
		\big\| \mc{J}_0\big[\ind_I |\nabla| \Re\big( P_\mu(\overline{w}_{\lambda_1}u_{\lambda_2})\big)\big]\big\|_{L^\infty_t L^2_x}
		&\lesa  \sum_{\substack{\tilde{\mu} \in 2^\NN \\ \lr{\tilde{\mu}} \approx \mu}} \tilde{\mu} \big\| \mc{J}_0\big[\ind_I \Re\big( \dot{P}_{\tilde{\mu}}(\overline{w}_{\lambda_1}u_{\lambda_2})\big)\big]\big\|_{L^\infty_t L^2_x} \label{eqn:freq loc bi wave:Linf bound} \\
		&\lesa \sum_{\substack{\tilde{\mu} \in 2^\NN \\ \lr{\tilde{\mu}} \approx \mu}} \tilde{\mu}^{1+\frac{5}{3}\theta} \|\dot{P}_{\tilde{\mu}}( \overline{w}_{\lambda_1} u_{\lambda_2})\|_{L^{q'}_t L^{r'}_x(I)} \notag \\
		&\lesa  \Big(\frac{\mu}{\lambda_1}\Big)^{\frac{1}{2}-3\theta}\Big( \frac{\mu}{\lambda_2} \Big)^\frac{1}{2} \Big(\| w_{\lambda_1} \|_{Y(I)} \| u_{\lambda_2} \|_{Y(I)} \Big)^{\theta}  \Big( \| w_{\lambda_1} \|_{\underline{S}^\frac{1}{2}_{\lambda_1}} \| u_{\lambda_2} \|_{\underline{S}^{s}_{\lambda_2}}\Big)^{1-\theta }. \notag
	\end{align}
This is enough to bound the $L^\infty_t L^2_x$ contribution in the case $\mu \approx \lambda_{\min}$, thus it only remains to consider the case $\mu\approx \lambda_{\max} \gg \lambda_{min}$. But this follows by combining \eqref{eqn:freq loc bi wave:Linf bound} with the  derivative gain
    \begin{align*}
 \big\| \mc{J}_0\big[ \ind_I |\nabla| \Re\big(P_{\mu}(\overline{w}_{\lambda_1} u_{\lambda_2})\big)\big] \big\|_{L^\infty_t L^2_x}  &\lesa \mu^{1-2s}  \| w_{\lambda_1} \|_{\underline{S}^s_{\lambda_1}} \| u_{\lambda_2} \|_{\underline{S}^{s}_{\lambda_2}}\\
  &\approx \mu^{1-2s}\lambda_1^{s-\frac{1}{2}} \| w_{\lambda_1} \|_{\underline{S}^{\frac{1}{2}}_{\lambda_1}} \| u_{\lambda_2} \|_{\underline{S}^{s}_{\lambda_2}} \lesa \mu^{\frac{1}{2}-s} \| w_{\lambda_1} \|_{\underline{S}^{\frac{1}{2}}_{\lambda_1}} \| u_{\lambda_2} \|_{\underline{S}^{s}_{\lambda_2}}
	\end{align*}
which is a consequence of \cite[Proof of (4.1) case 1]{Candy2021a} and \cite[Lemma 2.6]{Candy2021a} in the case $\mu \approx \lambda_{max} \gg \lambda_{min}$ (note that $W^\ell = W^{\ell, 0, 0}$ and $S^s = S^{s, 0, 0}$, and so $a=b=0$ in the notation of \cite{Candy2021a}).
\end{proof}

\subsection{Proof of Theorem \ref{thm:bi est improved}}

The proof of Theorem \ref{thm:bi est improved} is now a direct application of the corresponding frequency localised bounds.

\begin{proof}[Proof of Theorem \ref{thm:bi est improved}]
We only give the proof of the Schr\"odinger bound, as the wave bound is similar. Let $I\subset \RR$ be an interval. In view of \eqref{eqn:Y norm bound by sobolev}, have $\|u_\lambda \|_Y\lesa \|u_\lambda \|_{S^s_\lambda}$ and $\|V_\mu\|_Y\lesa \|V_\mu\|_{W^0_\mu}$, and hence we may take $\theta>0$ as small as we like. An application of Theorem \ref{thm:freq loc bi schro} together with a standard summation argument gives for any $s>\frac{1}{2}$
		$$ \big\| \mc{I}_0\big[ \ind_I \Re(V) u \big] \big\|_{S^\frac{1}{2}} \lesa \Big( \|u\|_{Y(I)} \|V\|_{Y(I)}\Big)^\theta \Big( \sum_{\mu\in 2^\NN} \|V_\mu\|_{W^0_\mu}^{4(1-\theta)} \Big)^\frac{1}{2} \Big( \sum_{\lambda \in 2^\NN} \| u_\lambda \|_{S^s_\lambda}^{4(1-\theta)}\Big)^{\frac{1}{2}}. $$
In particular, choosing $\theta>0$ slightly smaller if necessary, we can ensure that $4(1-\theta) \g 2$ and hence conclude that
		$$ \big\| \mc{I}_0\big[ \ind_I \Re(V) u \big] \big\|_{S^\frac{1}{2}} \lesa \Big( \|u\|_{Y(I)} \|V\|_{Y(I)}\Big)^\theta \Big( \| V\|_{W^0} \| u \|_{S^s}\Big)^{1-\theta}.$$
By definition of the temporally restricted spaces $S^s(I)$ and $W^0(I)$, we can replace the norms on the right hand side by their corresponding localised counterparts. Therefore, provided $0\in I$, the bound for the Schr\"odinger nonlinearity now follows from Lemma \ref{lem:duh loc time}.
\end{proof}

\section{Small Data Theory II}\label{sec:small data II}

In this section give a small data theory based on the dispersive norm $\| \cdot \|_Y$. In particular, we give the proof of Theorem \ref{thm:small data Hs} as well as a version of the stability result from Theorem \ref{thm:stab} adapted to the concentration compactness arguments to follow. The idea is simply to combine the refined bilinear estimates from the previous section together with the small data theory developed in Section \ref{sec:small data}. The first step is to observe that the improved bilinear estimate in Theorem \ref{thm:bi est improved} implies that the controlling quantity $\rho_I$ defined in \eqref{eqn:rho defn} can be bounded in terms of the dispersive norm $\|\cdot\|_{Y(I)}$.

\begin{corollary}[Controlling norm bounded by $Y(I)$]\label{cor:rho bounded by Y}
Let $s>\frac{1}{2}$. There exists $C>0$ and $0<\theta<1$ such that for any interval $0\in I\subset \RR$ we have
        $$ \rho_I(u) \les C \| u \|_{Y(I)}^\theta \| u  \|_{\underline{S}^s}^{1-\theta}. $$
\end{corollary}
\begin{proof}
This follows directly from the definition of the controlling quantity $\rho_I$, \eqref{eqn:rho defn}, together with Lemma \ref{lem:duh loc time} and Theorem \ref{thm:bi est improved}.
\end{proof}

We can now state the version of the stability result in Theorem \ref{thm:stab} that we exploit in the following.

\begin{theorem}[Stability II]\label{thm:stab II}
Let $A>0$, $0<B<\|Q\|_{\dot{H}^1}$, and $\frac{1}{2}<s<2$. There exists $\epsilon_0(A,B, s)>0$, $\theta=\theta(A, B, s)>0$, and $C=C(A, B, s)>0$ such that if $0<\epsilon<\epsilon_0(A, B, s)$, $t_0 \in I\subset \RR$ is an interval, and
    $$(f, g) \in H^s\times L^2, \qquad (F, G)\in N^\frac{1}{2}(I)\times M^0(I), \qquad (\psi, \phi) \in C(I, H^\frac{1}{2} \times L^2)$$
satisfies for some $t_0\in I$ the boundedness condition
    \begin{equation}\label{eqn:stab II:bound}
       \|f\|_{H^s} + \| \psi \|_{S^\frac{1}{2}(I)}\les A, \qquad \| \phi \|_{L^\infty_t L^2_x (I\times \RR^4)} \les B,
    \end{equation}
the smallness condition
      \begin{equation}\label{eqn:stab II:error}
           \big\| e^{i(t-t_0)\Delta}\big(f-\psi(t_0)\big)\big\|_{Y(I)} + \| g - \phi(t_0) \|_{L^2}  + \| \mc{I}_{0, t_0}[F] \|_{S^\frac{1}{2}(I)} + \| \mc{J}_{0, t_0}[G] \|_{W^0(I)} < \epsilon,
      \end{equation}
and $(\psi, \phi)$ solve the perturbed Zakharov equation \eqref{eqn:stab:approx sys} on $I\times \RR^4$, then there exists a (unique) solution $(u, V)\in C(I, H^s\times L^2)$ to the Zakharov equation \eqref{eqn:Zak} with data $(u, V)(t_0) = (f,g)$ and the bounds
    $$\big\| u - \psi - e^{i(t-t_0)\Delta}\big(f - \psi(t_0)\big)\big\|_{S^\frac{1}{2}(I)}\les C \epsilon^\theta, \qquad \|u\|_{D(I)}\les C \| f\|_{H^\frac{1}{2}}, \qquad \| V\|_{L^\infty_t L^2_x(I)} \les \tfrac{1}{2} \big( B + \|Q \|_{\dot{H}^1}\big).$$
\end{theorem}
\begin{proof}
In view of the persistence of regularity from \cite[Theorem 8.1]{Candy2021a}, it suffices to consider the case $\frac{1}{2}<s<1$. An application of Corollary \ref{cor:rho bounded by Y} together with \eqref{eqn:schro stri} implies that
        $$ \rho_I(e^{it\Delta} f) \lesa \| e^{it\Delta} f\|_{Y(I)}^\theta \| e^{it\Delta} f \|_{\underline{S}^s(I)}^{1-\theta} \lesa \epsilon^\theta \|f \|_{H^s}^{1-\theta} \lesa \epsilon^\theta A^{1-\theta}  $$
where $\theta$ and the implied constant depend only on $s>\frac{1}{2}$. In particular, choosing $\epsilon>0$ sufficiently small to ensure that $\epsilon^\theta A^{1-\theta} \ll \epsilon_0(A, B)$ where $\epsilon_0$ is as in Theorem \ref{thm:stab}, the conditions \eqref{eqn:stab II:bound} and \eqref{eqn:stab II:error} imply the corresponding conditions \eqref{eqn:stab:bound} and \eqref{eqn:stab II:error} in Theorem \ref{thm:stab}. Consequently an application of Theorem \ref{thm:stab} implies that there exists a unique solution $(u, V) \in C(I;H^s\times L^2)$ with data $(u, V)(t_0)=(f,g)$ satisfying
    $$ \big\| u - \psi - e^{i(t-t_0)\Delta}\big(f - \psi(t_0)\big)\big\|_{S^\frac{1}{2}(I)}\lesa_{A, B, s} \epsilon^\theta,  \qquad \| V\|_{L^\infty_t L^2_x(I)} \les \tfrac{1}{2} \big( B + \|Q \|_{\dot{H}^1}\big) $$
where (as in the remainder of the proof) the implied constant depends only  on $A$, $B$, and $S$. To obtain the claimed bound in the Strichartz norm, we note that via \eqref{eqn:D bdded by S} and the homogeneous Strichartz estimate, we have
    \begin{align*}
        \| u \|_{D(I)} &\lesa \| u - \psi - e^{i(t-t_0)}(f-\psi(t_0)\big) \big\|_{D(I)} + \| e^{i(t-t_0)\Delta}\big(f-\psi(t_0)\big)\big\|_{D(I)} \\
                       &\lesa \| u - \psi - e^{i(t-t_0)}(f-\psi(t_0)\big) \big\|_{S^\frac{1}{2}(I)} + \|f - \psi(t_0) \|_{H^\frac{1}{2}} \lesa \epsilon^\theta + A.
    \end{align*}
In particular, $\| u \|_{D(I)}\lesa_{A, B, S} 1$, and hence the claimed Strichartz bound follows from Theorem \ref{thm:persistence}.
\end{proof}

Although it is not needed in the following, strictly speaking the bounds implied by Theorem \ref{thm:stab} are slightly stronger than those stated in the conclusion in of Theorem \ref{thm:stab II}. For later use, we observe that in view of the bound \eqref{eqn:Y norm bound by sobolev}, we note that the smallness condition \eqref{eqn:stab II:error} can be replaced with
    \begin{equation}\label{eqn:mod stab error cond}
         \min\big\{ \| f-\psi(t_0)\|_{L^2}, \|f - \psi(t_0)\|_{\dot{H}^1}\big\} + \| g - \phi(t_0) \|_{L^2}  + \| \mc{I}_{0, t_0}[F] \|_{S^\frac{1}{2}(I)} + \| \mc{J}_{0, t_0}[G] \|_{W^0(I)} < \epsilon.
    \end{equation}
Thus we only need the Schr\"odinger data $f$ and $\psi(t_0)$ to be close in the \emph{homogeneous} spaces $L^2$ or $\dot{H}^1$. This smallness in the homogeneous spaces is crucial, as we want to conclude that solutions which are close to a dispersive solution in either mass $\mc{M}$ (so $L^2$) or energy $\mc{E}_Z$ (so roughly $\dot{H}^1$), are also dispersive. It is worth noting that, since we are working in the subcritical setting $s>\frac{1}{2}$,  proving that smallness in $L^2$ suffices is straightforward via the convexity type bound
        $$ \| f\|_{H^\frac{1}{2}} \lesa \|f \|_{L^2}^\theta \| f \|_{H^s}^{1-\theta}. $$
On the other hand, proving that smallness in $\dot{H}^1$ suffices is significantly more challenging, as the low frequencies can be badly behaved. In particular there does not seem to be an easy direct argument to conclude that smallness in energy suffices for stability. \\

We now come to the proof of Theorem \ref{thm:small data Hs}.

\begin{proof}[Proof of Theorem \ref{thm:small data Hs}]
Similar to the proof of Corollary \ref{cor:lwp with rho}, this follows by letting $F=G=0$ and $\phi = e^{it|\nabla|} g$ in Theorem \ref{thm:stab II}.
\end{proof}

\section{Profile Decomposition}\label{sec:prof decomp}

The proof of Theorem \ref{thm:min coun} requires a precise understanding of how a bounded sequence in $H^1\times L^2$ can fail to be compact. As usual, this is accomplished via the profile decomposition of Bahouri-G\'erard \cite{Bahouri1999}. To describe the profile decomposition we require, we first introduce our notation for the group elements responsible for the lack of compactness. Given $(s,y)\in \RR\times \RR^4$ we define the group elements $\mf{g}[s,y]: H^1\times L^2 \to H^1 \times L^2$ as
        $$ \mf{g}[s,y](f,g) =  \big( e^{-is\Delta} \tau_y f, e^{-is|\nabla|} \tau_y g\big)$$
where $(\tau_y \phi)(x) = \phi(x-y)$ is the translation operator on $\RR^4$. Note that there is no frequency parameter in the group elements. This is due to the fact that sequences in $H^1\times L^2$ with frequency parameters going to either $0$ or $\infty$, are small in the $Y$ norm. For instance, if we consider the sequence
        $$ (f_n, g_n)(x) = \big( \lambda_n^2 \lr{\lambda_n}^{-1} f(\lambda_n x), \lambda_n^2 g(\lambda_n x)\big)$$
with say $\supp (\widehat{f}, \widehat{g}) \subset \{|\xi|\approx 1\}$. Then $\| (f_n, g_n)\|_{H^1\times L^2}\approx \|(f,g)\|_{H^1\times L^2}$ is a bounded sequence with no limit in $H^1\times L^2$ if $\lambda_n + \lambda_n^{-1} \to \infty$. On the other hand, a short computation gives $\|(e^{it\Delta} f_n, e^{it|\nabla|} g_n)\|_{Y} \to 0$ if $\lambda_n + \lambda_n^{-1} \to \infty$, and hence we can absorb sequences with frequencies going to zero or infinity, into an error term which is small in $Y$ (which suffices for our purposes as we have now developed a robust small data theory in $Y$!).

A precise formulation of the profile decomposition we require is as follows.

\begin{theorem}[Profile decomposition]\label{thm:profile decomp}
Let $(f_n, g_n) \in H^1\times L^2$ be a bounded sequence. Then after passing to a subsequence if necessary, there exists $J^*\in \NN \cup \{\infty\}$, non-zero profiles $(f^{(j)}, g^{(j)})\in H^1\times L^2$, and group elements $\mf{g}_n^{(j)} = \mf{g}[t_n^{(j)}, x_n^{(j)}]$ such that if we define $(w_n^{(J)}, e_n^{(J)}) \in H^1\times L^2$ as
        $$ (f_n, g_n) = \sum_{j=1}^J \mf{g}_n^{(j)}( f^{(j)}, g^{(j)}) + (w_n^{(J)}, e_n^{(J)})$$
then we have the properties:
    \begin{enumerate}
      \item The energy of the profiles decouples, thus for any $J\les J^*$
        $$ \lim_{n\to \infty} \Big( \|f_n\|_{H^1}^2 - \sum_{j=1}^J \| f^{(j)}\|_{H^1}^2 - \| w^{(J)}_n\|_{H^1}^2\Big) = 0 = \lim_{n\to \infty} \Big( \| g_n\|_{L^2}^2 - \sum_{j=1}^J \|g^{(j)}\|_{L^2}^2 - \| e^{(J)}_n\|_{L^2}^2 \Big)$$
      and
            $$ \lim_{n\to \infty} \Big( \mc{E}_Z(f_n, g_n) - \sum_{j=1}^J \mc{E}_Z\big(\mf{g}^{(j)}_n(f^{(j)}, g^{(j)})\big) - \mc{E}_Z(w^{(J)}_n, e^{(J)}_n) \Big) = 0. $$

      \item The free evolution of the error $(w^{(J)}_n, e^{(J)}_n)$ goes to zero in $Y\times Y$, thus
                $$ \lim_{J\to J^*} \limsup_{n\to \infty} \big( \| e^{it\Delta} w^{(J)}_n\|_{Y} + \| e^{it|\nabla|} e^{(J)}_n \|_{Y} \big) = 0. $$

      \item For any $j\not = k$, the group elements $\mf{g}_n^{(j)} = \mf{g}[t^{(j)}_n, x^{(j)}_n]$ satisfy the asymptotic orthogonality property
                    $$ \lim_{n\to \infty} \big( |t_n^{(j)} - t_n^{(k)}| + |x^{(j)}_n - x^{(k)}_n|\big) = 0.$$
            Moreover, we have the normalisation condition that either $t_n^{(j)}=0$ for all $n\in \NN$, or $|t_n^{(j)}|\to \infty$ as $n\to \infty$.
    \end{enumerate}
\end{theorem}
\begin{proof}
The proof is a minor modification of the profile decomposition of Bahouri-G\'erard \cite{Bahouri1999}, and thus we only give a sketch of the proof. A standard argument, see for instance \cite{Killip2013}, shows that it suffices to prove that if
	\begin{equation}\label{eqn:prof decomp:assump}
		\lim_{n\to \infty} \| (f_n, g_n)\|_{H^1 \times L^2} = A, \qquad \lim_{n\to \infty} \|(e^{it\Delta}f_n, e^{it|\nabla|}g_n)\|_{Y\times Y} = \epsilon>0,
	\end{equation}
then there exits a profile $(f^{(1)}, g^{(1)})\in H^1\times L^2$ and a sequence of group elements $\mf{g}_n= \mf{g}[t_n, x_n]$  such that,  after potentially taking a subsequence, the sequence $(\mf{g}_n)^{-1} (f_n, g_n)$  converges weakly to $(f^{(1)}, g^{(1)})$ in $H^1\times L^2$, we have the lower bound
		\begin{equation}\label{eqn:prof decomp:lb}
			\|(f^{(1)}, g^{(1)})\|_{H^1\times L^2} \gtrsim \epsilon,
		\end{equation}
either $t_n=0$ for all $n$ or $|t_n|\to \infty$, and the energy (and $L^2$ and $H^1$ norms) decouples
	\begin{equation}\label{eqn:prof decomp:energy}
		\begin{split}
		\lim_{n\to \infty}\Big( \mc{E}_Z[(f_n, g_n)] - \mc{E}_Z\big[ \mf{g}_n (f^{(1)}, g^{(1)})\big] - \mc{E}_Z\big[ (f_n, g_n) - \mf{g}_n(f^{(1)}, g^{(1)})\big] \Big) &= 0, \\
			\lim_{n\to \infty}\Big(\|f_n\|_{H^1}^2 - \| f^{(1)}\|_{H^1}^2 - \|f_n - e^{-it_n\Delta} \tau_{x_n}f^{(1)} \|_{H^1}^2\Big) &= 0, \\
			 \lim_{n\to \infty} \Big( \| g_n\|_{L^2}^2  - \|g^{(1)}\|_{L^2}^2 - \| g_n - e^{-it_n|\nabla|} \tau_{x_n} g^{(1)}\|_{L^2}^2 \Big) &= 0.
			\end{split}
	\end{equation}
To construct the profile $(f^{(1)}, g^{(1)})$ we observe that by definition there exists $(s_n, x_n)\in \RR\times \RR^4$ and $\lambda_n\in 2^\NN$ such that
					$$ \lambda_n^{-4} \big| \big( e^{is_n\Delta} P_{\lambda_n} f_n, e^{is_n|\nabla|}P_{\lambda_n}g_n\big)(x_n)\big| \g \frac{1}{2} \epsilon. $$
An application of Bernstein's inequality gives  $\lambda_n^{-2} \| (f_n, g_n) \|_{H^1\times L^2} \gtrsim \epsilon$,  and hence $\limsup_{n\to \infty} \lambda_n^2 \lesa \epsilon^{-1} A$. In particular, as $\lambda_n\in 2^\NN$, there exists $\lambda_*\in 2^\NN$ and a subsequence such that $\lambda_n = \lambda_*$ for all $n\in \NN$. Define the group element $\mf{g}_n = \mf{g}_n[t_n, x_n]$, where $t_n=s_n$ if $|s_n|\to \infty$, and otherwise $t_n=0$ for all $n\in \NN$. As $(\mf{g}_n)^{-1}(f_n, g_n)$ is a bounded sequence in $H^1\times L^2$, after potentially taking a further subsequence, there exists $(f^{(1)}, g^{(1)}) \in H^1\times L^2$ such that $(\mf{g}_n)^{-1}(f_n, g_n)$ converges weakly to $(f^{(1)}, g^{(1)})$ in $H^1\times L^2$. Letting $K_{\lambda_*}$ denote the kernel of the Fourier multiplier $P_{\lambda_*}$, we have by the continuity of the flow $t\mapsto (e^{it\Delta}, e^{it|\nabla|})$ on $H^1\times L^2$,
	\begin{align*}
 |(P_{\lambda_*}f^{(1)}, P_{\lambda_*} g^{(1)})(0)| &= \limsup_{n\to \infty} \Big| \int_{\RR^4} K_{\lambda_*}(-y) (\mf{g}_n)^{-1}(f_n, g_n)(y) dy \\
 	&= \lambda_*^4 \limsup_{n\to \infty}  \lambda_n^{-4} \big|\big( e^{is_n\Delta} P_{\lambda_n} f_n, e^{is_n |\nabla|} P_{\lambda_n} g_n\big)(x_n)\big| \g   \frac{1}{2} \lambda_*^4 \epsilon
 	\end{align*}
and hence Bernstein's inequality and fact that $\lambda_*\in 2^\NN$ gives the lower bound \eqref{eqn:prof decomp:lb}. The $L^2$ and $H^1$ decoupling in \eqref{eqn:prof decomp:energy} follows immediately from the fact that $e^{it_n|\nabla|} \tau_{-x_n} f_n$ converges weakly in $H^1$ to $f^{(1)}$, while $e^{it_n|\nabla|} \tau_{-x_n} g_n$ converges weakly to $g^{(1)}$, and noting that both the $H^1$ and $L^2$ norms are invariant under the action of $\mf{g}_n$. To verify the energy decoupling in \eqref{eqn:prof decomp:energy} we have to work slightly harder. In view of the $H^1$ and $L^2$ decoupling, it suffices to prove that
			$$ \lim_{n\to \infty} \Big|\int_{\RR^4} g_n |f_n|^2 -  e^{-it_n|\nabla|}\tau_{x_n} g^{(1)}\big|e^{-it_n\Delta}\tau_{x_n}f^{(1)}\big|^2 - \big(g_n - e^{-it_n|\nabla|}\tau_{x_n}g^{(1)}\big)\big|f_n - e^{-it_n\Delta}\tau_{x_n}f^{(1)} \big|^2 dx \Big| = 0.$$
If $|t_n|\to \infty$, then after approximating by smooth functions the limit follows from the dispersive decay of the wave and Schr\"odinger propagators. Thus we may assume $t_n=0$ and after a short computation via translation invariance and weak convergence, our goal is to now prove that
	\begin{equation}\label{eqn:prof decomp:goal}
		 \lim_{n\to \infty}  \int_{\RR^4} 2|\tau_{-x_n}g_n| |f^{(1)}| |\tau_{-x_n} f_n - f^{(1)}| + |g^{(1)}| |\tau_{-x_n}f_n + f^{(1)}| |\tau_{-x_n} f_n - f^{(1)}| dx = 0.
	\end{equation}
But this follows by noting that since $\tau_{-x_n} f_n$ bounded in $H^1$, by the Rellich-Kondrachov Theorem, we have $\|\tau_{-x_n}f_n - f^{(1)}\|_{L^2(\Omega)} \to 0$ for any compact $\Omega \subset \RR^4$. Hence limit follows by localising in space.
\end{proof}

Later we exploit the fact that asymptotically orthogonal nonlinear profiles only interact weakly.

\begin{lemma}[Orthogonal profiles interact weakly]\label{lem:nonlin asymp decoup}
Let $(t_n, x_n), (t_n', x_n')\in \RR^{1+4}$ be sequences such that
	$$\lim_{n\to \infty} \big( |t_n - t_n'| + |x_n - x_n'| \big) = \infty.$$
Then for any $V\in \underline{W}^0$ and $u, w \in S^\frac{1}{2}$ we have
		\begin{equation}\label{eqn:nonlin asymp decoup:lim}
            \lim_{n\to \infty} \big\| \mc{I}_0\big[ \Re(V_n) u_n \big]\big\|_{S^\frac{1}{2}} = \lim_{n\to \infty} \big\| \mc{J}_0\big[ |\nabla| ( \overline{w}_n  u_n )\big] \big\|_{W^0} = \lim_{n\to\infty} \| \overline{w}_n u_n \|_{L^1_t L^2_x} = 0
        \end{equation}
where we take
	$$ u_n(t,x) = u(t-t_n, x-x_n), \qquad V_n(t,x) = V(t-t_n', x-x_n'), \qquad w_n(t,x) = w(t-t_n', x -x_n'). $$
\end{lemma}
\begin{proof}
The proof is essentially the standard approximation argument to reduce to the $C^\infty_0$ case. We begin by observing that after translating in space-time, for all limits in \eqref{eqn:nonlin asymp decoup:lim} it is enough to consider the case $t_n'=x_n' = 0$. For the first limit, as
		$$ \lim_{R\to \infty}  \| P_{\g R} V\|_{\underline{W}^0} = \lim_{R\to \infty} \| P_{\g R} u \|_{S^\frac{1}{2}} = 0 $$
an application of Theorem \ref{thm:bi prev} shows that we only have to consider bounded frequencies. In particular, since
		$$\big\| \mc{I}_0\big[ \Re(V_{<R}) P_{<R} u_n \big]\big\|_{S^\frac{1}{2}} \lesa R^\frac{1}{2} \| V_{<R} (P_{<R}u)_n \|_{L^2_t L^{\frac{4}{3}}_x}  $$
it suffices to prove that for the bounded frequency contribution we have
		$$ \lim_{n\to \infty} \| V_{<R} P_{<R}u_n \|_{L^2_t L^{\frac{4}{3}}_x} = 0. $$
This can be reduced further after noting that as $u\in L^2_t L^4_x \subset S^\frac{1}{2}$ we have
		$$ \lim_{R' \to \infty} \big\| \ind_{\{|t|+|x|\g R'\}} P_{<R} u \big\|_{L^2_t L^4_x} = 0.  $$
Letting $B_{R'} = \{|t|+|x|<R'\}$ we have
	\begin{align*}
 \big\| V_{\les R}  \big(\ind_{B_{R'}} P_{\les R} u\big)_n \big\|_{L^2_t L^{\frac{4}{3}}_x}
 		&\lesa \big\| V_{\les R}(t,x) \ind_{B_{R'}}(t-t_n, x-x_n) \big\|_{L^2_t L^\frac{4}{3}_x} \| P_{\les R} u \|_{L^\infty_{t,x}} \\
 		&\lesa R^\frac{3}{2} (R')^{\frac{7}{3}} \| V_{\les R}(t,x) \ind_{B_{R'}}(t-t_n, x-x_n)\|_{L^2_t L^6_x}  \| u \|_{L^\infty_t H^\frac{1}{2}}
	\end{align*}
and hence it is enough to prove that
		$$ \lim_{n\to \infty} \| V_{\les R}(t,x) \ind_{B_{R'}}(t-t_n, x-x_n)\|_{L^2_t L^6_x} = 0.$$
But this is immediate via the dominated convergence theorem since $\| V_{\les R} \|_{L^2_t L^6_x} \lesa R^{\frac{5}{6}} \| V \|_{\underline{W}^0}$ which is a consequence of \eqref{eqn:wave stri}. Hence the first limit in \eqref{eqn:nonlin asymp decoup:lim} follows.

To prove the second limit in \eqref{eqn:nonlin asymp decoup:lim}, an analogous argument  via Theorem \ref{thm:bi prev} shows that again is suffices to consider bounded frequencies. After observing that an application of Bernstein's inequality gives
        $$ \| \big\| \mc{J}_0\big[ |\nabla| ( \overline{w_{\les R}}  P_{\les R} u_n )\big] \big\|_{W^0} \lesa R \| \overline{w_{\les R}} P_{\les R} u_n \|_{L^1_t L^2_x \cap L^2_t L^\frac{4}{3}_x},$$
and noting that $u, w \in L^2_t L^4_x \subset S^\frac{1}{2}$, it only remains to prove that
        $$ \lim_{n\to \infty} \| \overline{w} u_n \|_{L^1_t L^2_x \cap L^2_t L^\frac{4}{3}_x} = 0. $$
However this is a consequence of the fact that $w, u \in L^2_t L^4_x$ together with the argument used to prove the first limit in \eqref{eqn:nonlin asymp decoup:lim}. This also completes the proof of the final limit in \eqref{eqn:nonlin asymp decoup:lim}. 		
\end{proof}

\section{The Ground State Constraint}\label{sec:ground state constraint}

Define the functional
        $$\mc{K}(f) = \| f\|_{\dot{H}^1}^2 - \| f \|_{L^4}^4  = \frac{d}{d\lambda} \mc{E}_{NLS}(\lambda f) \big|_{\lambda = 1}.$$
This functional can be seen as the scaling derivative of the NLS energy $\mc{E}_{NLS}(f)$ and appears in the Virial identity for both the NLS and Zakharov equations, see for instance the discussion in \cite{Guo2021, Guo2013}. The functional $\mc{K}$ is also closely related to the Zakharov energy, for instance we have the identity
        \begin{equation}\label{eqn:scaling der of zak energy}
            \frac{d}{d\lambda} \mc{E}_Z(\lambda f, \lambda^2 g) = \lambda^{-1} \mc{K}(\lambda f) + \lambda^3 \big\| g + |f|^2 \big\|_{L^2_x}^2.
        \end{equation}
This identity is particularly useful as it shows that provided $\mc{K}(\lambda f) >0$, the energy increases along the curve $\lambda \mapsto (\lambda f, \lambda^2 g)$ in $\dot{H}^1 \times L^2$.

An application of H\"older and Sobolev embedding gives
        $$ 4 \mc{E}_Z(f,g) \les 2 \| f \|_{\dot{H}^1}^2 + \| g\|_{L^2}^2 + 2 \|Q\|_{\dot{H}^1}^{-1} \| g\|_{L^2} \| f \|_{\dot{H}^1}^2 \lesa \| f \|_{\dot{H}^1}^2 ( 1  + \| g \|_{L^2}) + \|g\|_{L^2}^2 $$
and thus the energy is always finite provided $(f,g)\in \dot{H}^1\times L^2$. The reverse inequality is false in general, in particular the energy is not necessarily positive.
However, if we impose a size constraint on the functions $(f,g) \in \dot{H}^1 \times L^1$, then the energy is always positive. A natural condition to ensure that the energy is coercive can be phrased in terms of the ground state (or Aubin-Talenti function) $Q(x) = ( 1 + \frac{|x|^2}{8})^{-1}$. Recall that the ground state $Q\in \dot{H}^1$ satisfies the properties
        $$ \Delta Q = -Q^3, \qquad  \| Q \|_{L^4}^2 = \|Q \|_{\dot{H}^1}, \qquad \| Q \|_{\dot{H}^1}^{-\frac{1}{2}} = \sup_{\| f \|_{\dot{H}^1}=1} \| f \|_{L^4}, \qquad \mc{E}_{Z}(Q, -Q^2) = \mc{E}_{NLS}(Q) = \frac{1}{4}\|Q\|_{\dot{H}^1}^2.$$
The properties of the ground state quickly give the implications
            \begin{equation}\label{eqn:imp energy pos wave}
                      \| g \|_{L^2}  \les \| Q \|_{\dot{H}^1} \qquad \Longrightarrow \qquad \| g \|_{L^2}^2 \les 4 \mc{E}_{Z}(f,g)
            \end{equation}
and
            \begin{equation}\label{eqn:imp energy pos schro}
                     \|f\|_{\dot{H}^1}\les \| Q \|_{\dot{H}^1} \qquad \Longrightarrow \qquad \|f\|_{\dot{H}^1}^2 \les 4 \mc{E}_{Z}(f,g).
            \end{equation}
More precisely, we simply note that the (sharp) Sobolev embedding gives
    \begin{align*}
      4\mc{E}_Z(f,g) &\g 2 \| f \|_{\dot{H}^1}^2 + \| g \|_{L^2}^2 - 2 \| g \|_{L^2} \| f\|_{L^4}^2 \g   2 \| f \|_{\dot{H}^1}^2 + \| g \|_{L^2}^2 - 2 \| Q \|_{\dot{H}^1}^{-1} \| g \|_{L^2} \| f\|_{\dot{H}^1}^2
    \end{align*}
and hence rearranging we have the lower bound
    \begin{equation}\label{eqn:energy lower bound}
    \begin{split}
        4\mc{E}_{Z}(f,g) &\g \| g \|_{L^2}^2 + 2 \| Q \|_{\dot{H}^1}^{-1} \| f \|_{\dot{H}^1}^2 \big( \| Q \|_{\dot{H}^1} - \| g \|_{L^2}\big) \\
                        &= \| f \|_{\dot{H}^1}^2 + \big( \| g \|_{L^2} - \| Q \|_{\dot{H}^1}^{-1} \| f \|_{\dot{H}^1}^2\big)^2 + \| Q \|_{\dot{H}^1}^{-2} \| f \|_{\dot{H}^1}^2 \big( \| Q \|_{\dot{H}^1}^2 - \| f\|_{\dot{H}^1}^2\big).
    \end{split}
    \end{equation}
from which the implications \eqref{eqn:imp energy pos wave} and \eqref{eqn:imp energy pos schro} easily follow. These implications can be improved if we assume that the energy  is at most the energy $\mc{E}_{Z}(Q, -Q^2) = \frac{1}{4} \| Q \|_{\dot{H}^1}^2$ of the ground state solution $(Q, -Q^2)$.

\begin{lemma}[Energy coercive below ground state \cite{Guo2021}]\label{lem:ground state}
Let $(f, g)\in \dot{H}^1\times L^2$ and suppose that
    $$\min\{\|f\|_{\dot{H}^1}, \| g \|_{L^2} \} \les \| Q \|_{\dot{H}^1}\qquad \text{ and } \qquad  4\mc{E}_z(f,g)\les \|Q\|_{\dot{H}^1}^2.$$
Then we have the improved bounds
        $$ \max\{\|f\|_{\dot{H}^1}^2, \| g \|_{L^2}^2 \}\les 4 \mc{E}_{Z}(f,g) \qquad \text{ and } \qquad \mc{K}(f) + \big\| g + |f|^2 \big\|_{L^2_x} \g  4\mc{E}_{Z}(f, g) \frac{\|Q\|_{\dot{H}^1}^2 - 4 \mc{E}_Z(f,g)}{2\|Q\|_{\dot{H}^1}^2}.  $$
\end{lemma}
\begin{proof}
This is essentially contained in \cite[Lemma 6.1]{Guo2021}, but we give a slightly more direct proof by arguing directly from the properties of the ground state function $Q$. More precisely, by rearranging the lower bound  \eqref{eqn:energy lower bound} and applying the assumption  $4\mc{E}_Z(f,g) \les \|Q\|_{\dot{H}^1}^2$ we have
        $$   \big( \| g \|_{L^2} - \| Q \|_{\dot{H}^1}^{-1} \| f \|_{\dot{H}^1}^2\big)^2 \les \| Q \|_{\dot{H}^1}^{-2} \big( \| Q \|_{\dot{H}^1}^2 - \| f \|_{\dot{H}^1}^2\big)^2 $$
and
        $$ 2 \| f \|_{\dot{H}^1}^2 \| Q \|_{\dot{H}^1}^{-1} \big( \| Q \|_{\dot{H}^1} - \| g \|_{L^2}\big) \les \| Q\|_{\dot{H}^1}^2 - \| g \|_{L^2}^2.$$
In  particular, a short computation shows provided $4\mc{E}_Z(f,g) \les \|Q\|_{\dot{H}^1}^2$  we have the implication
        $$ \min\{ \| f \|_{\dot{H}^1}, \| g \|_{L^2} \} \les \| Q \|_{\dot{H}^1} \qquad \Longrightarrow \qquad  \max\{ \| f \|_{\dot{H}^1}, \| g \|_{L^2} \} \les \| Q \|_{\dot{H}^1}. $$
In view of the implications \eqref{eqn:imp energy pos wave} and \eqref{eqn:imp energy pos schro}, this completes the proof of the first bound.

Finally, the second bound in the statement of the lemma follows by observing that \eqref{eqn:energy lower bound} and $\| f\|_{\dot{H}^1}\les 4 \mc{E}_Z(f,g)$ in fact implies the slightly sharper bound
        $$ \| f\|_{\dot{H}^1}^2 \les \frac{\|Q\|_{\dot{H}^1}^2}{2\|Q\|_{\dot{H}^1}^2 - \| f\|_{\dot{H}^1}^2} 4 \mc{E}_Z(f,g) \les \frac{\|Q\|_{\dot{H}^1}^2}{2\|Q\|_{\dot{H}^1}^2 - 4\mc{E}_Z(f,g)} 4 \mc{E}_Z(f,g)  $$
and hence
          \begin{align*}
            \mc{K}(f) + \big\| g + |f|^2 \big\|_{L^2_x} = 4 \mc{E}_Z(f,g)  - \| f\|_{\dot{H}^1}^2
                &\g 4 \mc{E}_{Z}(f,g) \Big( \frac{\|Q\|_{\dot{H}^1}^2  - 4\mc{E}_Z(f,g)}{2\|Q\|_{\dot{H}^1}^2 - 4\mc{E}_Z(f,g)}\Big)
          \end{align*}
which clearly suffices as $\mc{E}_Z(f,g) \g 0$.
\end{proof}

\begin{remark}[Characterising ground state condition]\label{rem:ground state cond}
As observed in \cite{Guo2021}, there are a number of ways to characterise the ground state condition in Lemma \ref{lem:ground state}. For instance, as long as the energy is below the energy of the ground state, the sign of $\mc{K}(f)$ can be used to determine the coercivity of the energy for both $\mc{E}_Z$ and $\mc{E}_{NLS}$. This is well known in the case of the NLS \cite{Kenig2006}. In fact, for any $f\in \dot{H}^1$ with $4\mc{E}_{NLS}(f)<\| Q \|_{\dot{H}^1}$ we have the implications
    \begin{equation}\label{eqn:NLS imp1}
        \mc{K}(f)> 0 \,\,\,\, \Longleftrightarrow \,\,\,\, 0<\|f\|_{L^4}<\|Q\|_{L^4} \,\,\,\, \Longleftrightarrow \,\,\,\, 0<\| f \|_{\dot{H}^1} < \| Q \|_{\dot{H}^1} \,\,\,\, \Longleftrightarrow \,\,\,\, \| f \|_{\dot{H}^1} < 4 \mc{E}_{NLS}(f)
    \end{equation}
and
    \begin{equation}\label{eqn:NLS imp2}
        \mc{K}(f)= 0 \qquad \Longleftrightarrow \qquad f=0.
    \end{equation}
As in the proof of Lemma \ref{lem:ground state}, the implications \eqref{eqn:NLS imp1} and \eqref{eqn:NLS imp2} follow from the properties of the ground state solution, together with the identity
    $$ 4 \mc{E}_{NLS}(f) = 2 \| f \|_{\dot{H}^1}^2 - \| f \|_{L^4}^4 =  \| f \|_{\dot{H}^1}^2 +  \mc{K}(f). $$
A similar characterisation holds in the case of the Zakharov equation. In fact since $\mc{E}_{NLS}(f)\les \mc{E}_Z(f,g)$,  somewhat trivially, the implications \eqref{eqn:NLS imp1} and \eqref{eqn:NLS imp2} also hold for any $(f,g)\in \dot{H}^1\times L^2$ with $4\mc{E}_Z(f,g)< \|Q\|_{\dot{H}^1}^2$.

On the other hand, the ground state condition can also be characterised using the wave data $g\in L^2$ \cite[Lemma 6.1]{Guo2021}. More precisely for any $(f, g) \in H^1\times L^2$ with $4\mc{E}_Z(f,g)<\|Q\|_{\dot{H}^1}^2$ we have
	\begin{equation}\label{eqn:Zak imp1}
			\mc{K}(f)\g 0,  \qquad \Longleftrightarrow \qquad \|g\|_{L^2}< \|Q\|_{\dot{H}^1}\quad \Longleftrightarrow \quad \|g\|_{L^2}\les \mc{E}_Z(f,g)
    \end{equation}
and
    \begin{equation}\label{eqn:Zak imp2}
        \mc{K}(f)= 0 \qquad \Longleftrightarrow \qquad f=0.
    \end{equation}
The implications \eqref{eqn:Zak imp1} and \eqref{eqn:Zak imp2} follow directly from Lemma \ref{lem:ground state} together with the Schr\"odinger counterparts \eqref{eqn:NLS imp1} and \eqref{eqn:NLS imp2}. Note that the lack of strict inequalities in \eqref{eqn:Zak imp1} is a consequence of the fact that
		$$ 4 \mc{E}_Z(0, g) = \|g\|_{L^2}, \qquad 4\mc{E}_Z(f, 0) = 2\|f\|_{\dot{H}^1}^2 $$
and in particular, $f=0$ (or $g=0$) does not necessarily imply that $\mc{E}_Z(f,g)=0$.
\end{remark}

\section{A Palais-Smale Type Condition}\label{sec:P-S}

In this section our goal is to apply the results obtained in the previous section to show that show that any bounded sequence of solutions to the Zakharov equation for which the dispersive norm $\|\cdot\|_D \to \infty$ and lie below the ground state solution, must by precompact modulo translations. This type of result is the key step in the proof of Theorem \ref{thm:min coun}. The arguments used in this section are largely adapted from \cite{Keraani2006, Kenig2006, Killip2013}. The key point is that if a mass/energy threshold exists (see Definition \ref{defn:mass/energy}), then via the profile decomposition in Theorem \ref{thm:profile decomp}, together with the well-posedness theory in Section \ref{sec:small data II}, we can extract some compactness from any sequence of solutions approaching the critical threshold.

\begin{theorem}[Palais-Smale type condition]\label{thm:Palais-Smale}
Let $(M_c, E_c)$ be a mass/energy threshold  with $4E_c<\|Q\|_{\dot{H}^1}^2$. Suppose that $(u_n, V_n) \in C(\RR, H^1\times L^2)$ is a sequence of global solutions to \eqref{eqn:Zak} such that
    \begin{equation}\label{eqn:P-S:assump energy}
       u_n\in L^2_{t,loc}W^{\frac{1}{2}, 4}_x, \qquad  \lim_{n\to \infty} \mc{E}_Z(u_n, V_n) = E_c, \qquad \lim_{n\to \infty} \mc{M}(u_n) =  M_c, \qquad \sup_{n\in \NN} \| V_n(0) \|_{L^2_x} \les \| Q \|_{\dot{H}^1}
    \end{equation}
and
    \begin{equation}\label{eqn:P-S:assump non-disp}
        \lim_{n\to \infty} \| u_n \|_{D([0, \infty))} = \lim_{n\to \infty} \|u_n \|_{D((-\infty, 0])} = \infty.
    \end{equation}
Then there exists $x_n\in \RR$ such that the translated sequence  $(u_n, V_n)(0, x+x_n)$ has a convergent subsequence in $H^1 \times L^2$.
\end{theorem}
\begin{proof}
Define $(f_n, g_n)=(u_n, V_n)(0) \in H^1\times L^2$. The first step is to verify that the sequence $(f_n, g_n)$ is bounded in $H^1\times L^2$. Since $\lim_{n\to \infty} 4\mc{E}_Z(f_n, g_n) = 4E_c<\|Q\|_{\dot{H}^1}$, the assumption \eqref{eqn:P-S:assump energy} implies that for all sufficiently large $n$ we have
        $$ 4 \mc{E}_Z(f_n, g_n) < \|Q\|_{\dot{H}^1}^2, \qquad \|g_n\|_{L^2}\les \|Q\|_{\dot{H}^1}. $$
Consequently, the variational properties of the ground state (see Lemma \ref{lem:ground state}) give the upper bounds
    \begin{equation}\label{eqn:P-S:seq bdd}
			\limsup_{n\to \infty} \|f_n\|_{\dot{H}^1}^2 \les 4E_c, \qquad \limsup_{n \to \infty} \|g_n\|_{L^2}^2 \les 4 E_c.
	  \end{equation}
Together with the assumed boundedness of the mass, we conclude that $\sup_{n\in \NN} \|(f_n, g_n)\|_{H^1\times L^2}<\infty$. In other words the sequence $(f_n, g_n)$ is a bounded sequence in $H^1\times L^2$.

We now apply the profile decomposition in Theorem \ref{thm:profile decomp} to the sequence $(f_n, g_n)$, and obtain $J^* \in \NN \cup \{\infty\}$ and for each $1\les j \les J^*$ group elements $\mf{g}^{(j)}_n=\mf{g}[t_n^{(j)}, x_n^{(j)}]$ and profiles $(f^{(j)}, g^{(j)})\not = (0, 0)$ such that (after replacing $(f_n, g_n)$ with a suitable subsequence) for any $0\les J \les J^*$ we can write
        $$ (f_n, g_n) = \sum_{j=1}^J \mf{g}^{(j)}_n(f^{(j)}, g^{(j)}) + (w^{(J)}_n, e^{(J)}_n) $$
where the profiles and errors $(w^{(J)}_n, e^{(J)}_n)$ satisfy the conditions (i), (ii), and (iii) in the statement of Theorem \ref{thm:profile decomp}. In particular, we have the $L^2$ decoupling of the wave profiles
        \begin{equation}\label{eqn:P-S:wave decoup}
\sum_{j=1}^J \| g^{(j)}\|_{L^2}^2 + \limsup_{n\to \infty} \| e^{(J)}_n\|_{L^2}^2 \les \limsup_{n\to \infty} \|g_n(0)\|_{L^2}^2 \les 4E_c
			\end{equation}
the decoupling of the energy
			\begin{equation}\label{eqn:P-S:energy decoup}
				\limsup_{n\to \infty} \Big[ \sum_{j=1}^J \mc{E}_Z\big(\mf{g}^{(j)}_n(f^{(j)}, g^{(j)})\big) + \mc{E}_Z(w^{(J)}_n, e^{(J)}_n)\Big] \les \lim_{n\to \infty} \mc{E}_Z(f_n, g_n)(0) = E_c
			\end{equation}
and the decoupling of the Schr\"odinger mass
        \begin{equation}\label{eqn:P-S:mass decoup}
               \sum_{j=1}^J \mc{M}(f^{(j)}) + \limsup_{n\to\infty} \mc{M}(w^{(J)}_n) \les \lim_{n\to \infty} \mc{M}(f_n) = M_c.
        \end{equation}
The above limits quickly imply that each profile has energy below the ground state. More precisely, as $4E_c< \|Q\|_{\dot{H}^1}^2$, the $L^2$ decoupling \eqref{eqn:P-S:wave decoup} implies that for every $1\les j \les J$ we have
$\| g^{(j)}\|_{L^2} <\|Q\|_{\dot{H}^1}$, and for all sufficiently large $n$, the error satisfies $\|e^{(J)}_n\|_{L^2}<\|Q\|_{\dot{H}^1}$. Hence the implication \eqref{eqn:imp energy pos wave} implies that
    $$ \limsup_{n\to \infty} 4 \mc{E}_Z\big(\mf{g}^{(j)}_n(f^{(j)}, g^{(j)})\big) \g \| g^{(j)}\|_{L^2}^2 \qquad \text{ and } \qquad \limsup_{n\to\infty} 4\mc{E}_Z(w^{(J)}_n, e^{(J)}_n) \g \limsup_{n\to \infty} \|e^{(J)}_n\|_{L^2}^2. $$
Consequently, as each profile is non-zero, energies of each of the (translated) profiles $\mf{g}^{(j)}_n(f^{(j)}, g^{(j)})$ must be strictly positive. Together with \eqref{eqn:P-S:mass decoup}, we conclude that both the profiles and the error term have energy below the threshold. Namely  we have the bounds
    \begin{equation}\label{eqn:P-S:energy profile}
       0\les \mc{M}(f^{(j)})\les M_c, \qquad 0<\limsup_{n\to \infty} \mc{E}_Z\big(\mf{g}^{(j)}_n(f^{(j)}, g^{(j)})\big) \les E_c, \qquad \|g^{(j)}\|_{L^2}\les 4E_c,
    \end{equation}
and
    \begin{equation}\label{eqn:P-S:energy error}
        \limsup_{n\to \infty} \mc{E}_Z\big(w^{(J)}_n, e^{(J)}_n\big) \les E_c, \qquad \limsup_{n\to\infty} \|e^{(J)}_n\|_{L^2}\les 4E_c.
    \end{equation}
Moreover, the $H^1$ decoupling of the profiles together with \eqref{eqn:P-S:seq bdd} gives the upper bound
	\begin{equation}\label{eqn:P-S:H1 decoup}
		\sum_{j=1}^J \| f^{(j)} \|_{H^1}^2 + \limsup_{n\to\infty} \| w^{(J)}_n \|_{H^1}^2
						\les \limsup_{n\to \infty} \| f_n \|_{H^1}^2 \les M_c + 4E_c.
	\end{equation}

The next step is to evolve the profiles $(f^{(j)}, g^{(j)})$ via the Zakharov equation. More precisely, in view of the energy constraint \eqref{eqn:P-S:energy profile} and the assumption $4E_c < \|Q\|_{\dot{H}^1}^2$, we can apply Theorem \ref{thm:gwp below ground state} (if $t_n=0$) or Theorem \ref{thm:wave operators} (if $t_n\to \pm\infty$, potentially after reflecting in time) and obtain global solutions $(u^{(j)}, V^{(j)}) \in C(\RR; H^1\times L^2)$  to  \eqref{eqn:Zak} satisfying
        $$ \lim_{n\to \infty} \big\| (u^{(j)},V^{(j)})(-t_n) - \big( e^{-it_n\Delta} f^{(j)}, e^{-it_n|\nabla|}g^{(j)}\big) \big\|_{H^1\times L^2} = 0 $$
with mass $\mc{M}(u^{(j)}) = \mc{M}(f^{(j)}) \les M_c$ and energy
        \begin{equation}\label{eqn:P-S:energy nonlinear profiles}
        \begin{split}
                 \mc{E}_Z(u^{(j)}, V^{(j)}) = \lim_{n\to \infty} \mc{E}_Z\big( e^{-it_n\Delta} f^{(j)}, e^{-it_n|\nabla|}g^{(j)}\big) &=  \lim_{n\to \infty} \mc{E}_Z\big( \mf{g}^{(j)}_n(f^{(j)}, g^{(j)})\big)\les E_c, \\
                \max\big\{ \|u^{(j)}\|_{L^\infty_t \dot{H}^1_x}, \| V^{(j)}\|_{L^\infty_t L^2_x} \big\} &\les 4 \mc{E}_Z(u^{(j)}, V^{(j)}).
        \end{split}
        \end{equation}
We now translate in space-time, and define the translated solutions
        $$ u_n^{(j)}(t,x)=u^{(j)}(t-t_n, x-x_n), \qquad V^{(j)}_n(t,x) = V^{(j)}(t-t_n, x-x_n)$$
the key point being that we now have
       \begin{equation}\label{eqn:P-S:data nonlinear prof}
         \lim_{n\to \infty} \big\|\big(u_n^{(j)}, V_n^{(j)}\big)(0) - \mf{g}^{(j)}_n(f^{(j)}, g^{(j)})\|_{H^1\times L^2} = 0.
       \end{equation}
We now consider two cases, either there exists a profile with energy precisely $E_c$, or all profiles have energy strictly below $E_c$. In the former case, the decoupling inequalities \eqref{eqn:P-S:wave decoup}, \eqref{eqn:P-S:energy decoup}, and \eqref{eqn:P-S:mass decoup}, imply that there is only one profile and moreover the error goes to zero in $\dot{H}^1\times L^2$. Upgrading the convergence to $H^1\times L^2$ relies on the fact that $(M_c, E_c)$ is a mass/energy threshold together with a short argument via Theorem \ref{thm:stab II}. In the later case where the initial profile has energy strictly less than $E_c$, we show that each profile is uniformly bounded in the dispersive norm $\|\cdot\|_D$, and moreover only interacts weakly. Applying the stability result in Theorem \ref{thm:stab II} we eventually conclude that $\limsup_{n\to \infty} \|u_n\|_D<\infty$ which contradicts our initial assumption \eqref{eqn:P-S:assump non-disp}.  \\

\textbf{Case 1:} $\mc{E}_Z( u^{(1)}, V^{(1)}) = E_c$. \\

Our goal is to show that if $\mc{E}_Z(u^{(1)}, V^{(1)})=E_c$, then we have
		\begin{equation}\label{eqn:P-S:case 1 conclusion}
				\lim_{n\to \infty} \big\| (f_n, g_n)(x) - (f^{(1)}, g^{(1)})(x-x_n)\big\|_{H^1\times L^2}=0.		\end{equation}
We start by observing that the decoupling of the energy \eqref{eqn:P-S:energy decoup} together with \eqref{eqn:P-S:energy nonlinear profiles} implies that for all $1<j\les J$
        $$ \mc{E}_Z(u^{(j)}, V^{(j)}) =  0 = \limsup_{n\to \infty} \mc{E}_Z(w^{(J)}_n, e^{(J)}_n). $$
Hence, as energies of all profiles lies below the ground state solution (namely \eqref{eqn:P-S:energy profile} holds), an application of Lemma \ref{lem:ground state} gives
        $$ \big\| \big( u^{(j)}, V^{(j)}\big)\big\|_{\dot{H}^1\times L^2}^2 \les 4 \mc{E}_Z(u^{(j)}, V^{(j)}) = 0 $$
and
        $$ \limsup_{n\to \infty} \|(w^{(J)}_n, e^{(J)}_n)\|_{\dot{H}^1\times L^2}^2 \les 4 \limsup_{n\to \infty} \mc{E}_Z(w^{(J)}_n, e^{(J)}_n) = 0. $$
In other words there is only one profile, and moreover the error goes to zero in $\dot{H}^1\times L^2$.

This is almost what we want, but to obtain \eqref{eqn:P-S:case 1 conclusion} we need to upgrade the convergence to $H^1\times L^2$. To this end, suppose for the moment we have $\mc{M}(u^{(1)})\les M_c - \delta$ for some $\delta>0$. Since $(M_c, E_c)$ is a mass/energy threshold, we then see that
        $$ \| u^{(1)}\|_{D} \les L(M_c - \delta, E_c) < \infty. $$
In particular, applying translation invariance, Theorem \ref{thm:persistence}, and collecting the above bounds/limits, we have
        $$ \limsup_{n\to \infty} \Big(\|f_n\|_{H^1} + \| u^{(1)}_n \|_{S^\frac{1}{2}} \Big) \lesa_{L, M_c, E_c} 1, \qquad \limsup_{n\to\infty} \|V^{(1)}_n\|_{L^\infty_t L^2_x} \les 4E_c < \|Q\|_{\dot{H}^1}^2$$
and
        $$ \lim_{n\to \infty} \Big( \big\|f_n - u_n^{(1)}(0) \big\|_{\dot{H}^1} + \big\| g_n - V^{(1)}_n(0) \big\|_{L^2} \Big) = 0. $$
Thus in view of \eqref{eqn:mod stab error cond}, an application of Theorem \ref{thm:stab II} with $F=G=0$ implies the uniform dispersive bound $\limsup_{n\to \infty} \| u_n \|_D < \infty$ which clearly contradicts the assumption \eqref{eqn:P-S:assump non-disp}. Therefore it is not possible for both the dispersive norm to blow-up and the mass of the profile $(u^{(1)}, V^{(1)})$ to remain strictly less than $M_c$. In other words, we must have the mass constraint $\mc{M}(u^{(1)})=M_c$. In view of \eqref{eqn:P-S:mass decoup}, we then conclude that $\limsup_{n\to \infty} \mc{M}(w^{(J)}_n) = 0$. Consequently, unpacking the definition of $u^{(1)}_n$, we have the $H^1\times L^2$ limit
        $$\lim_{n\to \infty} \big\| (u_n, V_n)(0, x) - \big( e^{-it_n^{(1)}\Delta} f^{(1)}(x-x_n^{(1)}), e^{-it_n^{(1)}|\nabla|} g^{(1)}(x-x_n^{(1)})\big) \big\|_{H^1\times L^2} = 0. $$

To complete the proof of \eqref{eqn:P-S:case 1 conclusion}, it only remains to rule out the case $t_n^{(1)} \to - \infty$ (the case $t_n^{(1)}\to \infty$ would then also be excluded by time reversibility).  We start by noting that since $t_n^{(1)} \to -\infty$, the dispersive decay of the free Schr\"odinger and wave evolutions implies that
    $$ \lim_{n\to \infty} \big\| \mf{g}^{(1)}_n\big(e^{it\Delta} f, e^{it|\nabla|}g\big)\big\|_{Y\times Y([0, \infty))} = \lim_{n\to \infty}  \big\| \big(e^{it\Delta} f, e^{it|\nabla|}g\big)\big\|_{Y \times Y([-t_n, \infty))} = 0.$$
Therefore, applying the bound \eqref{eqn:Y norm bound by sobolev}, we have
    \begin{align*}
    \limsup_{n\to \infty}  \big\| \big( e^{it\Delta} f_n, e^{it|\nabla|}g_n\big)\big\|_{Y \times Y([0, \infty))}
            &\les  \limsup_{n\to \infty} \big\| \big( e^{it\Delta} f_n, e^{it|\nabla|}g_n\big) - \mf{g}^{(1)}_n( e^{it\Delta} f^{(1)}, e^{it|\nabla|} g^{(1)})\big\|_{Y \times Y([0, \infty))} \\
             &\qquad + \limsup_{n\to \infty}\| \mf{g}^{(1)}_n( e^{it\Delta} f^{(1)}, e^{it|\nabla|} g^{(1)})\|_{Y\times Y([0, \infty))} \\
            &\lesa \limsup_{n\to \infty} \| (f_n, g_n) - \mf{g}^{(1)}_n(f^{(1)}, g^{(1)})\|_{H^1\times L^2} = 0.
    \end{align*}
Consequently, we can apply the small data theory in Theorem \ref{thm:small data Hs} to the interval $[0, \infty)$, and conclude that $\limsup_n \| u_n \|_{D([0, \infty))}< \infty$. But this contradicts the assumption \eqref{eqn:P-S:assump non-disp} and hence we cannot have $t_n^{(1)} \to \pm \infty$. \\

\textbf{Case 2:} $\mc{E}_Z( u^{(1)}, V^{(1)}) <  E_c $. \\

As all profiles are non-zero and have positive energy, we conclude from \eqref{eqn:P-S:energy decoup}, \eqref{eqn:P-S:energy profile}, and \eqref{eqn:P-S:energy error} that we have the mass/energy bounds
            $$ \sup_{1\les j \les J^*}\mc{M}(u^{(j)})\les M_c, \qquad \sup_{1\les j \les J^*} \mc{E}_Z\big( u^{(j)}, V^{(j)}\big) \les E_c - \delta$$
for some $\delta>0$. In particular, in view of the definition of $E_c$, we have the global dispersive bound
            $$
                     \sup_{1\les j \les J^*} \| u^{(j)} \|_D \les L =  L(M_c, E_c - \delta)<\infty.
            $$
To upgrade this bound to control over the $\underline{S}^s$ norm, we first note that the bound \eqref{eqn:P-S:energy nonlinear profiles} together with the conservation of mass gives
        $$ \|u^{(j)}\|_{L^\infty_t H^1_x} \les M_c + 4E_c, \qquad \| V^{(j)}\|_{L^\infty_t L^2_x} \les 4E_c < \| Q \|_{\dot{H}^1}. $$
Consequently an application of Theorem \ref{thm:persistence} implies that for any $\frac{1}{2}\les s < 1$ we have
        \begin{equation}\label{eqn:P-S:uniform S bound}
                 \| u^{(j)} \|_{\underline{S}^s} \lesa_{s, L, E_c, M_c} \|f^{(j)}\|_{H^s}, \qquad \|V^{(j)}\|_{\underline{W}^0} \lesa_{s, L, E_c, M_c} \|g^{(j)}\|_{L^2} + \| f^{(j)} \|_{H^\frac{1}{2}}^2
        \end{equation}
where the implied constants depend only on $s$, $L$, $E_c$, and $M_c$. We now define
    $$ \Psi^{(J)}_n = \sum_{j=1}^J u_n^{(j)} + e^{it\Delta} w^{(J)}_n, \qquad \Phi^{(J)}_n = \sum_{j=1}^J V^{(j)}_n + e^{it|\nabla|} e^{(J)}_n. $$
We claim that we have the properties:
    \begin{enumerate}
        \item (Data agrees asymptotically) 
        For every $0\les J \les J^*$ we have
                $$ \lim_{n\to \infty} \big\| (f_n, g_n) - (\Psi^{(J)}_n, \Phi_n^{(J)})(0) \big\|_{H^1\times L^2} = 0, \qquad \limsup_{n\to \infty} \| \Phi^{(J)}_n\|_{L^\infty_t L^2_x}\les 4E_c < \|Q\|_{\dot{H}^1}.  $$

        \item (Uniformly bounded in $\underline{S}^s\times \underline{W}^0$) For every $\frac{1}{2}\les s<1$ we have
                            $$ \sup_{0\les J \les J^*} \limsup_{n\to \infty} \big\|(\Psi^{(J)}_n, \Phi^{(J)}_n)\big\|_{\underline{S}^s\times \underline{W}^0}  \lesa_{ s, L, E_c,  M} 1. $$

        \item (Approximate solution) We have
                $$ \lim_{J\to J^*} \limsup_{n\to \infty}\Big\| \mc{I}_0\Big[ \Re(\Phi^{(J)}_n) \Psi^{(J)}_n  - \sum_{j=1}^J \Re(V^{(j)}_n) u^{(j)}_n\Big] \Big\|_{S^\frac{1}{2}} = 0$$
        and
                $$ \lim_{J\to J^*} \limsup_{n\to \infty} \Big\| |\nabla| \mc{J}_0\Big[|\Psi^{(J)}_n|^2 - \sum_{j=1}^J |u_n^{(j)}|^2 \Big]\Big\|_{W^0} = 0. $$
    \end{enumerate}
Assuming these properties hold for the moment, an application of the stability result in Theorem \ref{thm:stab II} together with \eqref{eqn:mod stab error cond} implies that we have the global bound $\limsup_{n\to \infty} \|u_n \|_D <\infty$. But this contradicts the assumption that $\|u_n\|_{D([0, \infty))} \to \infty$ as $n\to \infty$. Hence Case 2 cannot occur.

It only remains to verify the properties (i), (ii), and (iii). The property (i) follows immediately from the construction of the profiles $(f^{(j)}, g^{(j)})$. To prove (ii), we start by observing that \eqref{eqn:P-S:uniform S bound} together with \eqref{eqn:P-S:wave decoup} and \eqref{eqn:P-S:H1 decoup} implies that
    $$\sup_{0\les J \les J^*} \Big( \sum_{j=1}^J \|u^{(j)}\|_{S^s}^2\Big)^\frac{1}{2} \lesa_{s, L, E_c, M_c} \sup_{0\les J \les J^*}\Big( \sum_{j=1}^J \|f^{(j)}\|_{H^s}^2\Big)^\frac{1}{2} \lesa_{s, L, E_c, M_c} 1$$
and
    $$\sup_{0\les J \les J^*} \Big( \sum_{j=1}^J \|V^{(j)}\|_{W^0}^2\Big)^\frac{1}{2} \lesa_{s, L, E_c, M_c} \sup_{0\les J \les J^*}\Big[ \Big( \sum_{j=1}^J \|g^{(j)}\|_{L^2}^2\Big)^\frac{1}{2} +  \sum_{j=1}^J \|f^{(j)}\|_{H^s}^2 \Big]\lesa_{s, L, E_c, M_c} 1. $$
Therefore, (i) and the bound \eqref{eqn:P-S:seq bdd}, together with an application of the energy inequality \eqref{eqn:energy ineq} and the bilinear estimate in Theorem \ref{thm:bi prev} gives
    \begin{align*}
     \sup_{0\les J \les J^*} \limsup_{n\to \infty} \big\| \Psi^{(J)}_n \big\|_{\underline{S}^s} &\lesa \sup_{0\les J \les J^*} \Big[ \limsup_{n\to \infty} \big\| \Psi^{(J)}_n(0) \big\|_{H^s} + \limsup_{n\to \infty} \Big\| \sum_{j=1}^J \Re( V^{(j)}_n) u^{(j)}_n \Big\|_{N^s} \Big]\\
                                                  &\lesa \sup_{0\les J \les J^*}\Big[  \limsup_{n\to \infty}  \big\| \Psi^{(J)}_n(0) \big\|_{H^s} +  \Big( \sum_{j=1}^J \|V^{(j)}\|_{W^0}^2\Big)^\frac{1}{2} \Big( \sum_{j=1}^J \|u^{(j)}\|_{S^s}^2\Big)^\frac{1}{2}\Big] \\
                                                  &\lesa_{s, L, E_c, M_c} 1
    \end{align*}
where we used the fact that the norms $\| \cdot \|_{W^0}$ and $\| \cdot \|_{\underline{S}^s}$ are translation invariant. Similarly, to bound the wave contribution we have
    \begin{align*}
     \sup_{0\les J \les J^*} \limsup_{n\to \infty} \big\| \Phi^{(J)}_n \big\|_{\underline{W}^0} &\lesa \sup_{0\les J \les J^*} \Big[ \limsup_{n\to \infty} \big\| \Phi^{(J)}_n(0) \big\|_{L^2} + \sum_{j=1}^J \Big\| \mc{J}_0\big[ |\nabla| |u^{(j)}|^2 \big] \Big\|_{\underline{W}^0}\Big] \\
                                                  &\lesa \sup_{0\les J \les J^*}\Big[  \limsup_{n\to \infty}  \big\| \Phi^{(J)}_n(0) \big\|_{L^2_x} +   \sum_{j=1}^J \|u^{(j)}\|_{S^s}^2\Big] \\
                                                  &\lesa_{s, L, E_c, M_c} 1
    \end{align*}
and hence (ii) follows.

Finally to prove (iii), we note that provided $s>\frac{1}{2}$, (ii) together with \eqref{eqn:P-S:wave decoup},  \eqref{eqn:P-S:H1 decoup}, and an application of Theorem \ref{thm:bi est improved} gives $\theta>0$ such that
\begin{align*}
      \big\| \mc{I}_0\big[ \Re(\Phi^{(J)}_n & - e^{it|\nabla|} e^{(J)}_n) e^{it\Delta} w^{(J)}_n \big] \big\|_{S^\frac{1}{2}} + \big\| \mc{I}_0\big[ \Re(e^{it|\nabla|} e^{(J)}_n) \Psi^{(J)}_n\big] \big\|_{S^\frac{1}{2}} \\
                &\lesa  \Big(\|\Phi^{(J)}_n\|_{\underline{W}^0} + \| e^{(J)}_n \|_{L^2} \Big)\| e^{it\Delta} w^{(J)}_n\|_{Y}^\theta \|w^{(J)}_n\|_{H^1}^{1-\theta} + \| e^{it|\nabla|}e^{(J)}_n\|_Y^\theta \| e^{(J)}_n \|_{L^2}^{1-\theta} \| \Psi^{(J)}_n \|_{\underline{S}^s} \\
                &\lesa_{s, L, E_c, M_c} \| e^{it\Delta} w^{(J)}_n\|_{Y}^\theta + \| e^{it|\nabla|}e^{(J)}_n\|_Y^\theta.
    \end{align*}
Therefore, the asymptotic decoupling provided by Lemma \ref{lem:nonlin asymp decoup} and the fact that the error vanishes in the dispersive norm $Y$ implies that
    \begin{align*}
      \limsup_{J\to J^*} &\limsup_{n\to \infty} \Big\| \mc{I}_0\Big[ \Re(\Phi^{(J)}_n) \Psi^{(J)}_n  - \sum_{j=1}^J \Re(V^{(j)}_n) u^{(j)}_n\Big] \Big\|_{S^\frac{1}{2}}\\
                &\lesa \limsup_{J\to J^*}\limsup_{n\to \infty} \Big( \sum_{\substack{ 1\les j, k \les J \\ j\not = k} } \big\| \mc{I}_0\big[ \Re(V^{(j)}_n) u^{(k)}_n\big] \big\|_{S^\frac{1}{2}} + \big\| \mc{I}_0\big[ \Re(\Phi^{(J)}_n - e^{it|\nabla|} e^{(J)}_n) e^{it\Delta} w^{(J)}_n \big] \big\|_{S^\frac{1}{2}} \\
                &\qquad \qquad \qquad \qquad+ \big\| \mc{I}_0\big[ \Re(e^{it|\nabla|} e^{(J)}_n) \Psi^{(J)}_n\big] \big\|_{S^\frac{1}{2}} \Big) \\
                &\lesa_{s, L, E_c, M_c} \limsup_{J\to J^*}\limsup_{n\to \infty} \Big( \sum_{\substack{ 1\les j, k \les J \\ j\not = k} } \big\| \mc{I}_0\big[ \Re(V^{(j)}_n) u^{(k)}_n\big] \big\|_{S^\frac{1}{2}} + \| e^{it\Delta} w^{(J)}_n\|_{Y}^\theta + \| e^{it|\nabla|}e^{(J)}_n\|_Y^\theta \Big)  = 0.   \end{align*}
Similarly, by another application of Theorem \ref{thm:bi est improved} we see that
    \begin{align*}
      \big\| |\nabla| \mc{J}_0\big[ \Re\big( \overline{\Psi}^{(J)}_n e^{it\Delta} w^{(J)}_n \big)\big]\big\|_{W^0} + &\big\| |\nabla| \mc{J}_0\big[ | e^{it\Delta} w^{(J)}_n|^2 \big]\big\|_{W^0}\\
                &\lesa \| e^{it\nabla} w^{(J)}_n \|_Y^\theta \| w^{(J)}_n \|_{H^1}^{1-\theta} \Big( \|\Psi^{(J)}_n\|_{\underline{S}^s} + \| w^{(J)}_n\|_{H^1}\Big) \lesa_{s, L, E_c, M_c}  \| e^{it\nabla} w^{(J)}_n \|_Y^\theta
    \end{align*}
and hence
    \begin{align*}
    \limsup_{J\to J^*}& \limsup_{n\to \infty} \Big\| |\nabla| \mc{J}_0\Big[|\Psi^{(J)}_n|^2 - \sum_{j=1}^J |u_n^{(j)}|^2 \Big]\Big\|_{W^0}\\
                &\lesa \limsup_{J\to J^*} \limsup_{n\to \infty}\Big( \sum_{\substack{1\les j, k \les J \\ j \not = k}} \big\| |\nabla| \mc{J}_0\big[ \Re\big( \overline{u}^{(j)}_n u^{(k)}_n \big)\big]\big\|_{W^0} +  \big\| |\nabla| \mc{J}_0\big[ \Re\big( \overline{\Psi}^{(J)}_n e^{it\Delta} w^{(J)}_n \big)\big]\big\|_{W^0} \\
                &\qquad \qquad \qquad \qquad \big\| |\nabla| \mc{J}_0\big[ | e^{it\Delta} w^{(J)}_n|^2 \big]\big\|_{W^0}\Big)\\
                &\lesa_{s, L, E_c, M_c}  \limsup_{J\to J^*} \limsup_{n\to \infty}\Big( \sum_{\substack{1\les j, k \les J \\ j \not = k}} \big\| |\nabla| \mc{J}_0\big[ \Re\big( \overline{u}^{(j)}_n u^{(k)}_n \big)\big]\big\|_{W^0} +  \| e^{it\nabla} w^{(J)}_n \|_Y^\theta\Big) = 0.
    \end{align*}
Consequently (iii) also holds.
\end{proof}

\section{Almost Periodic Solutions}\label{sec:almost per}

In this section we give the construction of the critical elements (or almost periodic solutions) in Theorem \ref{thm:min coun}. The first step is to show that if Conjecture \ref{conj:scattering} failed, then there must exist a mass/energy threshold with energy below the ground state.

\begin{lemma}[Existence of a mass/energy threshold]\label{lem:existence of threshold}
Suppose Conjecture \ref{conj:scattering} failed. Then there exists a mass/energy threshold $(M_c, E_c)$ with $4E_c<\|Q\|_{\dot{H}^1}^2$.
\end{lemma}
\begin{proof}
Given $M>0$ we let
        $$ E^*(M) = \sup\{ E>0 \mid L(M, E)<\infty\} $$
where we recall that $L(M, E)$ is defined in \eqref{eqn:L(M, E) defn}. Note that if $4E<\|Q\|_{\dot{H}^1}^2$ and $(u, V)\in \Omega(E)$ (see \eqref{eqn:defn Omega}) then Lemma \ref{lem:ground state} gives the $\dot{H}^1$ bound
    $$\|u\|_{L^\infty_t \dot{H}^1_x}^2 \les 4\mc{E}_Z(u, V)\les 4E.$$
In particular, Theorem \ref{thm:small data Hs} together with \eqref{eqn:Y bounded by sobolev intro} implies that for fixed $M>0$ and all sufficiently small $E>0$, we have $L(M, E)<\infty$ and hence $E^*(M)>0$. Moreover,  by construction we have the implication
        \begin{equation}\label{eqn:threshold:imp}
            E<E^*(M) \qquad \Longrightarrow \qquad L\big(M,   E\big) < \infty.
        \end{equation}
If Conjecture \eqref{conj:scattering} failed, then there must exist some $M_0>0$ such that $4E^*(M_0)<\|Q\|_{\dot{H}^1}$. We now define
        $$E_c = E^*(M_0), \qquad M_c = \inf\{ M>0 \mid E^*(M) = E^*(M_0)\}. $$
In view of \eqref{eqn:Y bounded by sobolev intro} and Theorem \ref{thm:small data Hs}, we have global well-posedness and scattering whenever the initial data satisfies
    $$\min\{\|f\|_{L^2}, \| f \|_{\dot{H}^1}\}\ll_{\|f\|_{H^1}} 1$$
and hence $E_c, M_c>0$. Moreover, as $L(M, E)$ is increasing in both $E$ and $M$, the definition of $E^*$ together with \eqref{eqn:threshold:imp} implies that if $E<E_c = E^*(M_0)$ then $L(M_c, E) \les L(M_0, E) < \infty$. On the other hand, if $M<M_c$, then by construction $E^*(M)$ is decreasing in $M$ and hence $E^*(M)>E^*(M_0)=E_c$ and so another application of the implication \eqref{eqn:threshold:imp} gives $L(M, E_c)<\infty$.

Therefore, to show that $(M_c, E_c)$ is a mass/energy threshold, it only remains to prove that $L(M_c, E_c)=\infty$. As usual, this is consequence of the (right) continuity of $L(M, E)$. More precisely, we claim that for any $0<4E<\|Q\|_{\dot{H}^1}^2$ and $M>0$, there exists $C = C(M, E)>0$, $\theta = \theta(M, E)>0$, and $\epsilon_0 = \epsilon_0(M, E)>0$ such that for any $0<\epsilon<\epsilon_0$ we have
        \begin{equation}\label{eqn:threshold:continuity}
             L(M, E) \les L(M+\epsilon, E+\epsilon) \les L(M, E) + C \epsilon^\theta.
        \end{equation}
Clearly \eqref{eqn:threshold:continuity} implies that $L(M_c, E_c) = \infty$. The right continuity of $L$ is a consequence of the stability result in Theorem \ref{thm:stab II} together with the observation in \cite{Guo2021} that $\lambda \mapsto \mc{E}_Z(\lambda f, \lambda^2 g)$ is an increasing function under the assumption that $(f,g)$ lie below the ground state $Q$. Roughly the point is that if we are at energy $E+\epsilon$, then flowing back in $\lambda$ reduces the energy at which point we can apply the definition of $L(M, E)$. Provided $\epsilon>0$ is sufficiently small, applying Theorem \ref{thm:stab II} then bounds $L(M+\epsilon, E+\epsilon)$ in terms of $L(M, E)$. Making this argument precise relies on the variational properties of the ground state contained in Lemma \ref{lem:ground state}.

We now turn to the details. The first inequality in \eqref{eqn:threshold:continuity} is immediate from the definition. To prove the second, it is enough to consider the case $L(M, E)<\infty$. Let $(u,V)\in \Omega(E+\epsilon)$ with $\mc{M}(u)\les M +\epsilon$. Our goal is to prove that $\|u\|_D - L(M, E) \lesa_{E, M} \epsilon^\theta$ with the implied constant only depending on $E$ and $M$. Define
        $$(f,g) = (\lambda u, \lambda^2 V)(0)\in H^1\times L^2$$
where $0<\lambda<1$ is to be chosen later. Note that
    $$ \mc{M}(u) - \mc{M}(f) = (1-\lambda^2) \mc{M}(u) \gtrsim_M (1-\lambda). $$
On the other hand, for any $0<\epsilon \les \frac{1}{8}( \|Q\|^2_{\dot{H}^1} - E)$, as $4E < \|Q\|_{\dot{H}^1}^2$ Lemma \ref{lem:ground state} gives for any $0< \lambda \les 1$ the lower bound
 \begin{align*}
    \mc{K}\big(\lambda u(0) \big) + \lambda^4 \big\| V(0) + |u(0)|^2\big\|_{L^2_x}^2 &\g \lambda^4 \Big( \mc{K}\big( u(0) \big) +  \big\| V(0) + |u(0)|^2\big\|_{L^2_x}^2  \Big) \\
        &\g \lambda^4 4\mc{E}_Z(u, V) \frac{\|Q\|_{\dot{H}^1}^2  - 4\mc{E}_Z(u, V)}{2\|Q\|_{\dot{H}^1}^2} \gtrsim_E\lambda^4 \mc{E}_Z(u, V).
 \end{align*}
Consequently the identity \eqref{eqn:scaling der of zak energy} implies that
    \begin{align*}
      \mc{E}_Z(u,V) - \mc{E}_Z(f,g) =   \int_\lambda^1 a^{-1} \mc{K}\big(a u(0) \big) + a^3 \big\| V(0) + |u(0)|^2\big\|_{L^2_x}^2 da  \gtrsim_E (1-\lambda) \mc{E}_Z(u, V).
    \end{align*}
In particular, the energy and mass decrease as $\lambda \to 0$, and hence choosing $(1-\lambda) \approx_E \epsilon$ we have $\mc{E}_Z(f,g) \les E$ and $\mc{M}(f)\les M$. Therefore, letting $(\psi, \phi) \in \Omega(E)$ denote the corresponding global solution to \eqref{eqn:Zak} with data $(\psi, \phi)(0) = (f, g)$ given by Theorem \ref{thm:gwp below ground state}, we see that $\| \psi\|_D \les L(M, E)$. To conclude the corresponding dispersive bound for $u$ we apply the stability result from Theorem \ref{thm:stab II}. More precisely, Theorem \ref{thm:persistence} and Lemma \ref{lem:ground state} implies that
            $$ \|u(0)\|_{H^1} +  \|\psi \|_{S^\frac{1}{2}} \lesa_{E,M} 1, \qquad \|\phi\|_{L^\infty_t L^2_x} \les 2E^\frac{1}{2} < \|Q\|_{\dot{H}^1}$$
(strictly speaking the implied constant here also depends on $L(M, E)$). On the other hand, the choice of $(f,g)$ gives
        $$ \| u(0) - \psi(0) \|_{\dot{H}^1} + \| V(0) - \phi(0) \|_{L^2} \lesa_{E,M} 1-\lambda \lesa_{E,M} \epsilon. $$
Hence provided $\epsilon>0$ is sufficiently small (depending only on $M$ and $E$, albeit via $L(M, E)$), \eqref{eqn:mod stab error cond} together with Theorem \ref{thm:stab II} implies that $u \in S^\frac{1}{2}$ with the bound
        $$  \| u - \psi\|_D \lesa  \| u - \psi - e^{it\Delta}(u-\psi)(0)\|_{S^\frac{1}{2}} + \|(u-\psi)(0)\|_{H^1}  \lesa_{E,M} \epsilon^\theta. $$
In other words, we have a constant $C=C(M, E)>0$ depending only on $E$ and $M$ (and $L(M, E)$) such that $\|u\|_D \les L(M, E) + C \epsilon^\theta$. Taking the sup over all $(u, V)\in \Omega(E+\epsilon)$ with $\mc{M}(u)=M$ we conclude the required bound.
\end{proof}

The proof of Theorem \ref{thm:min coun} is now an application of the Palais-Smale type property together with the stability result obtained earlier.

\begin{proof}[Proof of Theorem \ref{thm:min coun}]
Suppose Conjecture \ref{conj:scattering} failed. Applying Lemma \ref{lem:existence of threshold} we would then conclude that there exists a  mass/energy threshold $(M_c, E_c)$ with $4E_c<\|Q\|_{\dot{H}^1}^2$. In particular, as $L(M_c, E_c) = \infty$, there exists a sequence $(u_n, V_n) \in \Omega(E_c)$ such that
	$$ \lim_{n \to \infty} \mc{E}_Z(u_n, V_n) = E_c, \qquad \lim_{n\to \infty} \mc{M}(u_n) = M_c, \qquad \lim_{n\to \infty} \|u_n\|_D = \infty. $$
Choose $t_n\in \RR$ such that
		$$ \lim_{n\to \infty} \|u_n\|_{D((-\infty, t_n])} = \lim_{n\to \infty} \| u_n \|_{D([t_n, \infty))} = \infty. $$
After replacing $u_n(t)$ with $u_n(t+t_n)$, we may assume that $t_n=0$ for all $n\in \NN$. Theorem \ref{thm:Palais-Smale} then implies that, up to a subsequence, we have $(f_c, g_c)\in H^1\times L^2$ and $x_n\in \RR^4$ such that
		\begin{equation}\label{eqn:min coun:data conv}
			\lim_{n\to \infty} \big\| (u_n, V_n)(0, x+x_n) - (f_c, g_c)(x) \big\|_{H^1\times L^2} = 0.
		\end{equation}
Note that $\mc{E}_Z(f_c, g_c) = E_c$,  $\mc{M}(f_c)= M_c$, and $\|g_c\|_{L^2}\les \|Q\|_{\dot{H}^1}$. Hence applying Theorem \ref{thm:gwp below ground state} with data $(f_c, g_c)\in H^1 \times L^2$ we obtain a global solution $(\psi, \phi) \in C(\RR; H^1\times L^2)$ to \eqref{eqn:Zak} with
		\begin{equation}\label{eqn:min coun:min prop}
 \psi \in L^2_{t, loc} W^{\frac{1}{2}, 4}_x, \qquad \mc{E}_Z(\psi, \phi) = E_c, \qquad \mc{M}(\psi) = M_c, \qquad \|\phi\|_{L^\infty_t L^2_x} \les \|Q\|_{\dot{H}^1}.
		\end{equation}
Moreover, the limit \eqref{eqn:min coun:data conv} together with the stability result in Theorem \ref{thm:stab II} gives
		\begin{equation}\label{eqn:min coun:min nondisp}
			\| \psi \|_{D((-\infty, 0])}=\|\psi \|_{D([0, \infty))} = \infty.
		\end{equation}
It only remains to verify that there exists $x(t): \RR \to \RR^4$ such that the orbit
	\begin{equation}\label{eqn:min coun:orbit}
		\big\{ (\psi, \phi)\big(t, x + x(t)\big)  \,\, \big| \,\,  t\in \RR \big\}
	\end{equation}
 is precompact in $H^1\times L^2$. To prove the existence of the translations $x(t)$, one option is to argue abstractly as in \cite{Tao2008c}. Alternatively, and this is the approach we take here, we can give a concrete definition\footnote{This observation was kindly communicated to us by Kenji Nakanishi.} of the translations by noting that $x(t)$ should essentially be the ``centre of mass'' of the solution $(\psi, \phi)(t)$. To this end, we choose the components $x_j(t) \in \RR$ of $x(t)\in \RR^4$ as
	$$ \int_{x_j(t)}^\infty \int_{\RR^3} \big( |\nabla \psi|^2 + |\psi|^2 + |\phi|^2\big)(t,y) dy' dy_j  = \frac{1}{2} \| (\psi, \phi)(t) \|_{H^1\times L^2}^2  $$
where $y'\in \RR^3$ denotes the remaining spatial variables. In other words, $x(t)$ is roughly the centre of the $H^1\times L^2$ mass of $(\psi, \phi)$. Suppose \eqref{eqn:min coun:orbit} is not precompact. Then there exists sequence $t_n \in \RR$ and $C>0$ such that for all $n\not = m$ we have
		\begin{equation}\label{eqn:min coun:contra}
			\big\| (\psi, \phi)\big(t_n, x+x_n\big) - (\psi, \phi)\big(t_m, x+x_m\big) \big\|_{H^1\times L^2} \g C
		\end{equation}
where for ease of notation we let $x_n=x(t_n)$. Applying Theorem \ref{thm:Palais-Smale} to the sequence $(\psi, \phi)(t+t_n)$, the properties \eqref{eqn:min coun:min prop} and \eqref{eqn:min coun:min nondisp} imply that there exists $\tilde{x}_n\in \RR^4$ and $(\tilde{f},\tilde{g})\in H^1\times L^2$ such that up to a subsequence,  the translated sequence $(\psi, \phi)(t_n, x+ \tilde{x}_n)$ converges to $(\tilde{f}, \tilde{g})\in H^1\times L^2$. In particular, after translating once more, we have
		\begin{equation}\label{eqn:min coun:conv}
			\lim_{n\to \infty}\big\| (\psi, \phi)(t_n, x+x_n) - (\tilde{f}, \tilde{g})(x + x_n - \tilde{x}_n) \big\|_{H^1\times L^2} = 0.
		\end{equation}
If $\sup_n |x_n - \tilde{x}_n|< \infty$, then after taking a further subsequence we can assume $x_n - \tilde{x}_n$ converges. But then \eqref{eqn:min coun:conv} is clearly a contradiction to \eqref{eqn:min coun:contra}. Similarly, if $(\tilde{f}, \tilde{g})=0$, then again \eqref{eqn:min coun:conv} contradicts \eqref{eqn:min coun:contra}. Thus, we may assume that $(\tilde{f}, \tilde{g})\not = 0$, and writing
the components of the vectors $x_n, \tilde{x}_n\in \RR^4$ as $x_{n, j}$ and $\tilde{x}_{n, j}$, there must exist some $1\les j \les 4$ such that $|x_{n, j} - \tilde{x}_{n, j}| \to \infty$. Observe that, by our choice of $x(t)$ and \eqref{eqn:min coun:conv}, we have
	\begin{align*}
		\lim_{n\to\infty} \int_{x_{n,j} - \tilde{x}_{n,j}}^\infty \int_{\RR^3}\big(|\nabla \tilde{f}|^2 + |\tilde{f}|^2 + |\tilde{g}|^2\big)(y) \,dy' dy_j &= \lim_{n\to \infty} \int_{x_{n,j}}^\infty \int_{\RR^3} \big( |\nabla\psi|^2 + |\psi|^2 + |\phi|^2\big)(t_n, y) dy' dy_j \\
		&= \frac{1}{2} \lim_{n\to \infty} \big\| (\psi, \phi)(t_n)\big\|_{H^1\times L^2}^2 = \frac{1}{2}\big\|(\tilde{f}, \tilde{g})\big\|_{H^1\times L^2}^2.	\end{align*}
But this is again a contradiction, as the left hand side converges to either $0$ or $\|(\tilde{f}, \tilde{g})\|_{H^1\times L^2}^2 \not = 0$. Therefore \eqref{eqn:min coun:contra} cannot hold, and hence the orbit \eqref{eqn:min coun:orbit} is precompact as claimed.
\end{proof}

\section*{Acknowledgements}
The author would like to thank Kenji Nakanishi for many helpful conversations on the concentration compactness argument for dispersive PDE. Financial support from the Marsden Fund Council grant 19-UOO-142, managed by Royal Society Te Ap\={a}rangi is gratefully acknowledged.

\bibliographystyle{amsplain}
\bibliography{non-scat_Zak}
\end{document}